\pdfoutput=1

\documentclass[letterpaper, 11pt]{article}
\usepackage{amssymb,amsmath,amsthm}
\usepackage{epsf}
\usepackage{times,bm,mathrsfs}
\usepackage{amsmath}
\usepackage{amssymb}
\usepackage{amsfonts}
\usepackage{mathrsfs}
\usepackage{bbm}
\usepackage{color}
\usepackage{graphicx}
\usepackage{amscd}
\usepackage{epsfig}

\definecolor{Red}{rgb}{1,0,0}
\definecolor{Blue}{rgb}{0,0,1}
\definecolor{Olive}{rgb}{0.41,0.55,0.13}
\definecolor{Green}{rgb}{0,1,0}
\definecolor{MGreen}{rgb}{0,0.8,0}
\definecolor{DGreen}{rgb}{0,0.55,0}
\definecolor{Yellow}{rgb}{1,1,0}
\definecolor{Cyan}{rgb}{0,1,1}
\definecolor{Magenta}{rgb}{1,0,1}
\definecolor{Orange}{rgb}{1,.5,0}
\definecolor{Violet}{rgb}{.5,0,.5}
\definecolor{Purple}{rgb}{.75,0,.25}
\definecolor{Brown}{rgb}{.75,.5,.25}
\definecolor{Grey}{rgb}{.5,.5,.5}
\definecolor{Black}{rgb}{0,0,0}

\newcommand{\lscr}{\mathscr{L}}

\newcommand{\ccal}{\mathcal{C}}
\newcommand{\dcal}{\mathcal{D}}
\newcommand{\ecal}{\mathcal{E}}
\newcommand{\fcal}{\mathcal{F}}
\newcommand{\gcal}{\mathcal{G}}
\newcommand{\hcal}{\mathcal{H}}
\newcommand{\ical}{\mathcal{I}}
\newcommand{\jcal}{\mathcal{J}}

\newcommand{\mcal}{\mathcal{M}}

\newcommand{\ocal}{\mathcal{O}}

\newcommand{\rcal}{\mathcal{R}}
\newcommand{\scal}{\mathcal{S}}
\newcommand{\tcal}{\mathcal{T}}

\newcommand{\vcal}{\mathcal{V}}
\newcommand{\wcal}{\mathcal{W}}

\newcommand{\zcal}{\mathcal{Z}}

\newcommand{\real}{\mathbb{R}}

\newcommand{\nintgr}{\mathbb{N}}

\newcommand{\eps}{\varepsilon}

\newcommand{\ind}{\mathbf{1}}

\newcommand{\prob}{\mathbb{P}}
\newcommand{\expec}{\mathbb{E}}
\newcommand{\E}{\mathbb{E}}
\renewcommand{\P}{\mathbb{P}}

\newcommand{\poly}{\mathrm{poly}}

\newtheorem{theorem}{Theorem}
\newtheorem{proposition}{Proposition}

\newtheorem{definition}{Definition}[section]

\newtheorem{lemma}{Lemma} 

\newtheorem{claim}{Claim}[section]

\newcommand{\weight}{w}

\newcommand{\dist}{\mathrm{d}}
\newcommand{\gdist}{\dist^{\mathrm{g}}}

\newcommand{\path}{\mathrm{P}}

\newcommand{\phy}{\mathbb{Y}}
\newcommand{\hmgphy}{\mathbb{HY}}

\newcommand{\tv}{\Delta_{\mathrm{TV}}}
\newcommand{\swap}{\Delta_{\mathrm{SW}}}
\newcommand{\blowup}{\Delta_{\mathrm{BL}}}

\newcommand{\kl}{\mathrm{KL}}

\newcommand{\srevision}[1]{{#1}}

\newcommand{\rt}{\rho}
\newcommand{\hmgt}[1]{T^{(#1)}}
\newcommand{\hmgv}[1]{V^{(#1)}}
\newcommand{\hmgl}[1]{L^{(#1)}}
\newcommand{\hmgll}[2]{L^{(#1)}_{#2}}
\newcommand{\hmge}[1]{E^{(#1)}}

\newcommand{\hmgrt}[1]{\rt^{(#1)}}

\newcommand{\trees}{\mathbb{T}}

\newcommand{\s}{\sigma}

\newcommand{\hs}{\hat\s}

\newcommand{\evr}{\nu}

\renewcommand{\b}{\text{\tt B}}
\newcommand{\g}{\text{\tt G}}
\renewcommand{\r}{\text{\tt R}}

\newcommand{\y}{\text{\tt Y}}

\newcommand{\quantum}{\frac{1}{\Upsilon}}
\newcommand{\invquantum}{\Upsilon}

\renewcommand{\L}{\lscr}

\newcommand{\hattml}{\hat{T}^{\mathrm{ML}}_k}
\newcommand{\psiml}{\Psi^{\mathrm{ML}}}

\DeclareMathOperator*{\argmin}{arg\,min}

\newcommand{\hdcal}{\widehat{\dcal}}
\newcommand{\tdcal}{\widetilde{\dcal}}

\addtocounter{page}{-1}

\begin{document}

\title{\vspace{0cm}
Phase transition in the sample complexity of
likelihood-based phylogeny inference
}

\author{
Sebastien Roch\footnote{Department of Mathematics,
UW--Madison.
Work partly done at Microsoft Research, UCLA, IPAM and
the Simons Institute for the Theory of Computing.
Work supported by NSF grants DMS-1007144 and DMS-1149312 (CAREER), and an Alfred P. Sloan Research Fellowship.}
\and
Allan Sly\footnote{
Department of Mathematics, Princeton University. Work partly done at Microsoft Research. Work supported by NSF grants DMS-1208339 and DMS-1352013 and an Alfred P. Sloan Research Fellowship.}
}
\maketitle

\begin{abstract}
Reconstructing evolutionary trees
from molecular sequence data is a fundamental
problem in computational biology.
Stochastic models of sequence evolution
are closely related to 
spin systems that have been
extensively studied in statistical physics 
and
that connection has led to 
important insights on the
theoretical properties of 
phylogenetic reconstruction
algorithms as well as the development
of new inference methods.
Here, we study maximum likelihood,
a classical statistical technique which is
perhaps the most widely
used in phylogenetic practice
because of its superior empirical accuracy.

At the theoretical level,
except for its consistency,
that is, the guarantee of eventual
correct reconstruction
as the size of the input data grows,
much remains to be understood
about the statistical properties of 
maximum likelihood
in this context. In particular, 
the best bounds on the sample
complexity or
sequence-length
requirement of maximum likelihood,
that is, the amount of data required for
correct reconstruction,
are exponential in the number, $n$, of tips---far
from known lower bounds
based on information-theoretic arguments.
Here we close the gap by proving
a new upper bound on the
sequence-length requirement of maximum likelihood
that matches up to constants the known
lower bound for some
standard models of evolution.

More specifically, for the $r$-state symmetric
model of sequence evolution on a binary phylogeny
with bounded edge lengths,
we show that the sequence-length requirement
behaves logarithmically in $n$ when
the expected amount of mutation per edge
is below what is known as the Kesten-Stigum threshold.
In general, the sequence-length requirement 
is polynomial in $n$.
Our results imply moreover that the 
maximum likelihood estimator
can be computed efficiently on randomly
generated data provided
sequences are as above.

Our main technical contribution,
which may be of independent interest, 
relates the total
variation distance between the leaf state distributions
of two trees with a notion of combinatorial
distance between the trees. In words we show
in a precise quantitative manner that the more
different two evolutionary trees are, the easier it is to distinguish
their output.
\end{abstract}

\thispagestyle{empty}

\clearpage

\section{Introduction}\label{section:introduction}

\paragraph{Background}
Reconstructing evolutionary trees, or phylogenies,
from biomolecular data is a fundamental
problem in computational biology~\cite{SempleSteel:03,Felsenstein:04,DeTaWa:05,Steel:16,Warnow:u}.
Roughly, 
in the basic form of the problem, 
sequences from a common gene (or other
DNA region) are collected 
from representative individuals 
of contemporary species
of interest. From that sequence data
(which is usually aligned to account
for insertions and deletions), 
a phylogeny
depicting the shared history
of the species is inferred. 

From a formal statistical point of view,
one typically assumes that 
each site in the (aligned) data 
has evolved independently according to a common
Markov model of substitution along the
tree of life. The problem then boils down
to reconstructing this generating model from 
i.i.d.~samples at the leaves of the tree.
Such models are closely related to 
spin systems that have been
extensively studied in statistical physics~\cite{Liggett:85,Georgii:88} and
that connection has led to 
new insights on the
amount of data required for accurately 
reconstructing
phylogenies~\cite{Steel:01}. 
More specifically, under
broad modeling assumptions,
algorithmic upper bounds
have been obtained on the sample complexity of the
phylogenetic reconstruction problem,
together with matching information-theoretic
(i.e., applying to any method) lower bounds
 ~\cite{Mossel:03,Mossel:04a,Roch:10,DaMoRo:11a,MoRoSl:11,MiHiRa:13}.
In particular it was established 
that the best achievable sample complexity
undergoes a phase transition as the maximum 
branch length
varies. That phase transition is closely related
to the well-studied problem of 
reconstructing the root sequence
of a Markov model on a tree
given the leaf sequences~\cite{Mossel:04},
a tool which plays a key role in the above results.

The algorithmic results 
in~\cite{Roch:10,DaMoRo:11a,brown2012fast,MiHiRa:13}
concern ad hoc methods of inference.
On the other hand, little is known
about the precise
sample complexity of reconstruction methods
used by evolutionary biologists in practice
(with some exceptions~\cite{LaceyChang:06}).
Here we consider maximum likelihood (ML), introduced in phylogenetics in~\cite{Felsenstein:81},
where one computes (or approximates) the tree most likely 
to have produced the data among a class of allowed models.
Likelihood-based methods are perhaps
the most widely used and most
trusted methods
in current phylogenetic practice~\cite{Stamatakis:06}. 
 In previous theoretical work,
upper bounds were derived on the sample
complexity of ML that were far from the
lower bound~\cite{SteelSzekely:02}---in some
regimes, doubly exponentially far in the
number of species.
\emph{Here we close the gap by proving
	a new upper bound on the
	sample complexity of maximum likelihood
	that matches up to constants the known
	information-theoretic lower bound for some
	standard models of evolution.}

\paragraph{Overview of main results and techniques}
In order to state our main results more precisely,
we briefly describe the model of evolution
considered here. (See Section~\ref{section:preliminaries}
for more details.) The {\em unknown} phylogeny
$T$ is a
weighted binary tree with $n$ leaves labeled
by species names, one leaf for each species
of interest. 
Without loss of generality, we assume that
the leaf labels are $[n] = \{1,\ldots,n\}$.
The weights on the edges (or branches), $\{w_e\}_{e \in E}$ where $E$ is the set of edges of $T$,
are assumed to be discretized and bounded
between two constants $f < g$. The quantity
$w_e$ can be interpreted
as the expected number of mutations per site
along edge $e$. We denote by $\phy^{(n)}_{f,g}[\quantum]$ the set of all such phylogenies, where $\quantum$ is the 
discretization. 

Let $\rho$ be the root of the tree,
that is, the most recent common ancestor to the
species at the leaves (which formally can be chosen arbitrarily
as it turns out not to affect the distribution
of the data). Let $\rcal$ be a state space
of size $r$. A typical choice is $\rcal = \{\mathrm{A}, \mathrm{G},\mathrm{C},\mathrm{T}\}$ and $r=4$, but we consider
more general spaces as well. 
Define
$$
\delta_e = \frac{1}{r}\left(1 - e^{-\weight_e}\right).
$$
In the {\em $r$-state symmetric model},
we start at $\rho$ with a
sequence of length $k$ chosen uniformly in $\rcal^k$. 
Moving away from the root, each vertex $v$ in $T$
is assigned the sequence $(s_u^i)_{i=1}^k$ of its parent $u$ randomly ``mutated'' as follows:
letting $e$ be the edge from $u$ to $v$,
for each $i$, with probability $(r-1) \delta_e$
set $s_v^i$ to a uniform state in $\rcal-\{s_u^i\}$
(corresponding to a substitution), or otherwise
set $s_v^i = s_u^i$. Let $(s^i_{[n]})_{i=1}^k \in (\rcal^{[n]})^k$ be the sequences at the {\em leaves}.
Those are the sequences that are {\em observed}. We let
$\mu_{[n]}^T(s^i_{[n]})$ be the probability
of observing $s^i_{[n]}$ under $T$.

In the phylogenetic reconstruction problem,
we are given sequences $(s^i_{[n]})_{i=1}^k$,
assumed to have been generated under the
$r$-state symmetric model on an unknown
phylogeny $T$, and our goal is to recover
$T$ (without the root) as a {\em leaf-labeled
	tree} (that is, we care about the
locations of the species on the tree).
This problem is known to be well-defined
in the sense that, under our assumptions,
the phylogeny is uniquely
identifiable from the distribution
of the data at the leaves~\cite{Chang:96}.
A useful proxy to assess the ``accuracy''
of a reconstruction method is its sample
complexity or {\em sequence-length requirement}, roughly,
the smallest sequence length $k$ (as a function of $n$) such that
a perfect reconstruction is guaranteed
with probability approaching $1$ as 
$n$ goes to $+\infty$. (See Section~\ref{section:preliminaries}
for a more precise definition.)
A smaller
sequence-length requirement is an indication
of superior statistical performance.
We denote by $k_0(\Psi,n)$ the
sequence-length requirement of method $\Psi$.

Here we analyze the sequence-length requirement
of ML which, in our context,
we define as
\begin{equation*} 
\psiml_{n,k}((s^i_{[n]})_{i=1}^k)
\in \argmin_{T \in \phy^{(n)}_{f,g}[\quantum]}
\L_T[(s^i_{[n]})_{i=1}^k],
\end{equation*}
where
$
\L_T[(s^i_{[n]})_{i=1}^k]
= - \sum_{i=1}^k
\ln \mu^T_{[n]}(s^i_{[n]})
$
is the {\em log-likelihood}
(breaking ties arbitrarily).
In words ML selects a phylogeny
that maximizes the probability of observing
the data. This method is known to be
consistent, that is, the reconstructed phylogeny
is guaranteed to converge on the true tree
as $k$ goes to $+\infty$.
The previous best known bound on the sequence-length
requirement of ML in this context
is $k_0(\psiml,n) \leq \exp(K n)$,
for a constant $K$,
as proved in~\cite{SteelSzekely:02}.

Our first main result is that,
for any constant $f,g, \quantum$, 
the ML sequence-length requirement 
$k_0(\psiml, n)$
grows at most {\em polynomially} in $n$,
where the degree of the polynomial
depends on $g$.
Such a bound had been previously established
for other reconstruction methods, including
certain types of distance-matrix
methods~\cite{ErStSzWa:99a}, and 
it was a long-standing open problem
to show that a polynomial bound holds for ML as well. 
Interestingly,
our simple proof in fact uses the result of~\cite{ErStSzWa:99a}.
(The argument is detailed in Section~\ref{section:preliminaries}.)
Further, it is known that no method
in general achieves a better bound
(up to a constant in the degree of the
polynomial)~\cite{Mossel:03}.

On the other hand, 
our second---significantly more challenging---result establishes a
phase transition on the sequence-length
requirement of ML.
{\em We show that
when the maximum branch length of the true phylogeny 
is constrained to lie below 
a given threshold, an improved sequence-length
requirement is achieved, namely
that }
\begin{equation}
\label{eq:main-result-intro}
k_0(\psiml, n)
= 
O(\log n),\quad
\text{if $g < g^* := \ln\sqrt{2}$.}
\end{equation}
The same sequence-length requirement has been
obtained previously for other methods~\cite{Mossel:04a,DaMoRo:11a,Roch:10,brown2012fast,MiHiRa:13},
but as we mentioned above our result is the first one
that concerns an important method in practice
and greatly improves previous bound for ML.
It is known further that a sub-logarithmic sequence-length requirement is not possible in general for any method~\cite{Mossel:03}.
That can be seen by the
following back-of-the-envelope calculation: when $k = \Theta(\log n)$,
the total number of datasets is $e^{\Theta(n\log n)}$,
which is asymptotically of the same order as the number of phylogenies in $\phy^{(n)}_{f,g}[\quantum]$
(see, e.g.,~\cite{SempleSteel:03}); and, intuitively, we need at least as many datasets as we have possible phylogenies. 
Note that not all methods achieve 
the logarithmic sequence-length requirement
in~\eqref{eq:main-result-intro}.
The popular distance-matrix method 
Neighbor-Joining~\cite{SaitouNei:87}, 
for instance,
has been shown to require exponential
sequence lengths in general for any $g$~\cite{LaceyChang:06}. 

The question of whether the threshold 
in~\eqref{eq:main-result-intro}
is tight,
however, 
is not completely resolved and we do not address
this issue here. 
The quantity
$g^*$ corresponds to what is sometimes known as the
\emph{Kesten-Stigum threshold}~\cite{KestenStigum:66},
which is roughly speaking the threshold at which
reconstructing the root state from a ``weighted majority''
of the leaf states becomes 
no better than guessing at random as the depth of the (binary) tree diverges. See, e.g.,~\cite{EvKePeSc:00,Mossel:04}
for some background on this problem. See also~\cite{JansonMossel:04} for a different 
characterization of the threshold.
For $r=2$, no root state inference method has a better
threshold than weighted majority~\cite{Ioffe:96a}
and the bound in~\eqref{eq:main-result-intro} 
is known to be tight~\cite{Mossel:04a}.
In general, the question is not settled~\cite{Mossel:01,Sly:09,MoRoSl:11}.
In the case $r=4$, the most relevant
in the biological context, the threshold
$g^*$ translates into a $22\%$ substitution
probability along each edge. 
In general, many factors
affect the maximum branch length of a phylogeny, including how
densely sampled the species are and which
genes (whose mutation rates vary widely) are used.

To understand
the connection between root state reconstruction
and phylogenetic reconstruction, a connection which
was first articulated by Steel~\cite{Steel:01},
note that the depth of the phylogeny plays
a key role in phylogenetic reconstruction.
That is because we only have access to the sequences
at the {\em leaves} of the tree. 
When good estimates of internal
sequences are available, the phylogeny is
``shallower'' and reconstructing the deeper parts
of the tree is significantly easier, leading
to a better sequence-length requirement for some methods.
That the phase transition in root state reconstruction
should translate into a phase transition in the sequence-length requirement of phylogenetic reconstruction, namely
from logarithmic in $n$ in the ``reconstruction phase''
to polynomial in $n$ in the ``non-reconstruction phase,''
is known as {\em Steel's Conjecture}~\cite{Steel:01}.
It was first established rigorously by Mossel~\cite{Mossel:04a} in the case of balanced
binary trees with $r=2$.

In~\cite{Mossel:04a,DaMoRo:11a,Roch:10,brown2012fast,MiHiRa:13},
in order to achieve logarithmic sequence-length
requirement in the Kesten-Stigum
regime, {\em new inference methods} that explicitly
estimate internal sequences were devised. 
In ML, by contrast, the internal sequences play
a more implicit role in the definition
of the likelihood and our analysis
of ML proceeds in a very different manner. 
{\em Our main technical contribution, which
may be of independent interest, is a
quantitative bound on the total
variation distance between the leaf state distributions
of two phylogenies
as a function of a notion of combinatorial
distance between them.} In words,
the more different are the trees, the more different
are the data distributions at their leaves.
We prove this new bound by constructing explicit
tests that distinguish between the leaf distributions.
For general trees, this turns out to present
serious difficulties, as sketched
in Section~\ref{section:preliminaries}. 
The bound on the total
variation distance in turn gives a bound
on the probability that ML returns an incorrect
tree and allows us to perform a union bound
over all such trees.

It is worth pointing out that the reconstruction
methods
of~\cite{Mossel:04a,DaMoRo:11a,Roch:10,brown2012fast,MiHiRa:13}
have the advantage of running in polynomial
time, while computing the ML phylogeny
is in the worst-case NP-hard~\cite{Roch:06,ChorTuller:06}.
So why care about ML? 
Of course, worst-case computational complexity results
are not necessarily relevant in practice
as real data tend to be more structured.
Actually, good heuristics for ML have been developed
that have achieved considerable practical success
in large-scale phylogenetic analyses and are now seen as the  standard approach~\cite{Stamatakis:06,smith2014inferring}.
A side consequence of our results is that, on
randomly generated data of sufficient sequence
length, using the methods
of~\cite{Mossel:04a,DaMoRo:11a,Roch:10,brown2012fast,MiHiRa:13}
we are in fact guaranteed to recover what happens
to be the ML
phylogeny with high probability in polynomial
time. Although this is not per se an algorithmic
result in that we do not directly
solve the ML problem, it does show that computing the ML phylogeny
is easier than previously thought in an average sense and may help
explain the success of practical heuristics.

Although the discretization assumption above
may not be needed, removing it in the logarithmic
regime appears to
present significant technical challenges.
Note that this assumption is also needed for the results of~\cite{DaMoRo:11a,Roch:10,brown2012fast,MiHiRa:13}.

\paragraph{Further related work}
There exists a large
literature
on the sequence-length requirement of 
phylogenetic reconstruction methods,
stemming mainly from the
seminal work of Erd\"os et al.~\cite{ErStSzWa:99a}
which were the first to highlight 
the key role of the depth
in inferring phylogenies. 
Sequence-length requirement results---both upper and lower bounds---have been derived for more general
models of sequence evolution~\cite{ErStSzWa:99b,Mossel:03,MosselRoch:06,BoChMoRo:06,choi2011learning}, including
models of insertions and deletions~\cite{AnDaHaRo:10,DaskalakisRoch:13},
for partial or forest reconstruction~\cite{CrGoGo:02,Mossel:07,DaMoRo:11b,tan2011learning,gronau2012fast},
and for reconstructing mixtures of phylogenies~\cite{MosselRoch:12,MosselRoch:13}.
These results have in some cases also inspired successful
practical heuristics~\cite{HuNeWa:99}.

The connection between root state reconstruction
and phylogenetic reconstruction has also been
studied in more general models of evolution
where mutation probabilities
are not necessarily symmetric~\cite{Roch:08,Roch:10,MoRoSl:11}.
A good starting point for the extensive
literature on root state reconstruction
is~\cite{Peres:99,Mossel:04}.

Some bounds on the total variation distance
between leaf distributions that are related to our
techniques were previously obtained in the
special case of pairs of random trees, which are essentially at maximum combinatorial distance~\cite{SteelSzekely:06}. Similar
ideas were also used to reconstruct certain mixtures
of phylogenies in~\cite{MosselRoch:12}.

The sample complexity of maximum likelihood 
when all internal vertices are also observed 
was studied in~\cite{tan2011large}.

\paragraph{Organization}
The paper is organized as follows.
Basic definitions are provided 
in Section~\ref{section:preliminaries}. 
In Section~\ref{section:preliminaries}
we also state formally our main results and
give a sketch of the proof. 
The probabilistic aspects of the proof 
are sketched  
in Section~\ref{sec:prelim}. 
The combinatorial aspects are illustrated first in
a special case in Section~\ref{sec:homo}.
The general case is detailed in Section~\ref{sec:general}.
A few useful lemmas can be found in the appendix
for ease of reference.

\section{Definitions, Results, and Proof Sketch}\label{section:preliminaries}

In this section, we introduce formal definitions
and state our main results. 

\subsection{Basic Definitions}
\label{sec:basic-definitions}

\paragraph{Phylogenies}
A phylogeny is a graphical representation
of the speciation history of a collection of organisms.
The leaves correspond to current species (i.e., those
that are still living).
Each branching indicates a speciation event.
Moreover we associate to each edge a
positive weight.
As we will see below, this weight 
corresponds roughly to the
amount of evolutionary change on the edge.
More formally, we make the following definitions.
See e.g.~\cite{SempleSteel:03} for more background.
Fix a set of leaf labels (or species names)
$X = [n] = \{1,\ldots,n\}$.
\begin{definition}[Phylogeny]\label{def:phylo}
A \emph{weighted binary phylogenetic $X$-tree}
(or \emph{phylogeny} for short)
$T = (V,E;\phi;\weight)$ is a tree
with vertex set $V$,
edge set $E$,
leaf set $L$ with $|L| = n$,
edge
weights $\weight : E \to (0,+\infty)$,
and a bijective leaf-labeling $\phi : X \to L$
(that assigns ``species names'' to the leaves).
We assume that the degree of all internal vertices
$V-L$ is exactly $3$.
We let $\tcal_l[T] = (V,E;\phi)$ be the
{\em leaf-labelled topology} of $T$.
We denote by $\trees_n$ the set of all
leaf-labeled trees on $n$ leaves with internal degrees $3$
and we let $\trees = \{\trees_n\}_{n\geq 1}$.
We say that two phylogenies are isomorphic if there
is a graph isomorphism between them that preserves
the edge weights and the leaf-labeling. 
\end{definition}
\noindent We restrict ourselves
to the following setting introduced
in~\cite{DaMoRo:11a}.
\begin{definition}[Regular phylogenies]
\label{def:regphy}
Let $0 < \quantum \leq f \leq g < +\infty$.
We denote by $\phy^{(n)}_{f,g}[\quantum]$
the set of phylogenies $T = (V,E; \phi;\weight)$
with $n$ leaves
such that $f \leq \weight_e \leq g$,
$\forall e\in E$, where moreover $\weight_e$
is a multiple of $\quantum$. We also
let $\phy_{f,g}[\quantum] = \bigcup_{n \geq 1}
\phy^{(n)}_{f,g}[\quantum]$. (We assume
for simplicity that $f$ and $g$ are themselves
multiples of $\quantum$.)
\end{definition}
\noindent To illustrate our techniques, 
we also occasionally appeal to the
special case of homogeneous phylogenies. 
For an integer $h \geq 0$ and $n = 2^h$,
a homogeneous phylogeny is 
an $h$-level
complete binary tree
$\hmgt{h}_{\phi,\weight} = (\hmgv{h}, \hmge{h}; \phi; \weight)$
where the edge weight function
$\weight$ is identically $g$
and $\phi$ may be any one-to-one labeling
of the leaves.

\paragraph{Substitution model}
We use the following standard model of DNA sequence evolution.
See e.g.~\cite{SempleSteel:03} for generalizations.
Fix some integer $r > 1$.
\begin{definition}[$r$-State Symmetric Model of Substitution]\label{def:mmt}
Let $T = (V,E;\phi;\weight)$ be a phylogeny
and $\rcal = [r]$.
Let $\pi = (1/r,\ldots, 1/r)$ be the uniform distribution
on $[r]$ and 
let
$
\delta_e = \frac{1}{r}\left(1 - e^{-\weight_e}\right).
$
Consider the following stochastic
process.
Choose an arbitrary root $\rho \in V$.
Denote by $E_\downarrow$ the set $E$
directed away from the root.
Pick a state for the root at random according
to $\pi$. Moving away from the root toward the leaves,
apply the following Markov transition matrix to each edge $e = (u,v)$ independently:
\begin{equation*}
(M(e))_{ij} = (e^{\weight_e Q})_{ij}
=
\left\{
\begin{array}{ll}
1 -(r - 1) \delta_e & \mbox{if $i=j$}\\
\delta_e & \mbox{o.w.}
\end{array}
\right.
\end{equation*}
where
\begin{equation*}
Q_{ij}
=
\left\{
\begin{array}{ll}
-\frac{r - 1}{r} & \mbox{if $i=j$}\\
\frac{1}{r} & \mbox{o.w.}
\end{array}
\right.
\end{equation*} 
(Or equivalently run a continuous-time Markov jump process with rate matrix $Q$ started at the state of $u$.)
Denote the state so obtained by $s_V = (s_v)_{v\in V}$. In particular, $s_{L}$ is the state vector at the leaves,
which we also denote by $s_X$.
The joint distribution of $s_V$ is given by
\begin{equation*}
\mu^T_V(s_V) = \pi(s_\rho)
\prod_{e = (u,v) \in E_\downarrow}
\left[M(e)\right]_{s_{u} s_{v}}.
\end{equation*}
For $W \subseteq V$, we denote by $\mu^T_W$ the marginal
of $\mu^T_V$ at $W$.
We denote by $\dcal[T]$
the probability distribution of $s_V$. 
(It can be shown that the choice of the root does not
affect this distribution. See e.g.~\cite{Steel:94}.)
We also let $\dcal_l[T]$ denote
the probability distribution of
$
s_X := \left(s_{\phi(a)}\right)_{a\in X}.
$
More generally we take
$k$ independent samples $(s^{i}_{V})_{i=1}^k$
from the model above, that is, $s^1_V, \ldots,s^k_V$
are i.i.d.~$\dcal[T]$.
We think of $(s_v^i)_{i=1}^k$
as the sequence at node $v \in V$.
When considering
many samples $(s^i_V)_{i=1}^k$,
we drop the superscript
to refer to a single sample $s_V$.
\end{definition}
\noindent The case $r=4$, known as the Jukes-Cantor (JC)
model~\cite{JukesCantor:69}, is the
most natural choice in the biological context
where,
typically,
$\rcal = \{\mathrm{A}, \mathrm{G},\mathrm{C},\mathrm{T}\}$ and the model describes how DNA sequences
stochastically evolve by point mutations
along an evolutionary tree under the assumption
that each site in the sequences evolves
independently and identically.
For ease of presentation, we restrict ourselves
to the case $r=2$, known as the Cavender-Farris-Neyman (CFN) model~\cite{Cavender:78,Farris:73,Neyman:71},
but our techniques extend to a general $r$
in a straightforward manner. The CFN model
is equivalent to a ferromagnetic Ising
model with a free boundary (see e.g.~\cite{EvKePeSc:00}). 
For now on,
we fix $r=2$. 
We denote
by $\E_T, \P_T$ the expectation and
probability under the CFN model on 
a phylogeny $T$. We will
also use a random cluster representation
of the CFN model, which we recall
in Lemma~\ref{lem:fk}.
\srevision{
It will be convenient to work
on the state space $\{-1,+1\}$
rather than $\{1,2\}$. 
To avoid confusion, we introduce
a separate notation.
Let $\evr = (1,-1)$.
Given samples $(s^i_X)_{i=1}^k$,
we define $\boldsymbol{\s}_X = (\s_X^i)_{i=1}^k$ with
$\s_a^i = \evr_{s_a^i}$ for all $a,i$.
}

\paragraph{Phylogenetic reconstruction.}
In the {\em phylogenetic tree reconstruction (PTR) problem}, we are given
a set of sequences $(\s^i_X)_{i=1}^k$
and our goal is to recover the unknown generating tree.
An important theoretical criterion
in designing a PTR algorithm is
the amount of data required for
an accurate reconstruction.
At a minimum, a reconstruction algorithm should be consistent,
that is, the output
should be guaranteed to converge on the true tree as the sequence
length $k$ goes to $+\infty$.
Beyond consistency, the {\em sequence-length requirement (SLR)} of a
PTR algorithm is the sequence length required for a guaranteed high-probability
reconstruction. Formally:
\begin{definition}[Phylogenetic Reconstruction Problem]
A {\em phylogenetic
reconstruction algorithm} is a collection of maps
$\Psi = \{\Psi_{n,k}\}_{n,k \geq 1}$
from sequences $(\s^i_{X})_{i=1}^k \in (\{-1,+1\}^{X})^k$
to leaf-labeled trees in $\trees_n$,
where $X = [n]$.
Fix $\delta > 0$ (small) and
let $k(n)$ be an increasing function of $n$.
We say that $\Psi$
solves the {\em phylogenetic reconstruction problem}
on $\phy_{f,g}[\quantum]$
with sequence length $k = k(n)$
if for all $n \geq 1$,
and all $T \in \phy^{(n)}_{f,g}[\quantum]$,
\begin{equation*}
\prob\left[\Psi_{n,k(n)}\left((\s^i_{X})_{i=1}^{k(n)}\right) =
\tcal_l[T]\right] \geq 1 - \delta,
\end{equation*}
where $(\s^i_{X})_{i=1}^{k(n)}$ are i.i.d.~samples from $\dcal_l[T]$.
We let $k_0(\Psi,n)$ be the smallest
function $k(n)$ such that the above condition holds
(for fixed $f, g, \quantum, \delta$).
\end{definition}
\noindent We call the function $k_0(\Psi,n)$ the
{\em sequence-length requirement (SLR)}
of $\Psi$. For simplicity we
emphasize the dependence on $n$.
Intuitively the larger the tree, the more data is required to
reconstruct it.
One can also consider the dependence of $k_0$
on other structural parameters.
In the mathematical
phylogenetic literature,
the SLR has emerged  as a key measure
to compare the statistical performance of
different reconstruction methods.
A lower $k_0$ suggests a better statistical performance.
Note that, ideally,
one would like to compute the probability
that a method succeeds given a certain amount
of data, but that probability is a complex function
of all parameters. Instead the SLR, which can be
bounded analytically, is a proxy that measures
how effective a method is at extracting phylogenetic
signal from molecular data.

\paragraph{Maximum likelihood estimation}
The maximum likelihood (ML) estimator
for phylogenetic reconstruction
is given (in our setting) by
\begin{equation}\label{eq:mle}
\psiml_{n,k}((\s^i_X)_{i=1}^k)
\in \argmin_{T \in \phy^{(n)}_{f,g}[\quantum]}
\L_T[(\s^i_X)_{i=1}^k],
\end{equation}
where
$
\L_T[(\s^i_X)_{i=1}^k]
= - \sum_{i=1}^k
\ln \mu^T_{X}(\s^i_{X})
$ (breaking ties arbitrarily).
In words the ML selects a phylogeny
which maximizes the probability of observing
the data.
Computation of the likelihood on a given phylogeny can
be performed efficiently,
but solving the maximization problem above over
tree space is computationally intractable~\cite{Roch:06,ChorTuller:06}.
Fast heuristics have been developed
and are widely used~\cite{GuLeDuGa:05,Stamatakis:06}.
Despite the practical importance
of ML, much remains to be understood about its
statistical properties.
Consistency, that is, the convergence of
the ML estimate $\hattml$
on the true tree as the number of sites
$k \to \infty$, has been
established~\cite{Chang:96}. But
obtaining tight bounds on
the SLR of ML has remained an outstanding
open problem in mathematical phylogenetics.
The best previous known bound, due
to~\cite{SteelSzekely:02} was that
under the CFN model
there exists $K > 0$
such that $k_0(\psiml,n) \leq \exp(K n)$.

\subsection{Main results}
\label{section:results}

Our main result is the following.
\begin{theorem}[Sequence-length requirement of maximum likelihood]\label{thm:main}
Let $0 < \quantum < f < g^* := \ln\sqrt{2}$.
Then the sequence-length requirement of maximum likelihood for the phylogenetic tree reconstruction problem on
$\phy_{f,g}[\quantum]$
is
\begin{displaymath}
k_0(\psiml, n)
= \begin{cases}
O(\log n), & \text{if $g < g^*$,}\\
\poly(n), & \text{if $g \geq g^*$}.
\end{cases}
\end{displaymath}
\end{theorem}
\noindent Combined with the results of~\cite{DaMoRo:11a}, this bound implies that the ML
estimator can be computed in polynomial
time with high probability as long as $k \geq k_0(\psiml,n)$.
Note that our definition of the ML
estimator implicitly assumes that we know
(or have bounds on) the
parameters $f,g,\quantum$ as the search
is restricted over the space $\phy^{(n)}_{f,g}[\quantum]$. In practice
it is not unnatural to restrict the space
of possible models in this way.
\srevision{We note finally that our proof in the regime $g \geq g^*$ holds under a much weaker discretization assumption (see below).}

\subsection{Proof overview}
\label{section:overview}

\paragraph{Known results: identifiability, consistency and
the Steel-Sz{\'e}kely bound}
Before sketching the proof of Theorem~\ref{thm:main},
we first mention previously known facts about
the statistical properties of ML in phylogenetics.
Fix $f,g,\quantum, n$
and let $\phy = \phy^{(n)}_{f,g}[\quantum]$.
Let $T^0 \in \phy$
be the generating phylogeny
and denote by $\boldsymbol{\s}_X = (\s_X^i)_{i=1}^k$
a set of $k$ samples from the corresponding
CFN model.
Under our assumptions, the model
is known to be identifiable~\cite{Chang:96}, that is,
$$
T^0 \neq T^\# \in \phy
\implies
\dcal_l[T^0] \neq \dcal_l[T^\#].
$$
Moreover
the ML estimator is known to converge on $T^0$
almost surely as $k \to \infty$~\cite{Chang:96}.
That fact follows from the law of large numbers
by which
$$
\frac{1}{k}\L_{T^\#}(\boldsymbol{\s}_X) \to -\E_{T^0}[\ln \mu^{T^\#}_X(\s_X)],
$$
as $k \to \infty$,
identifiability,
the positivity of the
Kullback-Leibler (KL) divergence,
that is,
$$
T^0 \neq T^\#
\implies
\kl(T^0\,\|\,T^\#)
:= -\E_{T^0}[\ln \mu^{T^\#}_X(\s_X)]
+ \E_{T^0}[\ln \mu^{T^0}_X(\s_X)] > 0,
$$
and a compactness argument~\cite{Wald:49}.

Steel and Sz\'ekely~\cite{SteelSzekely:02} also
derived along the same lines 
a quantitative upper bound on the
SLR.
They used Pinsker's inequality
to lower bound the KL divergence with the total
variation distance. And they appealed to
concentration inequalities
to bound the probability that any leaf vector state frequency
is away from its expectation, thereby 
quantifying the speed of convergence of the
log-likelihood. 
The argument ends up depending inversely on the
lowest non-zero state probability,
which is exponentially small in $n$,
leading to an exponential SLR. The Steel-Sz\'ekely bound does not make use of the structure
of the phylogenetic problem and, in fact, is
derived in a more general setting.

\paragraph{A polynomial bound}
In order to make use of the structure of the problem,
we propose a different approach.
The basic idea is to design for each
incorrect tree $T^\#$ a {\em statistical test}
that excludes it from being
selected by ML with high probability.
We first illustrate this idea by sketching a polynomial
bound on the SLR of ML. This proves the polynomial
regime of Theorem~\ref{thm:main}.

In~\cite{ErStSzWa:99a}, a reconstruction
algorithm was provided that, for any $g$,
returns the correct phylogeny with
probability $1 - \exp(-n^{C_1})$
as long as $k \geq n^{C_2}$ for a large enough
$C_2 > 0$. We refer to this algorithm as
the ESSW algorithm.
Letting $T^0$ be the true phylogeny
generating the data and $T^\# \neq T^0$ be
in $\phy$, denote by $D_{T^\#}$
the event that the ESSW algorithm
reconstructs (incorrectly) $T^\#$ and
by $M_{T^\#}$ the event that ML prefers $T^\#$
over $T^0$ (including a tie),
that is, the set of $\boldsymbol{\s}_X = (\s^{i}_X)_{i=1}^k$
such that
$
\L_{T^\#}(\boldsymbol{\s}_X)
\leq \L_{T^0}(\boldsymbol{\s}_X)
$ or equivalently
\begin{equation}\label{eq:mt}
\frac{\mu^{T^\#}_X(\boldsymbol{\s}_X)}{\mu^{T^0}_X(\boldsymbol{\s}_X)}
\geq 1.
\end{equation}
Then, a classical result in hypothesis testing
(see e.g.~\cite[Chapter 13]{LehmannRomano:05})
is that the sum of Type-I and Type-II errors
is minimized by the likelihood ratio test,
which in our context amounts to
\begin{equation}
\label{eq:hypothesis-testing}
\P_{T^0}[M_{T^\#}] + \P_{T^\#}[M_{T^\#}^c]
\leq \P_{T^0}[A] + \P_{T^\#}[A^c],
\end{equation}
for any test (i.e., event) $A \subseteq [r]^{nk}$.
Taking in particular $A = D_{T^\#}$, we get from~\cite{ErStSzWa:99a}
that
\begin{equation}
\label{eq:mt2}
\P_{T^0}[M_{T^\#}]
\leq \P_{T^0}[D_{T^\#}]
+ \P_{T^\#}[D_{T^\#}^c]
\leq 2 e^{-n^{C_1}},
\end{equation}
whenever $k \geq n^{C_2}$.
Recall (e.g.~\cite{SempleSteel:03})
that the number of binary
trees on $n$ labeled leaves is
$
(2n - 5)!! = e^{O(n \log n)}.
$
For each such tree, our discretization assumption
implies that there are at most $((g-f)\quantum + 1)^n$
choices of branch lengths.
Hence, provided we choose $C_1$ and
$C_2$ large enough
and taking a union bound over the $e^{O(n \log n)}$
possible trees $T^\# \neq T^0$
in $\phy$,
we obtain:
under our assumptions,
there exists $K > 0$
such that $k_0(\psiml,n) \leq n^K$.
In fact, note that this argument still works when the discretization $\quantum$
is of order $n^{-C_3}$ for any $C_3 > 0$.

This new bound on $k_0(\psiml,n)$ improves significantly over
the Steel-Sz\'ekely bound.
It has interesting computational implications
as well.
Although ML for phylogenetic
reconstruction is NP-hard~\cite{Roch:06,ChorTuller:06},
our polynomial SLR bound
in combination with the computationally efficient
ESSW algorithm
indicates that the ML estimator can be
computed efficiently with high probability
when data is generated from a CFN model
with polynomial sequence lengths.

\paragraph{A refined union bound}
Dealing with logarithmic-length sequences
is significantly more challenging. 
As the argument below suggests,
certain close-by trees cannot
be distinguished using logarithmic-length
sequences {\em with exponentially small
	failure probability}. In particular
the naive union bound above cannot
work in this regime.
Instead
we use a more refined union bound.

We make two
observations. 
We introduce $\blowup(T^\#,T^0)$,
the {\em blow-up distance}
between the topologies
of $T^\#$ and $T^0$, that is, roughly
the smallest number of edges that need to be rearranged
to produce $T^\#$ from $T^0$
(see Definition~\ref{def:blowup} for
a formal definition).
The number of
trees at blow-up distance $D$ from $T^0$ is at most
$O(n^{2D})$ so that it suffices to prove
\begin{equation}\label{eq:tvspr}
\P_{T^0}[M_{T^\#}] \leq C_1 e^{-C_2 k \blowup(T^\#,T^0)},
\end{equation}
in order to apply a union bound over blow-up
distances, when $k$ is logarithmic in $n$.
To prove~\eqref{eq:tvspr}, we need
to use an appropriate test $A \subseteq [r]^{nk}$ in~\eqref{eq:hypothesis-testing} and~\eqref{eq:mt2}---as we did before---but now the error probability of the test must depend on
the blow-up distance between $T^\#$ and $T^0$. 
That is, we need a test $A \subseteq [r]^{nk}$ such that
\begin{equation}\label{eq:spra}
\P_{T^0}[A^c] + \P_{T^\#}[A]
\leq C_1 e^{-C_2 k \blowup(T^\#,T^0)}.
\end{equation}
This
is intuitively reasonable as we expect similar
trees to be harder to distinguish.

We note in passing that~\eqref{eq:hypothesis-testing}
follows from the fact that the likelihood ratio
test achieves the total variation
distance between the models generated by $T^\#$ and $T^0$
under $k$ samples, which we denote by $\tv^k(T^\#,T^0)$.
Thus, our main
technical contribution can be interpreted
as
relating combinatorial and variational distances
between trees. 
This claim, which may be of independent interest, 
is proved along with
Theorem~\ref{thm:log-bound}.
\begin{lemma}[Relating combinatorial and
	variational distances]
	\label{lemma:tv-bl}
	For $T^\#, T^0 \in \phy_{f,g}[\quantum]$
	with $g < g^*$,
\begin{equation*}
\tv^k(T^\#,T^0)
\geq 1 - C_1 e^{-C_2 k \blowup(T^\#,T^0)}.
\end{equation*}
\end{lemma}

\paragraph{Phase transition: Homogeneous case}
We sketch our construction of the test $A$ above
in the special case of homogeneous trees.
Fix $g, n = 2^h$
and let $\hmgphy = \hmgphy^{(h)}_{g}$.
Let $T^0 \in \hmgphy$
be the generating phylogeny
and denote by $\boldsymbol{\s}_X = (\s^i_X)_{i=1}^k$
a set of $k$ samples from the corresponding
CFN model. In the homogeneous case,
it will be more convenient to work with
we call the swap distance  $\swap(T^\#,T^0)$,
which is defined, roughly, as the smallest number of
same-level swaps of subtrees of $T^\#$ in order
to obtain $T^0$.
(See Section~\ref{sec:homo} for a formal definition.)

Recall that a {\em cherry} is a pair
of leaves with a common immediate ancestor.
A result of~\cite{MosselRoch:12} shows that if
$T^\#$ is obtained from $T^0$ by applying a uniformly
random permutation of the leaf labels of $T^0$
then, with high probability, there is a positive
fraction (independent of $n$) of the cherries
in $T^0$ such that the corresponding leaves
in $T^\#$ are far (at least a large constant
graph distance away) from each other. Let
$\ccal$ be such a collection of cherries.
As a result, it was shown that
the total pairwise correlation over $\ccal$
as measured for instance by
$$
\zcal^i = \frac{1}{n}\sum_{(a,b) \in \ccal} \s_a^i \s_b^i,
$$
is concentrated on two well-separated values
under $T^0$ and $T^\#$, and the event
$$
A =
\left\{\frac{1}{k} \sum_{i = 1}^k \zcal^i > z\right\},
$$
for a well-chosen value of $z$,
satisfies an
exponential bound as in~\eqref{eq:spra}.

Returning to our context this argument suggests
that, if the incorrect tree $T^\#$ is far from the
generating tree $T^0$ in swap distance,
a powerful enough test can
be constructed from the cherries
of $T^0$. 
One of our main contributions
is 
to show how to generalize this idea
to trees at an {\em arbitrary} combinatorial distance. 
This is non-trivial because
$T^\#$ and $T^0$ may only differ by
{\em deep} swap moves, in which case
cherries cannot be used in distinguishing
tests. Instead, we show
how to find
{\em deep} pairs of test nodes that are close
under $T^0$, but somewhat far under $T^\#$ (see Proposition~\ref{prop:existence-batteries-hmg}).
To build a corresponding test, we
reconstruct the ancestral
states at the test nodes and estimate
the correlation between the reconstructed
values as above (see Proposition~\ref{prop:distinguishing}). Note that
the reconstruction phase transition plays a critical role
in this argument.

The main challenge is to find such deep test pairs
and relate their number to the swap distance. 
For this purpose, we design a procedure that identifies dense subtrees that are shared by 
$T^\#$ and $T^0$, working recursively from the leaves up
(see Claim~\ref{claim:cohanging-hmg})
and we prove that this procedure leads to
a number of tests that grows linearly 
in the swap distance (see Claim~\ref{claim:swaps-r}).
A further issue is to guarantee enough
independence between the tests, which we
accomplish via a sparsification step (see
Claim~\ref{claim:sparsification-hmg}).
The full argument for homogeneous trees is
in Section~\ref{sec:homo}.

\paragraph{General case} In the homogeneous
case, we produce a sufficient number of deep test pairs by
identifying subtrees that are matching in $T^\#$
and $T^0$. As we mentioned above, that can be done recursively starting
from the leaves. In
the case of general trees,
the lack of symmetry makes
this task considerably more challenging. One significant new
issue that arises is that the matching subtrees 
found through the same type of procedure
may in fact
``overlap'' in $T^\#$,
that is, have a non-trivial intersection. 

Hence, to construct a linear number of tests
in blow-up distance, we proceed
in two phases. We first attempt to identify matching subtrees similarly to the homogeneous case.
We show that if the overlap produced
is small, then a linear number of tests
(see Claim~\ref{claim:blowup-overlap})
can be constructed in a manner similar
to the homogeneous case
(see Proposition~\ref{prop:battery-many-r}),
although several new difficulties arise. See
Section~\ref{section:manyr} for details.

On the other hand, if the overlap in $T^\#$ is too large, then the first
phase will fail. In that case, we show that
a sufficient number of
deep test pairs can be found around the
``boundary of the overlap'' in $T^\#$ (see Proposition~\ref{prop:battery-large-overlap}).
That construction is detailed in Section~\ref{section:largeOverlap}.


\section{Distinguishing between leaf distributions}
\label{sec:prelim}
\label{sec:distinguishing-tests}

In this section,
we detail our main tool for distinguishing
between the leaf distributions of different phylogenies.
Fix $f,g < g^*,\quantum, n$
and let $\phy = \phy^{(n)}_{f,g}[\quantum]$.
Let $T^0 \in \phy$
be the generating phylogeny
and denote by $\boldsymbol{\s}_X = (\s^i_X)_{i=1}^k$
a set of $k$ i.i.d.~samples from the corresponding
CFN model.

As outlined in Section~\ref{section:overview},
our strategy is to construct for each
erroneous tree $T^\#\neq T^0$ a statistical
test that distinguishes between the two
leaf distributions.
The classification error of the test
will ultimately depend on the
combinatorial distance between $T^\#$ and
$T^0$. We show in Sections~\ref{sec:homo}
(for homogeneous trees)
and~\ref{sec:general} (for general trees) how to construct
such tests. Here we define formally the type of test
we seek to use
and derive bounds on their classification error.
%

\subsection{Definitions}

We first need several definitions.
Let $T = (V,E;\phi;w)$
be a phylogeny in $\phy$ and denote its leaf set by $L$.
Recall from Definition~\ref{def:phylo} that $X = [n]$ is the set of leaf labels.
We will work with a special type of
subtrees defined as follows.
\begin{definition}[Restricted subtree]
\label{def:restricted}
A (connected) subtree $Y$ of $T$ is {\em restricted} if
there exists $V_R \subseteq V$
such that $Y$ is obtained by keeping only
those edges of $T$ lying on the path
between two vertices in $V_R$. We typically
restrict $T$ to a subset of the leaves
(in which case we denote $V_R$ by $L_R$ instead).
When $|V_R| = 4$, $Y$ is called a {\em quartet}.
The topology of a binary quartet on $V_R = \{u,v,x,y\}$
is characterized by the pairs in $V_R$ lying on each side
of the internal edge, e.g., we write $uv|xy$ if
$\{u,v\}$ and $\{x,y\}$ are on opposite sides.
Let $Y$ and $Z$ be restricted subtrees
of $T$.
We let $Y\cap Z$ (respectively $Y \cup Z$)
be the {\em intersection} (respectively the {\em union})
of the edge sets of $Y$ and $Z$.
\end{definition}
\noindent We will need to
compare restricted subtrees in $T^\#$ and $T^0$.
For this purpose, we will use the following
metric-based definition. We first
recall the notion of a tree metric.
\begin{definition}[Tree metric]
	\label{def:metric}
	A phylogeny $T = (V,E;\phi;\weight)$
	is naturally equipped with a
	{\em tree metric}
	$\dist_T : X\times X \to (0,+\infty)$ defined as follows
	\begin{equation*}
	\forall a,b \in X,\ \dist_T(a,b)
	= \sum_{e\in\path_T(\phi(a),\phi(b))} \weight_e,
	\end{equation*}
	where $\path_T(u,v)$ is the set of edges on the path between
	$u$ and $v$ in $T$.
	We will refer to $\dist_T(a,b)$ as the {\em evolutionary
		distance} between $a$ and $b$.
	In a slight abuse of notation,
	we also sometimes use
	$\dist_T(u,v)$ to denote the
	evolutionary distance between
	any two vertices $u,v$ of $T$ as defined above.
	We will also let $\gdist_T(a,b)$ denote the
	graph distance between $a$ and $b$ in $T$,
	that is, the number of edges on the path between
	$a$ and $b$ in $T$.
\end{definition}
\noindent Tree metrics satisfy the following {\em four-point
		condition}: $\forall a_1,a_2,a_3,a_4 \in X$,
	\begin{equation}\label{eq:four-point}
	\dist_T(a_1,a_2)
	+ \dist_T(a_3,a_4) \leq \max\{
	\dist_T(a_1,a_3)
	+ \dist_T(a_2,a_4),
	\dist_T(a_1,a_4)
	+ \dist_T(a_3,a_2)
	\}.
	\end{equation}
	In the {\em non-degenerate} case, one of the three
	sums above is strictly smaller than the other two,
	which are equal. From the four-point condition,
	it can be shown that to each tree metric
	corresponds a unique phylogeny (with positive
	edge weights). See e.g.~\cite{SempleSteel:03}.
\begin{definition}[Matching subtrees]
\label{def:matching}
Let $T = (V,E;\phi;\weight)$ and
$T'= (V',E';\phi';\weight')$
be trees in $\phy$
with $n$ leaves $L$ and $L'$
respectively (and the same leaf label set
$X = [n]$). Let $Y$ and $Y'$
be subtrees of $T$ and $T'$
restricted respectively to
leaf sets $L_R \subseteq L$
and $L_R' \subseteq L'$ spanning the same leaf labels, that is, $\phi(L_R) = \phi(L_R')$. We say that
$Y$ and $Y'$ are {\em metric-matching}
or simply {\em matching}
if:
the tree metrics corresponding
to $Y$ and $Y'$ are identical.
Note that, even if $Y$ and $Y'$ are metric-matching,
their vertex and edge sets may differ. E.g.,
an edge in $Y$ may correspond to a
(non-trivial) path in $Y'$, and vice versa.
However, thinking of $Y$ and $Y'$ as continuous
objects, for each vertex $v \in Y$, we can create
a corresponding {\em extra vertex} $v'$ in $Y'$.
\end{definition}
\noindent As we mentioned above, we will assign
a distinguished vertex to each subtree 
included in the tests. We think of these
as roots. The following definitions apply to 
such rooted subtrees.
\begin{definition}[Dense subtree]
\label{def:dense-subtree}
Let $\ell$ and $\wp \leq 2^{\ell}$ be nonnegative integers.
Let $Y$ be a restricted subtree of $T$ rooted at $y$.
The {\em $\ell$-completion $\lfloor Y\rfloor_\ell$}
of $Y$ is obtained
by adding complete binary subtrees
with $0$-length edges below the leaves of $Y$
so that all leaves in $\lfloor Y\rfloor_\ell$ are at the
same graph distance from $y$ and the height of
$\lfloor Y\rfloor_\ell$
is the smallest multiple of
$\ell$ greater than the height of $Y$.
We say that
$Y$ is {\em $(\ell,\wp)$-dense} in $T$
if: the number of vertices on the $(i\ell)$-th
level of $\lfloor Y\rfloor_\ell$ is at least $(2^{\ell} - \wp)^i$
for all $i \geq 0$ such that $i\ell$ is smaller than the
height of $\lfloor Y\rfloor_\ell$.
\end{definition}
\begin{definition}[Co-hanging subtrees]
	\label{def:co-hanging}
	Two rooted restricted subtrees $Y$ and $Z$
	of a tree $T$
	with empty (edge) intersection
	are {\em co-hanging} if the path between
	their roots does not intersect the edges
	in their union.
	The {\em linkage $Y \oplus Z$}
	of co-hanging rooted restricted subtrees $Y$ and $Z$
	is the (unrooted) restricted subtree
	obtained by adding to $Y$ and $Z$
	the path joining their roots.
\end{definition}
\noindent We need one last definition.
\begin{definition}[Topped subtree]
Let $T$ be rooted.
Let $\gamma \in \nintgr$ and $Y$ be a restricted subtree
of $T$ rooted at $y$.
The {\em $\gamma$-topping $\lceil Y \rceil^\gamma$}
of $Y$ is obtained from $Y$ by adding the $\gamma$
edges
immediately above $y$ on the path to the root of $T$
(or the entire path if it is has fewer than $\gamma$
edges), which we refer to as the {\em hat}
of $\lceil Y \rceil^\gamma$.
\end{definition}

\subsection{Batteries}

In the proof below, we will compare the
true phylogeny
$T^0 = (V^0,E^0;\phi^0;\weight^0)$
to an incorrect phylogeny, which will be denoted by
$T^\# = (V^\#,E^\#;\phi^\#;\weight^\#)$.
Assume that $T^0$ and $T^\#$ are rooted
at $\rt^0$ and $\rt^\#$ respectively.
The comparison will be based
on the following combinatorial definition
and the associated statistical test below.
A {\em test pair} in $T^0$
is a pair of vertices
(leaf or internal; possibly extra) $(y^0, z^0)$ in $T^0$,
which we will refer to as {\em test roots},
as well as a pair of restricted subtrees $(Y^0, Z^0)$ of $T^0$
rooted at $y^0$, $z^0$ respectively, which we will refer
to as {\em test subtrees}.
Similarly we define a test pair in $T^\#$.
We call a {\em test panel} two corresponding
test pairs in $T^0$ and $T^\#$. See Figure~\ref{fig:non-proximal} for an illustration.
\begin{figure}
	\centering
	\includegraphics[width = 0.9\textwidth]{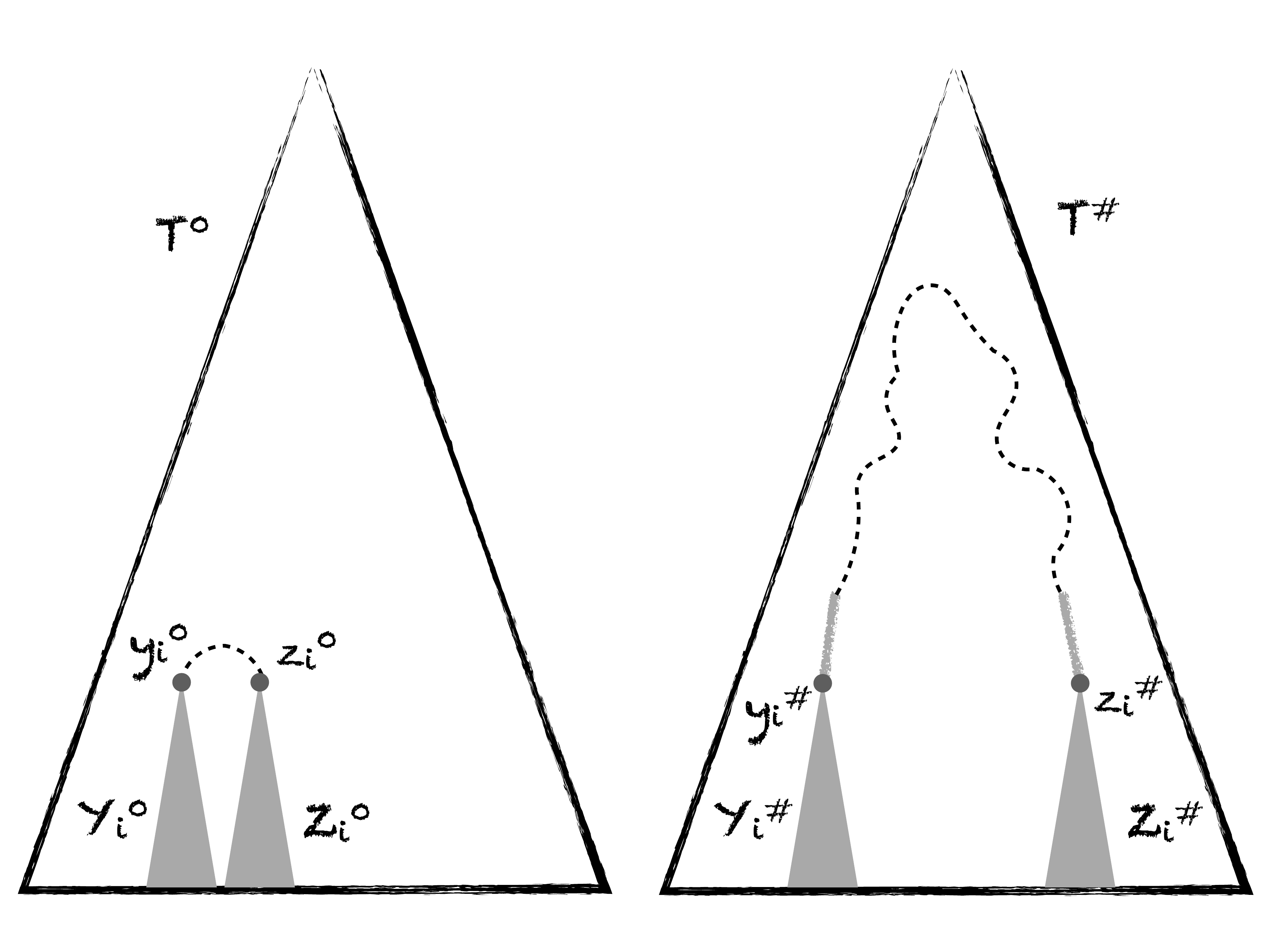}
	\caption{A test panel: proximal in $T^0$ and non-proximal in $T^\#$.}\label{fig:non-proximal}
\end{figure}

\srevision{At a high level, the idea behind our distinguishing statistic is to consider
pairs of subtrees, the test pairs, that are shared between
$T^\#$ and $T^0$ in the sense of Definition~\ref{def:matching} 
(Condition 2(a) below)
and that further have the property 
that the distance
between their roots differ in $T^\#$
and $T^0$ (Condition 2(d) below). The test itself 
(defined formally in Equations~\eqref{eq:asr-estimate},~\eqref{eq:distinguishing-statistic} and~\eqref{eq:distinguishing-statistic-2} below)
involves reconstructing the ancestral
states at the roots of the pairs
and comparing their correlation on
$T^\#$
and $T^0$.
To ensure a strong enough signal,
we require 
that the
subtrees are dense enough
in the sense Definition~\ref{def:dense-subtree}
(Condition 1(a) below)
to guarantee accurate reconstruction
of ancestral states
and that
the roots are close
(within a parameter $\Gamma$)
in either $T^0$ or $T^\#$
(Condition 2(c) below).
Although each test panel
contains at least one pair whose roots are close,
the other pair may not be---potentially producing unwanted
dependencies between the test panels
in cases where the roots are particularly 
far from each other (at distance at least $\gamma_t$). Such dependencies
are dealt with in Proposition~\ref{prop:distinguishing} below.
Another requirement of the test is that the paths
connecting the roots of each pair
do not intersect the corresponding
subtrees, in the sense of Definition~\ref{def:co-hanging} (Condition 2(b) below). That last property ensures that
the errors of the ancestral state estimates
are conditionally independent given
the root states.}
\begin{definition}[Battery of Tests]\label{def:battery}
Fix nonnegative integers $\ell \geq 2$,
$0 \leq \wp \leq 2^\ell - 1$,
$\Gamma \geq 1$ and $\gamma_t \geq 1$.
We say that a collection of test panels
$$
\{((y^0_i,z^0_i);(Y^0_i,Z^0_i))\}_{i=1}^I
\text{ in $T^0$
and }
\{((y^\#_i,z^\#_i);(Y^\#_i,Z^\#_i))\}_{i=1}^I
\text{ in $T^\#$}
$$
form an {\em $(\ell,\wp,\Gamma,\gamma_t,I)$-battery}
if:
\begin{enumerate}
\item {\bf Cluster requirements}
\begin{enumerate}
\item {\em (Dense subtrees)} All test subtrees
are $(\ell,\wp)$-dense.

\end{enumerate}
\item {\bf Pair requirements}
\begin{enumerate}
\item {\em (Matching subtrees)} The subtrees $Y^0_i$ and $Y^\#_i$
are matching for all $i = 1, \ldots, I$,
and similarly for $Z^0_i$ and $Z^\#_i$.

\item {\em (Co-hanging)} For $i=1,\ldots,I$,
we require that
$Y^0_i$ and $Z^0_i$ be co-hanging.
Similarly for the pairs in $T^\#$.

\item {\em (Proximity)} For $i=1,\ldots,I$,
if the graph distance between $y^0_i$
and $z^0_i$
is less than $\Gamma$,
we say that
the corresponding pair is {\em proximal}.
(If $y^0_i$
or $z^0_i$ is an extra vertex
we use the graph distance in $T^0$ to
the closest neighbor.)
Else,
if the graph distance between $y^0_i$
and $z^0_i$ is less than $\gamma_t$,
we say that
the corresponding pair is {\em semi-proximal}.
In both proximal and semi-proximal cases,
we
let
\begin{equation}\label{eq:pair2}
\mathcal{F}^0_i = Y^0_i \oplus Z^0_i.
\end{equation}
Else,
if the graph distance between $y^0_i$
and $z^0_i$ is greater than $\gamma_t$,
in which case we say that
the corresponding pair is {\em non-proximal},
we
define (with a slight abuse of notation)
\begin{equation}\label{eq:pair1}
\mathcal{F}^0_i = \lceil Y^0_i \rceil^{\gamma_t}
\cup
\lceil Z^0_i \rceil^{\gamma_t},
\end{equation}
to be the forest with corresponding edge set. 
We refer to the path
between $y^0_i$
and $z^0_i$ in $T^0$ as the {\em connecting path}
of the test pair. (In the non-proximal
case, the hats of $Y^0_i$ and $Z^0_i$
may not lie entirely on the connecting path.)
We similarly define $\mathcal{F}_i^\#$s from the
pairs in $T^\#$.

\item {\em (Evolutionary distance)} For each $i = 1, \ldots, I$, we have
$$
\left|\dist_{T^0}(y^0_i,z^0_i) -
\dist_{T^\#}(y^\#_i,z^\#_i)\right|
\geq \quantum,
$$
and at least one of the corresponding pairs
is proximal. Further, we let
\begin{equation}
\label{eq:def-alpha}
\alpha_i
=
\begin{cases}
+1, & \text{if $\dist_{T^0}(y^0_i,z^0_i) < \dist_{T^\#}(y^\#_i,z^\#_i)$}\\
-1, & \text{o.w.}
\end{cases}
\end{equation}

\end{enumerate}

\item {\bf Global requirements}
\begin{enumerate}
\item {\em (Global intersection)} The $\mathcal{F}^0_i$s
have empty pairwise intersection. Similarly
for the $\mathcal{F}^\#_i$s.
\end{enumerate}

\end{enumerate}
\end{definition}

\paragraph{Tests}
For a restricted subtree $Y$ rooted at $y$,
we denote by $X[Y]$ the leaf labels of $Y$
and we let
\begin{equation}
\label{eq:asr-estimate}
\hs_y^j
=
\begin{cases}
+1, & \text{if $\P_Y\left[\s_y = +1\,\Big|\,\s_{X[Y]}^j\right]
> \P_Y\left[\s_y = -1\,\Big|\,\s_{X[Y]}^j\right]$},\\
-1, & \text{o.w.}
\end{cases}
\end{equation}
be the MLE of the state at $y$ on site $j$,
given $\s_{X[Y]}^j$.
Let
$$
\{((y^0_i,z^0_i);(Y^0_i,Z^0_i))\}_{i=1}^I
\text{ in $T^0$
and }
\{((y^\#_i,z^\#_i);(Y^\#_i,Z^\#_i))\}_{i=1}^I
\text{ in $T^\#$}
$$
form a {\em $(\ell,\wp,\Gamma,\gamma_t,I)$-battery}
with corresponding $\alpha_i$s
(as defined in~\eqref{eq:def-alpha}).
The {\em distinguishing statistics} of the battery
are defined as
\begin{equation}
\label{eq:distinguishing-statistic}
\hdcal^0 = \sum_{i=1}^I \sum_{j=1}^k
\alpha_i\hs^j_{y^0_i}\hs^j_{z^0_i},
\quad \text{and}
\quad
\hdcal^\# = \sum_{i=1}^I \sum_{j=1}^k
\alpha_i\hs^j_{y^\#_i}\hs^j_{z^\#_i}.
\end{equation}
We observe that, because the subtrees in $T^0$
and $T^\#$ are matching (that is, they are identical
as sub-phylogenies as remarked after Definition~\ref{def:metric}), $\hdcal^0$
and $\hdcal^\#$ are in fact identical {\it as a function
of the leaf states}, which we denote
by $\hdcal$. However their
distributions, in particular their means
$\dcal^0 = \E_{T^0}[\hdcal]$ and $\dcal^\#
= \E_{T^\#}[\hdcal]$
respectively, differ as we quantify below.
The distinguishing
event is then defined as
\begin{equation}
\label{eq:distinguishing-statistic-2}
A = \left\{
\hdcal - \frac{\dcal^0 + \dcal^\#}{2} > 0
\right\}.
\end{equation}

\paragraph{Properties of batteries}
We show that the distinguishing event $A$ is likely
to occur under $T^0$, but unlikely to occur under $T^\#$.
The proof is in the next section.
\begin{proposition}[Batteries are distinguishing]
\label{prop:distinguishing}
For any positive integers $\wp$ and $\Gamma$,
there exist constants
$\ell = \ell(g,\wp) \geq 2$ large enough,
$\gamma_t = \gamma_t(g, \wp, \ell, \Gamma, \invquantum)$
large enough,
and
$C = C(g, \wp, \ell, \Gamma, \invquantum, \gamma_t) > 0$
small enough
such that the following holds.
If
$$
\{((y^0_i,z^0_i);(Y^0_i,Z^0_i))\}_{i=1}^I
\text{ in $T^0$
and }
\{((y^\#_i,z^\#_i);(Y^\#_i,Z^\#_i))\}_{i=1}^I
\text{ in $T^\#$},
$$
form a {\em $(\ell,\wp,\Gamma,\gamma_t,I)$-battery}
with corresponding $\alpha_i$s,
$\hdcal$, $\dcal^0$, $\dcal^\#$, and $A$,
then
$$
\max
\{
\P_{T^0}[A^c],
\P_{T^\#}[A]
\}
\leq
\exp\left(
- C k I
\right),
$$
for all $I$ and $k$.
\end{proposition}

\subsection{Proof of Proposition~\ref{prop:distinguishing}}

We give a proof of Proposition~\ref{prop:distinguishing}.
The proof has several steps:
\begin{enumerate}
	\item We bound the accuracy of the ancestral
	state estimator~\eqref{eq:asr-estimate}
	using \srevision{Lemma~\ref{lem:ekps} from Appendix~\ref{section:appendix-1}}.
	\srevision{
	\begin{claim}[Accuracy of ancestral reconstruction]
	\label{claim:prop1-claim1}
	There is $\ell \geq 2$ large enough and a constant $0 < \beta_{g,\wp} < +\infty$ depending
	on $g$ and $\wp$ such that for
	any subtree $Y$ in the battery, it holds that
	\begin{equation}\label{eq:bias}
	\P_Y[\hs_y = \s_y] \geq \frac{1 + e^{-\beta_{g,\wp}}}{2},
	\end{equation}
	where  $y$ is the root of $Y$.
	\end{claim}
	Crucially $\beta_{g,\wp}$
	does not depend on $n$, that is,
	the accuracy of the reconstruction
	does not deteriorate as one considers
	larger, deeper trees.}

	\item We show that the distinguishing 
	statistics~\eqref{eq:distinguishing-statistic} have well-separated expectations. That follows
	from the fact that, by the assumption 2(d) in Definition~\ref{def:battery}, the evolutionary distances
	between the roots of the corresponding
	subtrees differ on $T^\#$ and $T^0$. 
	The accuracy of the ancestral state estimation in Claim~\ref{claim:prop1-claim1}
	also guarantees that the signal is strong
	enough at the leaves.
	\srevision{
	\begin{claim}[Separation of expectations]
	\label{claim:prop1-claim2}
	There exists $\dcal_\delta > 0$
	depending on
	$g$, $\wp$, $\Gamma$ and $\invquantum$
	such that
	\begin{equation}\label{eq:dcal-delta}
	\dcal^0 - \dcal^\#
	\geq \dcal_\delta k I.
	\end{equation}	
	\end{claim}
	}
	
	\item Finally, in the more delicate
	step of the argument, we establish that the distinguishing
	statistics~\eqref{eq:distinguishing-statistic} 
	are concentrated around their respective
	means. 
	\srevision{
	\begin{claim}[Concentration]
	\label{claim:prop1-claim3}
	There is 
	$\gamma_t >0$
	large enough
	and
	$C > 0$
	small enough
	such that
	$$
	\max
	\{
	\P_{T^0}[A^c],
	\P_{T^\#}[A]
	\}
	\leq
	\exp\left(
	- C k I
	\right),
	$$
	for all $I$ and $k$,
	where $A$ is defined in~\eqref{eq:distinguishing-statistic-2}.
	\end{claim}
	}
	Proving concentration is complicated
	by the fact that the terms in the sums~\eqref{eq:distinguishing-statistic}
	are not independent. That is the result
	of the non-proximal pairs having connecting
	paths that may intersect with other test subtrees. \srevision{When the number of non-proximal
	pairs is not small, we show that the 
	corresponding terms are ``almost 
	independent'' of the other terms
	by bounding the probability that
	their hat is closed.
	A related argument is used
	in~\cite{MosselRoch:12}.}
\end{enumerate}
\begin{proof}[Proof of Proposition~\ref{prop:distinguishing}]
Proposition~\ref{prop:distinguishing}
follows immediately from
Claim~\ref{claim:prop1-claim3}.
\end{proof}
It remains to prove the claims.
\begin{proof}[Proof of Claim~\ref{claim:prop1-claim1} (Accuracy of ancestral reconstruction)]
	Let $Y$ be any subtree in the battery and let $y$
	be its root.
	To obtain a bound on the probability of
	erroneous ancestral
	reconstruction through Lemma~\ref{lem:ekps} (Appendix~\ref{section:appendix-1}),
	it suffices to bound the denominator in \eqref{eq:flow}
	for the $\ell$-completion $\lfloor Y \rfloor_\ell$.
	Choose a unit flow $\Psi$ such that
	the flow through each vertex
	on level $i\ell$ of $\lfloor Y \rfloor_\ell$
	splits evenly
	among its descendant vertices on level $(i+1)\ell$.
	Let $\ecal(\lfloor Y \rfloor_\ell)$ denote
	the edges of $\lfloor Y \rfloor_\ell$ and
	let
	$
	R_y(e) = \left(1 - \theta_e^2\right)
	\Theta_{y,x}^{-2}
	$ 
	where $\Theta_{\rt,y} = e^{-\dist_T(y,x)}$ and
	$\theta_e = e^{-\weight_e}$
	(as defined in~\eqref{eq:def-resistance} of
	Appendix~\ref{section:appendix-1}).
	Note that $0 \leq \left(1 - \theta_e^2\right) \leq 1$.
	Hence we have
	\begin{eqnarray*}
		\sum_{e = (x',x) \in \ecal(\lfloor Y \rfloor_\ell)} R_y(e) \Psi(e)^2
		&\leq&  \sum_{e = (x',x) \in \ecal(\lfloor Y \rfloor_\ell)}
		\Theta_{y,x}^{-2} \Psi(e)^2\\
		&\leq& \sum_{i=0}^{+\infty} \sum_{j=1}^\ell
		2^{i\ell + j} \left(e^{- (\ell i + j) g}\right)^{-2} \\
		&&\quad \times\left(
		\frac{1}{(2^\ell - \wp)^i \max\{1,2^{\ell-j} -\wp\}}\right)^2\\
		&\leq& \sum_{i=0}^{+\infty}
		2^{i\ell} e^{2 \ell i g} \left(\frac{1}{(2^\ell - \wp)^i}\right)^2\\
		&&\quad \times \left[\sum_{j=1}^\ell
		2^{j} e^{2 j g} \left(\frac{1}{\max\{1,2^{\ell-j} -\wp\}}\right)^2\right],
	\end{eqnarray*}
	where, in the second inequality, the quantity
	$\max\{1,2^{\ell-j} -\wp\}$ is a lower bound on the
	number of descendants on level $(i+1)\ell$
	of a vertex at graph distance $j$ below level
	$i\ell$.
	The term in square bracket on the last line
	is bounded by
	a positive constant $0 < K_{\ell,\wp,g} < +\infty$
	depending only on
	$\ell,\wp,g$.
	Recall that $g < g^* = \ln  \sqrt{2}$ and let $g' = \frac{g^* + g}{2}$. Choose $\ell$ large enough (depending only
	on $g$ and $\wp$) such that
	$$
	\frac{2^{\ell}}{(2^\ell - \wp)^{2}}
	\leq \frac{1}{e^{2 \ell g'}},
	$$
	which is possible because $g' < g^*$ and
	$e^{2 g^*} = 2$.
	Then
	\begin{eqnarray*}
		\sum_{e = (x',x) \in \ecal(\lfloor Y \rfloor_\ell)} R_y(e) \Psi(e)^2
		&\leq& K_{\ell,\wp,g} \sum_{i=0}^{+\infty} e^{2\ell i(g - g')}\\
		&=& \frac{K_{\ell,\wp,g}}{1 - e^{-2\ell (g' - g)}}\\
		&=& \frac{K_{\ell,\wp,g}}{1 - e^{-\ell (g^* - g)}} < +\infty.
	\end{eqnarray*}
	Hence by Lemma~\ref{lem:ekps} (Appendix~\ref{section:appendix-1})
	the probability of correct ancestral
	reconstruction is bounded away from
	$1/2$ from below. Let $\beta_{g,\wp}$, depending
	on $g$ and $\wp$ (and implicitly on $\ell$),
	such that
	\begin{equation}\label{eq:bias-def}
	\P_Y[\hs_y = \s_y] =: \frac{1 + e^{-\beta_Y}}{2} \geq \frac{1 + e^{-\beta_{g,\wp}}}{2},
	\end{equation}
	where the first equality is a definition.
\end{proof}
\begin{proof}[Proof of Claim~\ref{claim:prop1-claim2} (Separation of expectations)]
	Let $((y,z);(Y, Z))$ be a test pair in the battery
	with corresponding tree $T$ (equal to either $T^0$ or $T^\#$). Then,
	by the co-hanging requirement of the
	battery,
	the Markov property,
	and~\eqref{eq:bias},
	we have
	\begin{eqnarray}\label{eq:proximal1}
	-\ln \E_T[\hs_{y}\hs_{z}]
	&=& -\ln \left[
	\E_T[\E_T[\hs_{y}\hs_{z} \,|\,\s_y,\s_z]]
	\right]\nonumber\\
	&=& -\ln \left[
	\E_T[\E_T[\hs_{y}\,|\,\s_y]\E_T[\hs_{z} \,|\,\s_z]]
	\right]\nonumber\\
	&=& -\ln \left[
	\E_T[e^{-\beta_Y}\s_y e^{-\beta_Z}\s_{z}]
	\right]\nonumber\\
	&=&  \beta_Y + \beta_Z + \dist_T(y,z),
	\end{eqnarray}
	where $\beta_Y, \beta_Z \leq \beta_{g,\wp}$,
	as defined in~\eqref{eq:bias-def}.
	The last equality follows from
	\begin{eqnarray*}
		\E_T[\s_{y}\s_{z}]
		&=& \E_T[\E_T[\s_{y}\s_{z}\,|\,\s_y]]\\
		&=& \E_T[\s_{y}\E_T[\s_{z}\,|\,\s_y]]\\
		&=& \E_T[\s_{y} e^{-\dist_T(y,z)} \s_y]\\
		&=&  e^{-\dist_T(y,z)}\E_T[\s_{y}^2]\\
		&=& e^{-\dist_T(y,z)},
	\end{eqnarray*}
		where the third equality follows from
		the fact that $\P_T[\s_z = \s_y\,|\,\s_y] 
		= \frac{1 + e^{-\dist_T(y,z)}}{2}$,
		which can be deduced from Definition~\ref{def:mmt}.
	In the proximal case, we have further that
	\begin{eqnarray}\label{eq:proximal2}
	-\ln \E_T[\hs_{y}\hs_{z}]
	&=&  \beta_Y + \beta_Z + \dist_T(y,z)\nonumber\\
	&\leq& 2 \beta_{g,\wp} + g \Gamma\nonumber\\
	&=:& \chi_{g,\wp,\Gamma}.
	\end{eqnarray}
	
	Recall that the expected difference of
	$\hdcal$ under $T^0$ and $T^\#$
	is given by
	\begin{eqnarray*}
		\dcal^0 - \dcal^\#
		&=& \sum_{i=1}^I \sum_{j=1}^k
		\alpha_i \left(\E_{T^0}\left[\hs^j_{y^0_i}\hs^j_{z^0_i}\right]
		-\E_{T^\#}\left[\hs^j_{y^\#_i}\hs^j_{z^\#_i}\right].
		\right)
	\end{eqnarray*}
	Each term in the sum,
	whether it corresponds to a proximal-proximal,
	semi-proximal-proximal, or non-proximal-proximal case,
	is at least
	$$
	\dcal_\delta := e^{-\chi_{g,\wp,\Gamma}}(1 - e^{-\quantum}),
	$$
	by~\eqref{eq:def-alpha},~\eqref{eq:proximal1},~\eqref{eq:proximal2}, and
	the fact that
	\begin{equation}\label{eq:proximal3}
	|\dist_{T^0}(y^0_i,z^0_i) -
	\dist_{T^\#}(y^\#_i,z^\#_i)| \geq \quantum,
	\end{equation}
	by the evolutionary distance requirement of the battery.
	Hence,
	\begin{equation*} 
	\dcal^0 - \dcal^\#
	\geq \dcal_\delta k I.
	\end{equation*}
\end{proof}
\begin{proof}[Proof of Claim~\ref{claim:prop1-claim3} (Concentration)]
Consider $\hdcal$ on $T^0$. (The argument is
the same on $T^\#$.)
Let $\ical^0_{np}$ be the set of non-proximal
pairs in $T^0$ and
let $\ical^0_p$ be the set of pairs that are either
semi-proximal or proximal.
We also let
$$
\jcal^0_{np} = \left\{(i,j) : i \in \ical^0_{np}, j \in \{1,\ldots,k\}\right\},
$$
and
$$
\jcal^0_p = \left\{(i,j) : i \in \ical^0_p, j \in \{1,\ldots,k\}\right\}.
$$
Fix $0 < \eps < 1$ small (to be determined below).

We first illustrate our argument in the
easier case where the number of non-proximal pairs is small:
$$
|\ical^0_{np}| < \eps I.
$$
We define the following sum
\begin{eqnarray*}
	\tdcal = \sum_{(i,j) \in \jcal^0_p}
	\alpha_i\hs^j_{y^0_i}\hs^j_{z^0_i}
	- \left|\jcal^0_{np}\right|,
\end{eqnarray*}
that is, we set the terms in $\jcal^0_{np}$ to their
worst-case value, $-1$, which implies
$\hdcal \geq \tdcal$.
We claim that \textit{the remaining
terms in the sum are independent}. To prove this, 
we first make an observation. For all $(i,j) \in \jcal^0_p$,
the term $\hs^j_{y^0_i}\hs^j_{z^0_i}$
is independent of the state $\s^j_{x^0_i}$ at the
root $x^0_i$ of $\fcal^0_i$, where the latter is defined in~\eqref{eq:pair2}. This follows by the symmetry of the substitution process between the $+1$ and $-1$ states. To prove the independence claim above, we then proceed by generating the substitution process as follows, for each $j =1,\ldots,k$ independently. Define $\mathcal{X}$ to be the set of those roots $x^0_i$ such that $(i,j) \in \jcal^0_p$. 
\begin{enumerate}
	\item Let $\mathcal{H}$ be the set of those $i$ such that: (i) $(i,j) \in \jcal^0_p$ and (ii) the root $x^0_i$ does not have an ancestor in $T^0$ among $\mathcal{X}$. Pick the states $\s^j_{x^0_i}$, $i \in \mathcal{H}$, and the corresponding quantities $\hs^j_{y^0_i}\hs^j_{z^0_i}$, which depend only on the states within $\fcal^0_i$. By the Markov property and the condition on $\mathcal{H}$, the quantities $\hs^j_{y^0_i}\hs^j_{z^0_i}$, for $i \in \mathcal{H}$, are mutually independent. 
	
	\item Then let $\mathcal{H}'$ be the set of those $i \notin \mathcal{H}$ such that: (i) $(i,j) \in \jcal^0_p$ and (ii) the root $x^0_i$ does not have an ancestor in $\mathcal{X}-\{x^0_i\,:\, i\in \mathcal{H}\}$. Conditioned on the previously assigned states,
	pick the states $\s^j_{x^0_i}$, $i \in \mathcal{H}'$, and the corresponding quantities $\hs^j_{y^0_i}\hs^j_{z^0_i}$, which depend only on the states within $\fcal^0_i$. By the global requirement of the battery, the $\fcal^0_i$s in $\mathcal{H}'$ have empty edge intersection with the $\fcal^0_i$s in $\mathcal{H}$. Together with the observation above, it follows that the quantities $\hs^j_{y^0_i}\hs^j_{z^0_i}$, for $i \in \mathcal{H}'$, are independent of each other as well as of the quantities $\hs^j_{y^0_i}\hs^j_{z^0_i}$, for $i \in \mathcal{H}$.
	
	\item Add $\mathcal{H}'$ to $\mathcal{H}$, and proceed similarly to the previous step until all terms in $\jcal^0_p$ have been generated.
	
\end{enumerate}

The event
$$
A^c = \left\{
\hdcal - \frac{\dcal^0 + \dcal^\#}{2} \leq 0
\right\},
$$
implies the event
$$
\tdcal \leq \frac{\dcal^0 + \dcal^\#}{2},
$$
or, after rearranging,
\begin{equation}\label{eq:ac2}
\tdcal - \E_{T^0}[\tdcal]
\leq  - \frac{\dcal^0 - \dcal^\#}{2}
+ \{\dcal^0 - \E_{T^0}[\tdcal]\},
\end{equation}
which in turn implies
$$
\tdcal - \E_{T^0}[\tdcal]
\leq - \frac{\dcal_\delta}{2} k I
+ 2\eps k I,
$$
where we used the fact that the $\jcal^0_p$-terms
cancel out in the expression in curly brackets
in~\eqref{eq:ac2}.
Choose $\eps$ small enough
so that the RHS is less than $- \frac{\dcal_\delta}{3} k I$.
Then, by Lemma~\ref{lemma:azuma} (Appendix~\ref{section:appendix-3}),
\begin{eqnarray}
\P_{T^0}[A^c]
&\leq& 2 \exp\left(-\frac{\left(\frac{\dcal_\delta}{3} k I\right)^2}{2(2)^2 (kI)}\right)\nonumber\\
&=& \exp\left(- \Omega(k I)\right)\label{ac:2prime}.
\end{eqnarray}

\srevision{
		Consider now the case where
		$$
		|\ical^0_{np}| \geq \eps I.
		$$
		We show how to deal with the extra complication that non-proximal pairs have connecting
		paths that may intersect with other test subtrees, thereby creating unwanted dependencies.
}

		Let $((y,z);(Y,Z))$ be a non-proximal test pair in $T^0$
		and consider
		the $\gamma_t$-toppings $\lceil Y \rceil^{\gamma_t}$
		and $\lceil Z \rceil^{\gamma_t}$.
		Note that at least one of $\lceil Y \rceil^{\gamma_t}$
		and $\lceil Z \rceil^{\gamma_t}$
		has a hat of length $\gamma_t/2$ as otherwise
		$Y$ and $Z$ would be connected through the root at
		distance at most $\gamma_t$, contradicting
		the non-proximal assumption. We refer
		to the corresponding hat as {\it the} hat of the pair.
		If both toppings have long enough hats,
		choose the lowest one of the two so that the
		hat is necessarily part of the connecting path.
		The probability
		that all edges in this hat are open under the
		random cluster representation of the model
		\srevision{described in Lemma~\ref{lem:fk} (Appendix~\ref{section:appendix-2})},
		which we refer to as an {\it open hat},
		is at most $e^{-f \gamma_t/2}$.
		If at least one such edge is closed,
		in which case we say {\it the hat is  closed},
		$\hs_{y}$ and $\hs_{z}$ are independent.
		Hence, we have in the
		non-proximal case
		\begin{eqnarray}\label{eq:stretched}
		-\ln \E_T[\hs_{y}\hs_{z}] \geq
		-\ln \left\{(e^{-f\gamma_t/2})(1) + (1)(0)\right\}
		\geq f \gamma_t/2,
		\end{eqnarray}
		where the first term in curly brackets 
		accounts for the fact that, under an open hat, the correlation between the reconstructed states is at most $1$,
		while the second term accounts for the fact
		that, under a closed hat, the reconstructed 
		states are independent and therefore have 
		$0$ correlation.
		
		Let $[\jcal^0_{np}]_c$ be the random set
		corresponding to those pairs in
		$\jcal^0_{np}$ with a closed hat
		and
		let $[\jcal^0_{np}]_o$ be the random set
		corresponding to those pairs in
		$\jcal^0_{np}$ with an open hat.
		We consider the following sum
		\begin{eqnarray*}
			\tdcal = \sum_{(i,j) \in \jcal^0_p\cup [\jcal^0_{np}]_c}
			\alpha_i\hs^j_{y^0_i}\hs^j_{z^0_i}
			- \left|[\jcal^0_{np}]_o\right|,
		\end{eqnarray*}
		that is, we set the terms with open hats to their
		worst-case value $-1$ which implies
		$\hdcal \geq \tdcal$.
		We claim that the remaining
		terms in the sum are conditionally independent given
		$[\jcal^0_{np}]_c$.
		Indeed, considering only the
		proximal and semi-proximal test subtrees
		and their connecting paths as well
		as the non-proximal test subtrees whose hat
		is closed, we claim that the roots of the former and the
		hats of the latter form a separating set in the sense
		that any path
		between two of these subtrees must
		go through one of the roots or an entire hat.
		Indeed,
		a path between any two of these test subtrees
		must enter
		one of them from above. Moreover if one of the
		two subtrees is non-proximal but the path
		does not visit its entire hat,
		then the other test subtree must
		also be entered from above (as the path
		must deviate downwards from the hat) and its entire hat
		be visited if non-proximal (otherwise the two
		hats would intersect). Arguing as in the small number of non-proximal pairs, this implies mutual
		independence.

		By the argument above~\eqref{eq:stretched},
		\begin{equation}\label{eq:large-expec}
		\E_{T^0}[|[\jcal^0_{np}]_c|] \geq (1 - e^{-f \gamma_t/2})
		|\jcal^0_{np}|.
		\end{equation}
		Hence, by Lemma~\ref{lemma:azuma} (Appendix~\ref{section:appendix-3}), letting
		$0 < \eps' < 1 - e^{-f\gamma_t/2}$ (determined below)
		\begin{eqnarray}
		\P_{T^0}\left[|[\jcal^0_{np}]_c| <
		\E_{T^0}[|[\jcal^0_{np}]_c|] - \eps'|\jcal^0_{np}|\right]
		&\leq& 2 \exp\left(
		- \frac{(\eps' |\jcal^0_{np}|)^2}
		{2 (1)^2 |\jcal^0_{np}|}
		\right)\nonumber\\
		&=& \exp\left(- \Omega((\eps')^2|\jcal^0_{np}|)\right)\nonumber\\
		&=& \exp\left(- \Omega((\eps')^2 \eps k I)\right).\label{eq:large-inp-aux1}
		\end{eqnarray}
		On the event
		$$
		\ccal = \{|[\jcal^0_{np}]_c| \geq
		\E_{T^0}[|[\jcal^0_{np}]_c|]
		- \eps' |\jcal^0_{np}|\},
		$$		
		we have, using~\eqref{eq:large-expec},
		\begin{eqnarray}
		\left|[\jcal^0_{np}]_o\right|
		&=& \left|\jcal^0_{np}\right| - \left|[\jcal^0_{np}]_c\right|\nonumber\\
		&\leq& \left|\jcal^0_{np}\right| - \E_{T^0}[|[\jcal^0_{np}]_c|]
		+ \eps' |\jcal^0_{np}|\nonumber\\
		&\leq& (e^{-f \gamma_t/2} + \eps') |\jcal^0_{np}|\nonumber\\
		&\leq& (e^{-f \gamma_t/2} + \eps') kI\nonumber\\
		&\leq& - \frac{\dcal_\delta}{6} k I,\label{eq:large-inp-aux2}
		\end{eqnarray}		
		for $\eps'$ small enough
		and $\gamma_t$ large enough.
		\srevision{
		The rest of the argument follows similarly 
		to the small $|\ical^0_{np}|$ case
		by bounding
		\begin{eqnarray*}
			\P_{T^0}[A^c]
			&=& \P_{T^0}[A^c\,|\,\ccal^c] \P_{T^0}[\ccal^c]
			+ \P_{T^0}[A^c\,|\,\ccal] \P_{T^0}[\ccal] 
			\leq \P_{T^0}[\ccal^c] +
			\E_{T^0}[\P_{T^0}[A^c | | [\jcal^0_{np}]_c, \ccal ] | \ccal], 
		\end{eqnarray*}	
		using~\eqref{eq:large-inp-aux1}
		for the first term, and~\eqref{eq:large-inp-aux2}
		along with Claim~\ref{claim:prop1-claim2}
		for the second one.			
		%
	}
\end{proof}

\section{Homogeneous trees}
\label{sec:homo}
\label{sec:high-level-proof-hmg}

We first detail our techniques for constructing
batteries of tests on a special case:
homogeneous trees.
Formally, we define homogeneous phylogenies as follows.
\begin{definition}[Homogeneous phylogenies]\label{ex:homo}
	For an integer $h \geq 0$ and $n = 2^h$,
	we denote by
	$\hmgphy^{(h)}_{g}$
	the subset of
	$\phy^{(n)}_{f,g}[\quantum]$
	comprised of all
	$h$-level
	complete binary trees
	$\hmgt{h}_{\phi,\weight} = (\hmgv{h}, \hmge{h}; \phi; \weight)$
	where the edge weight function
	$\weight$ is identically $g$
	and $\phi$ may be any one-to-one labeling
	of the leaves.
	We denote by $\hmgrt{h}$ the natural root of
	$\hmgt{h}_{\phi,\weight}$.
	For $0\leq h'\leq h$, we let $\hmgll{h}{h'}$ be the
	vertices on level $h - h'$ (from the root). In particular,
	$\hmgll{h}{0} = \hmgl{h}$ denotes the leaves
	of the tree and $\hmgll{h}{h} = \{\hmgrt{h}\}$
	denotes the root.
\end{definition}
\noindent Fix $g, n = 2^h$
and let $\hmgphy = \hmgphy^{(h)}_{g}$
be the set of homogeneous phylogenies
with $h$ levels and branch lengths $g$.
Let $T^0 \in \hmgphy$
be the generating phylogeny
and denote by $\boldsymbol{\s}_X = (\s^i_X)_{i=1}^k$
a set of $k$ i.i.d.~samples from the corresponding
CFN model.
We first need an appropriate notion
of distance between homogeneous trees.
Note that tree
operations routinely used in phylogenetics, 
such as subtree-prune-regraft or
nearest-neighbor interchange (see e.g.~\cite{SempleSteel:03}),
may not result in homogeneous trees.
It will be more convenient to work with the following
definition. We say that two homogeneous trees
$T$ and $T'$ are equivalent, denoted by $T\sim T'$,
if $\forall a,b \in [n], \dist_T(a,b) = \dist_{T'}(a,b)$,
that is, if they agree as tree metrics.
(Recall that tree metrics are defined in Definition~\ref{def:metric}.)
\begin{definition}[Swap distance]
\label{def:swap}
We call a {\em swap} the operation
of choosing two (non-sibling) vertices $u$ and $v$
on the same level of a homogeneous tree
and exchanging the subtrees rooted at $u$
and $v$.
The {\em swap distance}  $\swap(T,T')$
between $T$ and $T'$ in $\hmgphy$
is the smallest number of
swaps needed to transform $T$ into $T'$
(up to $\sim$).
\end{definition}
\noindent Because a swap operation is invertible,
we have $\swap(T',T) = \swap(T,T')$.
By simply re-ordering the leaves, it holds
that $\swap(T,T') \leq n-1$.
We need a bound on the size of the neighborhood
around a tree.
Since for each swap operation, we choose one of
$2n-2$ vertices, then choose one of at most $n-2$
non-sibling vertices on the same level, we have:
\begin{claim}[Neighborhod size: Swap distance]
\label{claim:neighborhood-size-hmg}
Let $T$ be a phylogeny in $\hmgphy$.
The number of phylogenies at swap distance
$\Delta$ of $T$ is at most $(2n^{2})^\Delta$.
\end{claim}

In the following subsections, we prove
the existence of a sufficiently large
battery of distinguishing
tests.
\begin{proposition}[Existence of batteries]
\label{prop:existence-batteries-hmg}
Let $\wp = 1$, $\ell = \ell(g,\wp) \geq 2$ as
in Proposition~\ref{prop:distinguishing},
$
\Gamma = 2 \ell,
$
and
$\gamma_t \geq \Gamma$ and $C$
as
in Proposition~\ref{prop:distinguishing}.
For all $T^\# \neq T^0 \in \hmgphy$,
there exists a
$(\ell,\wp,\Gamma,\gamma_t,I)$-battery
$$
\{((y^0_i,z^0_i);(Y^0_i,Z^0_i))\}_{i=1}^I
\text{ (in $T^0$) and }
\{((y^\#_i,z^\#_i);(Y^\#_i,Z^\#_i))\}_{i=1}^I
\text{ (in $T^\#$)},
$$
with
$$
I \geq  \frac{\swap(T^0,T^\#)}{
C_\scal (1 + 2^{2\gamma_t + 2})
},
$$
where
$C_\scal > 0$ is a constant (defined in Claim~\ref{claim:swaps-r}
below).
\end{proposition}
\noindent The formal proof of this
proposition can be found in Section~\ref{section:hmg-app-battery}.

From 
Propositions~\ref{prop:distinguishing}
and~\ref{prop:existence-batteries-hmg} as well as Claim~\ref{claim:neighborhood-size-hmg}, we obtain our main theorem in this special case.
\begin{theorem}[Sequence-length requirement of ML: Homogeneous trees]
\label{thm:main-hmg}
For all $\delta > 0$, there exists $\kappa > 0$
depending on $\delta$ and $g$
such that the following holds.
For all $h \geq 2$, $n = 2^h$ and generating phylogeny
$T^0 \in \hmgphy^{(h)}_g$,
if $\boldsymbol{\s}_X = (\s^i_X)_{i=1}^k$ is
a set of $k = \kappa \log n$ i.i.d.~samples from the corresponding
CFN model,
then
the probability that MLE
fails to return $T^0$ is at most $\delta$.
\end{theorem}
\begin{proof}[Proof of Theorem~\ref{thm:main-hmg}]
	For $T^\# \neq T^0$ in $\hmgphy^{(h)}_g$,
	let $M_{T^\#}$ be the event that the {\it MLE} prefers $T^\#$
	over $T^0$ (including a tie),
	that is, the set of $\boldsymbol{\s}_X = (\s^{i}_X)_{i=1}^k$
	such that
	$
	\L_{T^\#}(\boldsymbol{\s}_X)
	\leq \L_{T^0}(\boldsymbol{\s}_X)
	$.
	Combining Propositions~\ref{prop:distinguishing}
	and~\ref{prop:existence-batteries-hmg},
	for all $T^\# \neq T^0 \in \hmgphy^{(h)}_g$,
	there exists an event $A_{T^\#}$
	such that
	\begin{equation}\label{eq:distinguishing-event-hmg}
		\max\{\P_{T^\#}[A_{T^\#}^c], \P_{T^0}[A_{T^\#}]\}
		\leq e^{-C_1 k \swap(T^0,T^\#)},
	\end{equation}
	where $C_1$ depends only on $g$.
	Then, by a union bound,~\eqref{eq:hypothesis-testing},~\eqref{eq:distinguishing-event-hmg},
	and Claim~\ref{claim:neighborhood-size-hmg},
	\begin{eqnarray*}
		\P_{T^0}[\exists T^\# \neq T^0,\ M_{T^\#}]
		&\leq& \sum_{T^\# \neq T^0}\P_{T^0}[M_{T^\#}]\nonumber\\
		&\leq&  \sum_{T^\# \neq T^0} [\P_{T^0}[A_{T^\#}] + \P_{T^\#}[A_{T^\#}^c]]\\
		&\leq& \sum_{\Delta = 1}^{n-2}
		(2 n^{2})^\Delta [2 e^{-C_1 k \Delta}]\\
		&\leq& \sum_{\Delta = 1}^{+\infty}
		e^{-(C_1 \kappa \log n - 2 \log n - \log 2)\Delta}\\
		&\leq& \frac{e^{-(C_1 \kappa \log n - 2 \log n  - \log 2)}}
		{1 - e^{-(C_1 \kappa \log n - 2 \log n  - \log 2)}}\\
		&\leq&  \frac{1}
		{e^{(C_1 \kappa - 2)\log n  - \log 2} - 1}\\
		&\leq& \delta,		
	\end{eqnarray*}
	for $\kappa$ large enough, depending only on
	$g$ and $\delta$,
	for all $n \geq 2$.
\end{proof}

\subsection{Finding matching subtrees}
\label{sec:finding-matching-hmg}

We now describe a procedure to construct
a battery of distinguishing tests on homogeneous trees.
To be clear,
the procedure is not carried out on
data. It takes as input the true (unknown) generating phylogeny
$T^0$ and an alternative tree $T^\# \neq T^0$.
It merely serves
to prove the existence
of a distinguishing statistic that, in turn,
implies a bound on the failure probability
of maximum likelihood as detailed in
the proof of Theorem~\ref{thm:main-hmg}.
In essence, the procedure attempts to build
maximal matching subtrees between $T^0$ and
$T^\#$ and pair them up appropriately to construct
distinguishing tests. \begin{definition}[$\ell$-vertices]
For a fixed positive integer $\ell$,
we call {\em $\ell$-vertices}
those vertices in $T^0$ whose graph distance from the root
is a multiple of $\ell$.
The sets of all $\ell$-vertices at the same distance
from the root are called {\em $\ell$-levels}.
For an $\ell$-vertex $x$, its descendant $\ell$-vertices
on the next $\ell$-level (that is, farther from the root)
are called the {\em $\ell$-children of $x$},
which we also refer to as a family of $\ell$-siblings.
\end{definition}
\noindent Let $\wp = 1$ and $\ell = \ell(g,\wp) \geq 2$
be as in Proposition~\ref{prop:distinguishing}.
Assume for simplicity that the total number of levels
$h$ in $T^0$ is a multiple of $\ell$. Extending the analysis
to general $h$ is straigthforward.

\paragraph{Procedure}
Our goal
is to color each
$\ell$-vertex $x$ of $T^0$ with the following
intended meaning:
\begin{itemize}
\item Green $\g$: indicating a matching subtree
rooted at $x$
that can be used to reconstruct ancestral states
reliably on both $T^0$ and $T^\#$ using the same function
of the leaf states.

\item Red $\r$: indicating the presence among
the $\ell$-children of $x$
of a pair of matching subtrees
that can be used in a distinguishing test
as their pairwise distance differs in $T^0$
and $T^\#$.

\item Yellow $\y$: none of the above.
\end{itemize}
We call $\g$-vertices (respectively $\g$-children)
those $\ell$-vertices (respectively $\ell$-children)
that are colored $\g$, and similarly for the
other colors.
Before describing the coloring procedure in
details, we need a definition.
\begin{definition}[$\g$-cluster]
Let
$x$ be a $\g$-vertex.
Assume that each $\ell$-vertex below $x$ in $T^0$
has been colored $\g$, $\r$,
or $\y$ and that the
leaves have been colored $\g$. The {\em $\g$-cluster}
rooted at $x$ is the restricted subtree of
$T^0$ containing
all vertices and edges (not necessarily
$\ell$-vertices) lying
on a path between $x$
and a leaf below $x$ that traverses only $\ell$-vertices
colored $\g$.
\end{definition}
We now describe the coloring procedure.
\begin{enumerate}
\item Initialization
\begin{enumerate}

\item  All leaves of $T^0$ are colored $\g$.

\end{enumerate}

\item For each $\ell$-vertex $x$ in the
$\ell$-level furthest from the
root that has yet to be colored, do:
\begin{enumerate}
\item Vertex $x$ is colored  $\g$ if:
\begin{itemize}
\item at most one
of its $\ell$-children is non-$\g$ and;
\item
the resulting $\g$-cluster rooted at $x$
and the corresponding restricted
subtree in $T^\#$, that is, the subtree of $T^\#$
restricted to the same leaf set, are matching.
\end{itemize}

\item Else, vertex $x$ is colored  $\r$
if:
\begin{itemize}
\item at most one
of its $\ell$-children is non-$\g$;
\item
but, {\bf if} $x$ were colored $\g$,
the resulting $\g$-cluster rooted at $x$
and the corresponding restricted
subtree in $T^\#$ would {\bf not} be matching.
\end{itemize}

\item Else, vertex $x$ is colored $\y$. 
\end{enumerate}

\end{enumerate}
In particular observe that, if $x$ is colored
$\y$,
at least two of its $\ell$-children are non-$\g$.

As explained above, we are interested
in $\r$-vertices because tests can be constructed from them.
We prove that the number of $\r$-vertices
scales linearly in the swap distance. More precisely, we show that
$$
\#\r \geq 2^{-\ell - 2}\, \swap\left(T^0,T^\#\right) .
$$

\paragraph{Relating combinatorial distance and
	the number of matching subtrees}
We relate the swap distance between $T^0$ and
$T^\#$ to the number of $\r$-vertices
in the procedure above.
Let $\#\g$ be the number of $\g$-vertices
in $T^0$ in the construction,
and similarly for the other colors.
For an $\ell$-vertex $x$ in $T^0$,
we let $T^0_x$ be the subtree of $T^0$
rooted at $x$ and we let $\vcal_\ell(T^0_x)$
be the set of $\ell$-vertices in $T^0_x$.
Recall from Definition~\ref{def:metric} that
we denote by $\gdist_{T^0}$ the graph
distance on $T^0$. We first bound the
number of $\y$-vertices.
\begin{claim}[Bounding the number of yellow vertices]
	\label{claim:bounding-yellow-hmg}
	We have
	$$
	\#\y \leq \#\r.
	$$
\end{claim}
\begin{proof}
	From our construction, each $\y$-vertex in $T^0$ has at least two non-$\g$-children.
	Hence, intuitively,
	one can think of the $\y$-vertices as forming the internal vertices of
	a forest of multifurcating trees whose leaves are $\r$-vertices.
	The inequality follows.
	
	Formally, if $x$ is a $\y$-vertex, from the observation
	above we have
	\begin{equation}\label{eq:bounding-yellow-1-hmg}
	\sum_{y \in \vcal_\ell(T^0_x)}
	2^{-\frac{\gdist_{T^0}(x,y)}{\ell}} \ind\{\text{$y$ is a $\r$-vertex}
	\} \geq 1,
	\end{equation}
	by induction on the $\ell$-levels starting with
	the level farthest away from the root. Similarly
	if $y$ is an $\r$-vertex,
	we have
	\begin{equation}\label{eq:bounding-yellow-2-hmg}
	\sum_{x : y \in \vcal_\ell(T^0_x)}
	2^{-\frac{\gdist_{T^0}(x,y)}{\ell}} \ind\{\text{$x$ is a $\y$-vertex}\} < 1,
	\end{equation}
	where the inequality follows from the fact that
	the sum is over a path from $x$ to the root of $T^0$.
	Summing~\eqref{eq:bounding-yellow-1-hmg} over
	$\y$-vertices $x$ and~\eqref{eq:bounding-yellow-2-hmg} over
	$\r$-vertices $y$ gives the same
	quantity on the LHS, so that the RHS gives the inequality. 
\end{proof}
\noindent We can now relate the swap
distance to the output of the procedure.
\begin{claim}[Relating swaps and $\#\r$]
	\label{claim:swaps-r}
	We have
	$$
	\swap\left(T^0,T^\#\right) \leq C_\scal \#\r,
	$$
	where $C_\scal = 2^{\ell + 2}$.
\end{claim}
\begin{proof}
	Pick a lowest non-$\g$-vertex $u$ in $T^0$. Being lowest, all
	$\ell$-children of $u$ must be colored $\g$. In fact,
	all $\ell$-vertices on the level below $u$ must be colored
	$\g$. Make
	$u$ a $\g$-vertex by transforming the subtree below
	$u$ in $T^\#$ to match the corresponding subtree
	in $T^0$. This takes at most $2^{\ell+1}$ swaps.
	
	Repeat until $T^0$ and $T^\#$ match.
	The inequality then follows from Claim~\ref{claim:bounding-yellow-hmg}.
\end{proof}

\subsection{Constructing a battery of tests}
\label{section:constructing}

We now construct a battery of tests from the $\r$-vertices.
The basic idea is that each $\r$-vertex has
two $\g$-children which satisfy many of the requirements
of a battery and therefore can potentially be used as a test pair.
In particular, they are the roots of dense subtrees that are matching with their corresponding restricted subtrees
in $T^\#$, but their evolutionary distance differs
in $T^0$ and $T^\#$. Note that we also have a
number of $\r$-vertices
that scales linearly in the swap distance
by Claim~\ref{claim:swaps-r}.
However one issue to address
is the global requirement of the battery.
In words, we need to ensure that the test pairs
{\em do not intersect}. We achieve this by
sparsifying the battery. A similar argument
was employed in~\cite{MosselRoch:12}.

In this section, $T^0$ and $T^\#$ are fixed. To simplify
notation, we let $\Delta = \swap\left(T^0,T^\#\right)$.
Fix $\wp = 1$. Choose $\ell = \ell(g,\wp) \geq 2$ as
in Proposition~\ref{prop:distinguishing}.
Then take
$
\Gamma = 2\ell,
$
and set
$\gamma_t \geq \Gamma$ and $C$ as
in Proposition~\ref{prop:distinguishing}.
In the rest of this subsection, we build a
$(\ell,\wp,\Gamma,\gamma_t,I)$-battery
$$
\{((y^0_i,z^0_i);(Y^0_i,Z^0_i))\}_{i=1}^I
\text{ (in $T^0$)
and }
\{((y^\#_i,z^\#_i);(Y^\#_i,Z^\#_i))\}_{i=1}^I
\text{ (in $T^\#$)}
$$
with
corresponding $\alpha_i$s.
We number the $\r$-vertices $i=1,\ldots,I'$
and we build one test panel for each
$\r$-vertex. Here $I' \geq I$ as we will later
need to reject some of the test panels to avoid unwanted
correlations.

\paragraph{Co-hanging pairs in $T^\#$.}
We first construct test panels that satisfy
the pair and cluster requirements of the battery.
Let $x^0_i$ be an $\r$-vertex in $T^0$
and let $x^\#_i$ be the corresponding vertex
in $T^\#$. Because $x^0_i$ is colored
$\r$, by construction it has at least $2^\ell - 1$
$\g$-children, but its $\g$-children are ``connected in
different ways'' in $T^\#$.
In particular, at least one pair of $\g$-children
$(y^0_i,z^0_i)$
must be at a different evolutionary distance
in $T^0$ than the corresponding pair $(y^\#_i,z^\#_i)$
in $T^\#$ (see the remark after Definition~\ref{def:metric}). We use these pairs
as our test panel. 
\begin{claim}[Test panels]
\label{claim:cohanging-hmg}
For each $\r$-vertex $x^0_i$
we can find
a test pair of $\g$-children $(y^0_i,z^0_i)$ of $x^0_i$,
with corresponding test pair $(y^\#_i,z^\#_i)$ in $T^\#$,
such that the test panel satisfy the cluster and
pair requirements of a battery.
\end{claim}
\begin{proof} 
	By construction, the test subtrees, that is,
	the $\g$-clusters rooted at the test vertices, are
	$(\ell,1)$-dense.
	The test subtrees are also matching, co-hanging and their
	roots are at different evolutionary distances in $T^0$
	and $T^\#$.
	Finally, the test pair in $T^0$ is proximal as
	$$
	\gdist_{T^0}(y^0_i, z^0_i) \leq 2\ell \leq \Gamma.
	$$
\end{proof}

\paragraph{Sparsification in $T^\#$.}
It remains to satisfy the global requirements of the
battery. By construction the test subtrees
are non-intersecting in both $T^0$ and $T^\#$
(see the proof of Claim~\ref{claim:sparsification-hmg}).
However we
must also ensure that proximal/semi-proximal connecting paths
and non-proximal hats do not intersect with each other
or with test subtrees from other test panels.
By construction, this
is automatically satisfied in $T^0$ where all test pairs
are proximal.
To satisfy this requirement in $T^\#$,
we make the collection of test pairs sparser
by rejecting a fraction of them.
\begin{claim}[Sparsification in $T^\#$]
\label{claim:sparsification-hmg}
Assume $\ell \geq 2$.
Let $\hcal' = \{(y^0_i,z^0_i); (y^\#_i,z^\#_i)\}_{i=1}^{I'}$
be the test panels
constructed in Claim~\ref{claim:cohanging-hmg}.
We can find a subset $\hcal \subseteq \hcal'$
of size
$$
|\hcal| = I \geq \frac{1}{1 + 2^{2\gamma_t + 2}} I' \geq \frac{\Delta}{C_\scal (1 + 2^{2\gamma_t + 2})}
$$
such that the test panels in $\hcal$
satisfy all global requirements of a battery.
\end{claim}
\begin{proof} 
	Let $\{(Y^0_i,Z^0_i); (Y^\#_i,Z^\#_i)\}_{i=1}^{I'}$
	be the test subtrees corresponding to $\hcal'$.
	Let $W^0_1$ and $W^0_2$ be two test subtrees 
	in $T^0$ (not necessarily from the same test pair) and let  $W^\#_1$ and $W^\#_2$ be their matching subtrees in $T^\#$. We argue that these subtrees are non-intersecting in both $T^0$ and $T^\#$. We start with $T^0$. By construction (see Claim~\ref{claim:cohanging-hmg}), $W^0_1$ is a maximal $\g$-cluster: its root $w^0_1$ is a 
	$\g$-vertex whose parent $\ell$-vertex is colored $\r$; all $\g$-children of $w^0_1$ are in $W^0_1$, as well as all of their $\g$-children and so forth. The same goes for $W^0_2$, whose root we denote by $w^0_2$. If neither $w^0_1$ nor $w^0_2$ is a descendant
	of the other, then $W^0_1$ and $W^0_2$ are necessarily non-intersecting. Assume instead, w.l.o.g., that
	$w^0_2$ is a descendant of $w^0_1$. Because the
	parent $\ell$-vertex of $w^0_2$ is colored $\r$,
	then by construction $w^0_2$ and all of its decendants (including the subtree $W^0_2$) cannot
	be in $W^0_1$. 
	Note in particular that $W^0_1$ and $W^0_2$
	share no leaf.
	We move on to $T^\#$. Because 
	$T^\#$ is in $\hmgphy$, the subtrees $W^\#_1$ and
	$W^\#_2$ are isomorphic as graphs to $W^0_1$ and
	$W^0_2$. In particular, their structure is the same
	as the one described above. Let $w^\#_1$ and
	$w^\#_2$ be the vertices corresponding to
	$w^0_1$ and $w^0_2$ in $T^\#$. Again,
	if neither $w^\#_1$ nor $w^\#_2$ is a descendant
	of the other one, then $W^\#_1$ and $W^\#_2$ are non-intersecting. Assume instead, w.l.o.g., that
	$w^\#_2$ is a descendant of $w^\#_1$. If $W^\#_1$
	and $W^\#_2$ share a vertex, say $z$, then the parent $\ell$-vertex of $z$, say $\tilde{z}$, is also shared because of the structure
	of $W^\#_1$ and $W^\#_2$. But then $W^\#_1$ and $W^\#_2$ contain at least $2^\ell-1$ $\g$-children
	of $\tilde{z}$---so they must have at least one such $\g$-child in common (recall that $\ell \geq 2$). The same holds for the $\g$-children of these common $\g$-children, and so on. As a result, $W^\#_1$ and $W^\#_2$ must share at least one leaf, which contradicts the fact that $W^0_1$ and $W^0_2$ 
	(which have the same leaf sets as $W^\#_1$ and $W^\#_2$) share no leaf.  
		
	As discussed above, it remains to appropriately sparsify the set $\hcal'$ of test pairs. We proceed as follows.
	Start with test panel $((y^0_1,z^0_1); (y^\#_1,z^\#_1))$.
	Remove
	from $\hcal'$ all test panels $i \neq 1$ such that
	\begin{equation}\label{eq:definition-hcal}
	\min\{\gdist_{T^\#}(v,w)\ :\  v \in \{y^\#_1,z^\#_1\},
	w \in \vcal(Y^\#_i)\cup\vcal(Z^\#_i)\}
	\leq 2\gamma_t.
	\end{equation}
	Because there are at most $2\cdot 2^{2\gamma_t + 1}$
	vertices in $T^\#$ satisfying the above condition and that
	the test subtrees are non-overlapping in $T^\#$, we remove
	at most $2^{2\gamma_t + 2}$ test panels from
	$\hcal'$.
	
	Let $i$ be the smallest index remaining in $\hcal'$.
	Proceed as above and then repeat until
	all indices in $\hcal'$ have been selected or rejected.
	
	At the end of the procedure, there are at least
	$$
	\frac{1}{1 + 2^{2\gamma_t + 2}} I'
	$$
	test panels remaining, the set of which we denote by $\hcal$.
	Recalling that $\gamma_t \geq \Gamma$,
	note that, in $\hcal$, the connecting paths of proximal/semi-proximal
	pairs
	and the hats of non-proximal pairs
	cannot intersect with each other or with any of the
	test subtree rooted at test vertices in $\hcal$
	by~\eqref{eq:definition-hcal}.
\end{proof}

\subsection{Proof of Proposition~\ref{prop:existence-batteries-hmg}}
\label{section:hmg-app-battery}

It remains to prove
Proposition~\ref{prop:existence-batteries-hmg}.
Recall that $\wp = 1$, $\ell = \ell(g,\wp)$ is chosen as
in Proposition~\ref{prop:distinguishing},
$\Gamma = 2 \ell$,
and
$\gamma_t \geq \Gamma$ and $C$
are also chosen as
in Proposition~\ref{prop:distinguishing}.

By Claim~\ref{claim:swaps-r}, 
the number of $\r$-vertices is
at least $\frac{1}{C_\scal}\, \swap\left(T^0,T^\#\right)$.
By Claim~\ref{claim:cohanging-hmg}, 
for each $\r$-vertex, we can construct 
a test panel satisfying the pair and
cluster requirements of the battery.
By Claim~\ref{claim:sparsification-hmg},
we can further choose a fraction $\frac{1}{1 + 2^{2\gamma_t + 2}}$ of these test panels that also
satisfy the global requirement of the battery.
To sum up, we have built a
	$(\ell,\wp,\Gamma,\gamma_t,I)$-battery
	$$
	\{((y^0_i,z^0_i);(Y^0_i,Z^0_i))\}_{i=1}^I
	\text{ (in $T^0$) and }
	\{((y^\#_i,z^\#_i);(Y^\#_i,Z^\#_i))\}_{i=1}^I
	\text{ (in $T^\#$)}
	$$
	with
	$$
	I \geq  \frac{\Delta}{C_\scal (1 + 2^{2\gamma_t + 2})}.
	$$
That concludes the proof of Proposition~\ref{prop:existence-batteries-hmg}.

\section{General trees}
\label{sec:general}

We now prove our main result in the
case of general trees. Once again,
we use the tests introduced in
Section~\ref{sec:distinguishing-tests}.
We also use a procedure similar to that 
in the homogeneous case to construct
dense subtrees shared
by $T^\#$ and $T^0$. However, 
as described in the next subsections,
a number
of new issues arise, mainly
the possibility of overlapping
subtrees and non-co-hanging pairs.
Fix $f,g < g^*,\quantum$
and let $\phy = \phy^{(n)}_{f,g}[\quantum]$.

\subsection{Blow-up distance}

We first need an appropriate notion
of distance between general trees. Although standard
definitions exist~\cite{SempleSteel:03}, the following definition
(related to tree bisection and reconnection~\cite{AllenSteel:01})
will be particularly convenient for
our purposes.
\begin{definition}[Blow-up distance]
	\label{def:blowup}
	A {\em $B$-blowup operation} on a phylogeny
	consists in two steps: 
	\begin{itemize}
		\item Remove a subset
		of $B$ edges. The
		non-leaf, isolated vertices resulting from this 
		first step are also removed.
		
		\item Add $B$ new weigthed edges to form a
		new phylogeny with the same leaf set. 
	\end{itemize}
	The {\em blowup distance}
	$\blowup(T,T')$ between phylogenies $T$ and $T'$
	is defined as the smallest $B$
	such that there is a $B$-blowup operation
	transforming $T$ into $T'$ up to isomorphism. 
\end{definition}
\noindent 
Because the blowup
operation is invertible, we have
$\blowup(T,T') = \blowup(T',T)$.
\srevision{Observe that the blow-up distance is
	a metric.}
We will need a bound on the size of the neighborhood
around a tree.
\begin{claim}[Neighborhood size: Blowup distance]
	\label{claim:neighborhood-size}
	Let $T$ be a phylogeny in $\phy$.
	\srevision{The number of phylogenies that can be obtained from $T$ by a $\Delta$-blowup operation
	is at most $(12 g\invquantum n^{2})^\Delta$.}
\end{claim}
\begin{proof}
	There are $2n-3$ edges in $T$ so there are
	at most $(2n-3)^\Delta$ choices for the first
	step of the blowup operation.
	
	\srevision{For the second step, we add edges one by one. The weight
	of each edge can take at most $g\invquantum$ values.
	Each new edge must further be incident with a vertex existing at the end of the first step
	or adjacent to a newly added edge. Observe that 
	the edge removal
	in the first step produces at most $2\Delta$ vertices
	which can be used in the second step to attach a new edge.
	Moreover, each edge addition produces at most one
	new vertex to which subsequent edges can be attached.
	Since we add a total of $\Delta$ edges, there are at any stage
	at most $3\Delta$ choices for an attachment. 
	That is, there are at most $(3\Delta g\invquantum)^\Delta$
	choices for the second step of the operation.}
	
	Since clearly the blowup distance is
	$\leq 2n-3$, there are overall at most
	$$
	(2n-3)^\Delta (3\Delta g\invquantum)^\Delta
	\leq (3 g\invquantum (2n-3)^2)^\Delta
	\leq (12 g\invquantum n^{2})^\Delta,
	$$
	phylogenies that can be produced with a
	$\Delta$-blowup operation.
\end{proof}

\subsection{Main steps of the proof}
\label{sec:high-level-proof}

In the following subsections, we prove
the existence of a sufficientlty large
battery of distinguishing
tests.
\begin{proposition}[Existence of batteries]
	\label{prop:existence-batteries}
	Let $\wp = 5$, $\ell = \ell(g,\wp)$ as
	in Proposition~\ref{prop:distinguishing},
	$$
	\Gamma = \max\left\{(6 + 2\invquantum g) \ell,
	6 g \invquantum \log_2 \left(\frac{8}{1 - 1/\sqrt{2}}\right) + 2\ell g \invquantum + 4\right\},
	$$
	and
	$\gamma_t \geq \Gamma$, a multiple of $\ell$,
	and $C$
	as
	in Proposition~\ref{prop:distinguishing}.
	For all $T^\# \neq T^0 \in \phy$,
	there exists a
	$(\ell,\wp,\Gamma,\gamma_t,I)$-battery
	$$
	\{((y^0_i,z^0_i);(Y^0_i,Z^0_i))\}_{i=1}^I
	\text{ (in $T^0$) and }
	\{((y^\#_i,z^\#_i);(Y^\#_i,Z^\#_i))\}_{i=1}^I
	\text{ (in $T^\#$)},
	$$
	with
	$$
	I \geq  \frac{\blowup(T^0,T^\#)}{
		20 C_\ocal (1 + 2^{6\gamma_t + C_w + 3}g\invquantum)
	},
	$$
	where
	$$
	C_w = 3 g \invquantum \log_2 \left(\frac{8}{1 - 1/\sqrt{2}}\right) + \ell g \invquantum + 2,
	$$
	and
	$C_\ocal$ is a constant (defined in Claim~\ref{claim:blowup-overlap}
	below).
\end{proposition}
\begin{proof}
	This follows from
	Propositions~\ref{prop:battery-many-r}
	and~\ref{prop:battery-large-overlap}
	below.
\end{proof}
\noindent The choice of $\Gamma$ above will be justified in Claim~\ref{claim:cohanging-case1}
and in~\eqref{eq:gamma-requirement}.

From
Propositions~\ref{prop:distinguishing} 
and~\ref{prop:existence-batteries}
we obtain
of our main bound in the general case.
\begin{theorem}[Sequence-length requirement of ML: General trees]
	\label{thm:log-bound}
	For all $\delta > 0$, there exists $\kappa > 0$
	depending on $\delta$, $g$ and $\invquantum$
	such that the following holds.
	For all $n \geq 2$ and generating phylogeny
	$T^0 \in \phy^{(n)}_{f,g}[\quantum]$,
	if $\boldsymbol{\s}_X = (\s^i_X)_{i=1}^k$ is
	a set of $k = \kappa \log n$ i.i.d.~samples from the corresponding
	CFN model,
	then
	the probability that MLE
	fails to return $T^0$ is at most $\delta$.
\end{theorem}
\begin{proof}
	For $T^\# \neq T^0$,
	let $M_{T^\#}$ be the event that the {\it MLE} prefers $T^\#$
	over $T^0$ (including a tie),
	that is, the set of $\boldsymbol{\s}_X = (\s^{i}_X)_{i=1}^k$
	such that
	$
	\L_{T^\#}(\boldsymbol{\s}_X)
	\leq \L_{T^0}(\boldsymbol{\s}_X)
	$.
	By Propositions~\ref{prop:distinguishing}
	and~\ref{prop:existence-batteries},
	for all $T^\# \neq T^0 \in \phy^{(n)}_{f,g}[\quantum]$,
	there exists an event $A_{T^\#}$
	such that
	\begin{equation}\label{eq:distinguishing-event-general}
	\max\{\P_{T^\#}[A_{T^\#}^c], \P_{T^0}[A_{T^\#}]\}
	\leq e^{-C_1 k \blowup(T^0,T^\#)},
	\end{equation}
	where $C_1$ depends only on $g$ and $\invquantum$.
	Then, by a union bound,~\eqref{eq:hypothesis-testing},~\eqref{eq:distinguishing-event-general} and Claim~\ref{claim:neighborhood-size}, 
	arguing as in the proof of Theorem~\ref{thm:main-hmg}
	\begin{equation*}
	\P_{T^0}[\exists T^\# \neq T^0,\ M_{T^\#}]
		\leq \delta,
	\end{equation*}
	for $\kappa$ large enough, depending only on
	$g$, $\invquantum$ and $\delta$,
	for all $n \geq 2$.
	(This also proves Lemma~\ref{lemma:tv-bl}
	in Section~\ref{section:preliminaries}.)
\end{proof}

Finally:
\begin{proof}[Proof of Theorem~\ref{thm:main}]
	Combining Theorem~\ref{thm:log-bound}
	and the polynomial bound from
	Section~\ref{section:overview}
	immediately
	gives Theorem~\ref{thm:main}.
\end{proof}

\subsection{Finding matching subtrees}
\label{sec:finding-matching}

We now describe our procedure to construct
a battery of distinguishing tests for general trees. 
As in the homogeneous case, 
the procedure attempts to build dense,
maximal subtrees shared by $T^0$ and
$T^\#$. These subtrees are paired up 
appropriately to construct
distinguishing tests. For general trees, however, care must be taken
to deal with possible ``overlaps'' in $T^0$ and $T^\#$ (see Definition~\ref{def:overlap}).
Such overlaps produce unwanted dependencies
between the test pairs.

As a result,
we proceed in two stages: 
\begin{enumerate}
	\item First, similarly to the homogeneous case,
	pairs of matching subtrees are constructed.
	
	\item Second, if the overlap between the tests is too large,
	new distinguishing tests are constructed along
	the ``boundary of the overlap.''
\end{enumerate}
The first stage
is described below. Analysis of the size
of the overlap
is presented in Section~\ref{section:overlapSize}.
The more delicate second stage is described
in Sections~\ref{section:manyr} and~\ref{section:largeOverlap}.

We root $T^0$
arbitrarily.
\begin{definition}[$\ell$-vertices]
	For a fixed positive integer $\ell$,
	we call {\em $\ell$-vertices}
	those vertices in $T^0$ whose graph distance from the root
	is a multiple of $\ell$.
	The sets of all $\ell$-vertices at the same distance
	from the root are called {\em $\ell$-levels}.
	For an $\ell$-vertex $x$, its descendant $\ell$-vertices
	on the next $\ell$-level (that is, farther from the root)
	are called the {\em $\ell$-children of $x$}.
	We also refer to the $\ell$-children of $x$ as a family of {\em $\ell$-siblings}.
	By convention, the leaves
	of $T^0$ are also considered $\ell$-vertices irrespective
	of their distance from the root. They belong to the
	$\ell$-level immediately below them.
\end{definition}

Our goal
is to color each
$\ell$-vertex $x$ with the following
interpretation:
\begin{itemize}
	\item Green $\g$: indicating a matching subtree
	rooted at $x$
	that can be used to reconstruct ancestral states
	on $T^0$ and $T^\#$ using the same function
	of the leaf states.
	
	\item Red $\r$: indicating the presence among
	the $\ell$-children of $x$
	of a pair of matching subtrees
	that can be used as a distinguishing test
	because the distance between their roots differs in $T^0$
	and $T^\#$.
	
	
	\item Yellow $\y$: none of the above.
\end{itemize}
As before, we call $\g$-vertices (respectively $\g$-children)
those $\ell$-vertices (respectively $\ell$-children)
that are colored $\g$, and similarly for the
other colors.
Before describing the coloring procedure in
details, we need a definition.
\begin{definition}[$\g$-cluster]
	Let
	$x$ be a $\g$-vertex.
	Assume that each $\ell$-vertex below $x$ in $T^0$
	has been colored $\g$, $\r$,
	or $\y$ and that the
	leaves have been colored $\g$. The {\em $\g$-cluster}
	rooted at $x$ is the restricted subtree of
	$T^0$ containing
	all vertices and edges (not necessarily
	$\ell$-vertices) satisfying the following property: they lie
	on a path between $x$
	and a leaf below $x$ that traverses only $\ell$-vertices
	colored $\g$.
\end{definition}
We now describe the coloring procedure.
Below, when counting the $\ell$-children of an $\ell$-vertex
$x$ with a specified property, each leaf among the
$\ell$-children of $x$ counts as
$2^{\ell-d}$ vertices if $d$ is the
graph distance between $x$ and that leaf.
\begin{enumerate}
	\item Initialization
	\begin{enumerate}
		
		\item Root $T^0$ at an arbitrary vertex.
		(Note that $T^\#$ remains unrooted for this part of the
		proof where we are concerned with metric-matching
		as defined in Definition~\ref{def:matching}.)
		\item  All leaves of $T^0$ are colored $\g$.
		
	\end{enumerate}
	
	\item For each $\ell$-vertex $x$ in the
	$\ell$-level furthest from the
	root that is not yet colored, do the following:
	\begin{enumerate}
		\item Vertex $x$ is colored  $\g$ if:
		\begin{itemize}
			\item at most one
			of its $\ell$-children is non-$\g$ and;
			\item
			the resulting $\g$-cluster rooted at $x$
			and the subtree of $T^\#$
			restricted to the same leaf set are matching.
		\end{itemize}
		
		\item Else, vertex $x$ is colored  $\r$
		if:
		\begin{itemize}
			\item at most one
			of its $\ell$-children is non-$\g$;
			\item
			\srevision{and the following condition holds:} {\bf if} $x$ were colored $\g$,
			the resulting $\g$-cluster rooted at $x$
			and the corresponding matching
			subtree in $T^\#$ would {\bf not} be matching.
		\end{itemize}
		
		
		\item Else, vertex $x$ is colored $\y$. 
	\end{enumerate}
	
\end{enumerate}
In particular observe that, if $x$ is colored
$\y$,
at least two of its $\ell$-children are non-$\g$.

\subsection{Relating combinatorial distance,
	the number of matching subtrees and the overlap size}
\label{section:overlapSize}

Let $\#\g$ be the number of $\g$-vertices
in $T^0$ in the construction,
and similarly for the other colors.
For an $\ell$-vertex $x$ in $T^0$,
we let $T^0_x$ be the subtree of $T^0$
rooted at $x$ and we let $\vcal_\ell(T^0_x)$
be the set of $\ell$-vertices in $T^0_x$.
Recall from Definition~\ref{def:metric} that
we denote by $\gdist_{T^0}$ the graph
distance on $T^0$. 

Unlike the homogeneous case (see the proof of Claim~\ref{claim:sparsification-hmg}),
observe that it is possible for $\g$-clusters
to ``overlap'' in $T^\#$, that is, pairwise intersect.
We define the overlap formally as follows.
\srevision{See Figures~\ref{fig:overlap}
and~\ref{fig:overlap-detail}
for an illustration.}
\begin{figure}
 	\centering
 	\includegraphics[width = 0.9\textwidth]{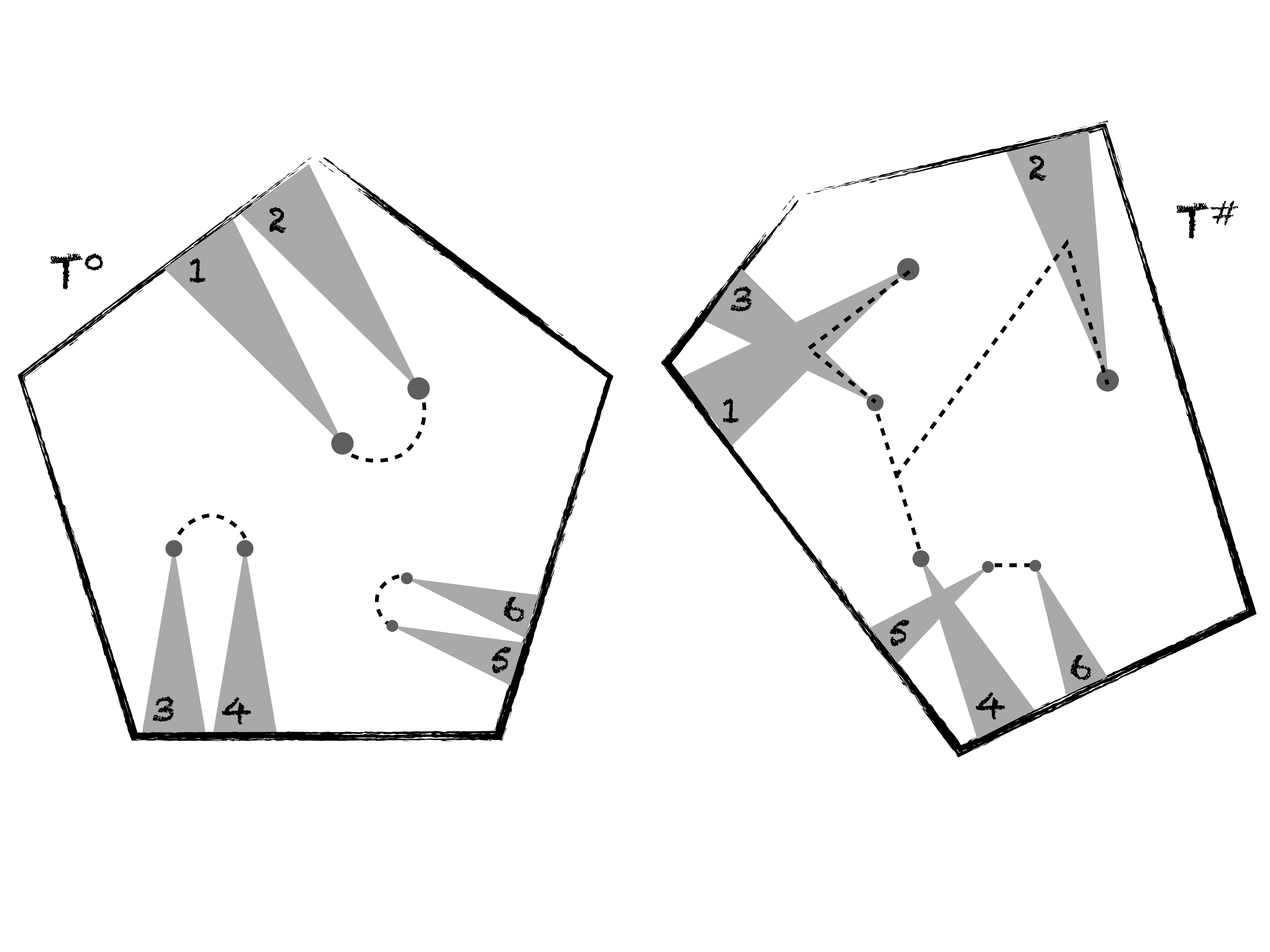}
 	\caption{Test subtrees overlap in $T^\#$. Matching subtrees are labeled with the same number.}\label{fig:overlap}
\end{figure}
\begin{figure}
	\centering
	\includegraphics[width = 0.9\textwidth]{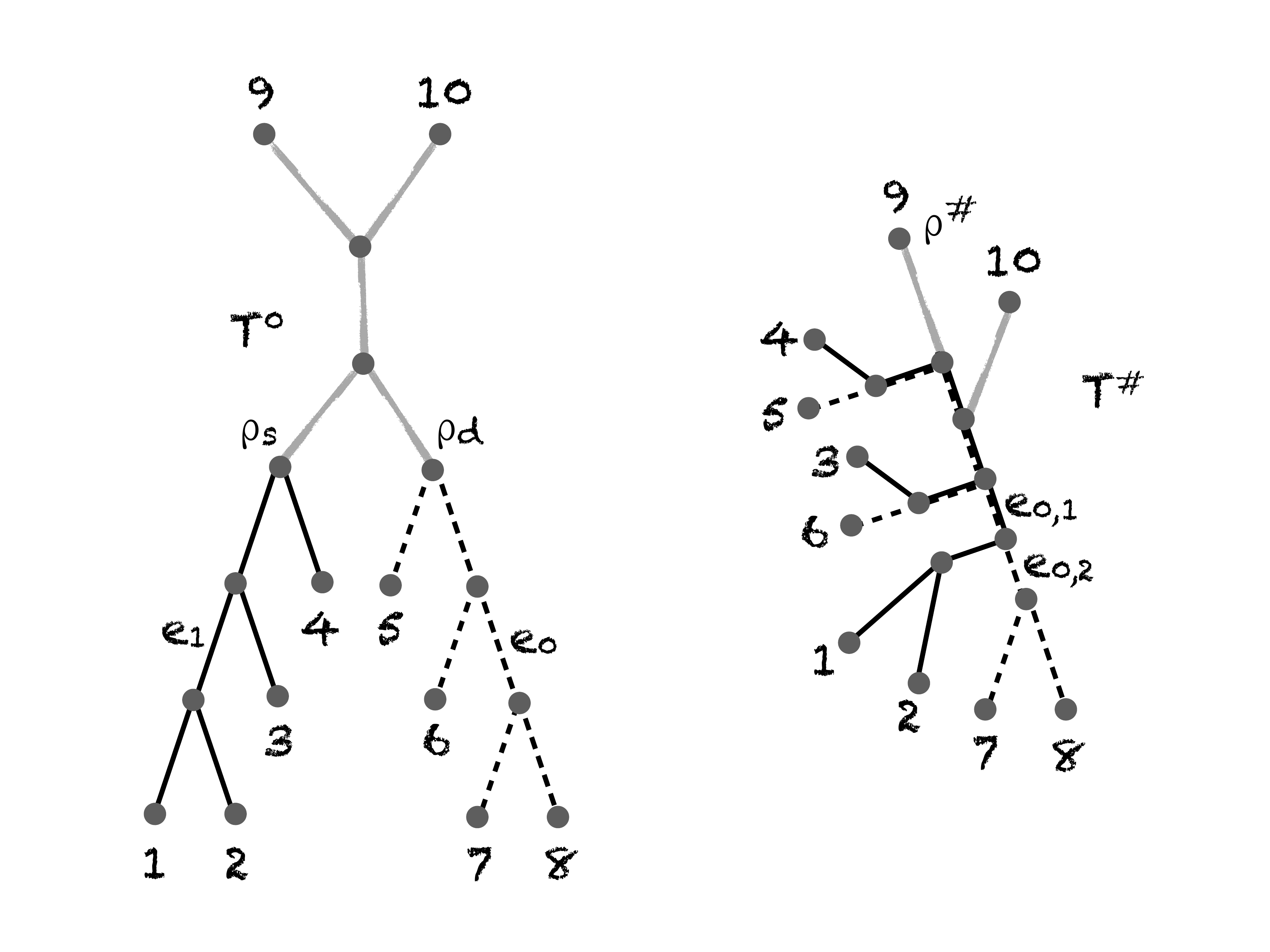}
	\caption{\srevision{A more detailed view of an overlap. The black solid and dashed subtrees are matching in $T^0$ and $T^\#$. The edges that are simultaenously solid and dashed in $T^\#$ are in the overlap. Note that an edge in $T^0$ may correspond to a path in $T^\#$. For instance $e_0$ is matched to the path formed by $e_{0,1}$ and  $e_{0,2}$. On the other hand, an edge in the overlap in $T^\#$ corresponds to several edges in $T^0$.
			For instance, $e_{0,1}$ corresponds to both $e_0$ and $e_1$. The subtrees in $T^0$ are rooted at $\rho_s$ and $\rho_d$ respectively. This rooting is consistent
			with the global rooting of $T^\#$ at vertex $\rho^\#$. }}\label{fig:overlap-detail}
\end{figure}
\begin{definition}[Overlap]
	\label{def:overlap}
	\srevision{
	An edge $e^\#$ in $T^\#$ {\em is in the overlap}
	if it belongs to the matching restricted subtrees in $T^\#$
	(the collection of which we denote by $\{\mcal_i\}_i$)
	of at least two distinct maximal $\g$-clusters
	in $T^0$ (the collection of which we denote by $\{\gcal_i\}_i$, where $\gcal_i$ and $\mcal_i$ are matching).
	The edge $e^\#$ is on a path of some $\mcal_i$
	corresponding to an edge $e^0$ in the matching
	$\gcal_i$.}
	We also say that $e^0$ {\em is in the overlap}.
	Let $\ocal^\#$ (respectively $\ocal^0$)
	denote the overlap, as a set of edges,
	in $T^\#$ (respectively $T^0$).
	We say that a vertex in $T^0$
	is {\em in the overlap}
	if it is adjacent to an edge in $\ocal^0$, and similarly
	for $T^\#$.
\end{definition}
\noindent 
The following bound allows us to work with the overlap in either $T^0$ or $T^\#$, whichever
is more convenient depending on the context.
Notice that it is not immediately clear that $\ocal^0$ and $\ocal^\#$ are roughly the same size because, by definition,
each edge in $\ocal^\#$ corresponds to several edges
in $\ocal^0$. 
\begin{claim}[Overlaps in $T^0$ and $T^\#$]
	\label{claim:overlapRelationship}
	We have
	$$
	\left|\ocal^0\right|  = \Theta\left(\left|\ocal^\#\right|\right),
	$$
	where the constants depend on $f, g, \quantum, \ell$.
\end{claim}
\begin{proof}
	\srevision{One direction is straightforward.
		Let $e^\#$ be an edge in $\ocal^\#$.
		There is an edge $e^0$ (in fact at least two)
		in a $\g$-cluster in $T^0$
		whose corresponding path in $T^\#$ includes $e^\#$. See Figure~\ref{fig:overlap-detail} for an illustration.
		Note that $e^0$ has weight at most $g$ and
		therefore can be identified in this way with at most
		$g\invquantum$ edges in $\ocal^\#$.
		That is, for every edge in $\ocal^0$ there are
		at most $g\invquantum$ edges in $\ocal^\#$, or
		$$
		\left|\ocal^0\right| \geq \frac{1}{g\invquantum}
		\left|\ocal^\#\right|.
		$$}
		
	The other direction is trickier because
	each edge in $\ocal^\#$ corresponds, by definition, to {\em several} edges
	in $\ocal^0$. However we claim that, in fact, only a small
	number of maximal $\g$-clusters can ``overlap on a given edge'' in $\ocal^\#$. To prove this we note that, being
	on a tree, most edges in the overlap are
	close to the ``boundary of the overlap,'' that is, they are close to
	vertices outside the overlap. But vertices outside
	the overlap
	necessarily belong to a single $\g$-cluster---which leads
	to a bound on the number of clusters overlapping
	on a given edge in $\ocal^\#$. 

	We first formalize what we mean by ``being close to the boundary of the overlap.'' Root $T^\#$ at an arbitrary vertex $\rt^\#$.
	Let $\gcal$ be a maximal $\g$-cluster in $T^0$
	and re-root $\gcal$ consistently with the rooting
	in $T^\#$, that is, at the vertex corresponding to the root of 
		the matching subtree in $T^\#$.
	See Figure~\ref{fig:overlap-detail} for an illustration. (Observe that there is no global
	rooting in $T^0$ that is consistent with the global rooting
	in $T^\#$. Instead, for this proof, each maximal $\g$-cluster in $T^0$ is rooted separately as explained above.)
	Let $\vcal^\gcal$ and $\wcal^{\gcal}$
	be the vertices in $\gcal$ and the vertices in the overlap
	in $\gcal$ respectively.
	Let $\wcal^\gcal_x$ (respectively $\vcal^\gcal_x$) be the vertices in $\wcal^\gcal$ (respectively $\vcal^\gcal$)
	below vertex $x$ (including $x$).
	\begin{definition}[Overlap-shallow vertices]
	\label{def:overlap-shallow}	
	We say that $x$
	is {\em overlap-shallow (with parameter $\beta$)} if
	\begin{equation}
	\label{eq:shallow-def}
	\sum_{y \in \wcal^\gcal_x}
	2^{-\frac{\gdist_{T^0}(x,y)}{2}} < \frac{\beta}{1 - 1/\sqrt{2}}.
	\end{equation}
	We let $\scal^\gcal$ be the set of overlap-shallow
	vertices in $\gcal$.	
	\end{definition}
	\noindent To see why this condition characterizes
	shallowness in the overlap, let $C > 0$ be a constant and say that
	$y$ is
	a {\em witness} for $x$ if 1) $y \in (\vcal^\gcal \setminus \wcal^{\gcal}) \cup (\vcal^\gcal \cap L)$, that is, $y$ is in $\gcal$ outside the overlap or is a leaf in $\gcal$, and if 2) $y$ is at 
	graph distance at most $C$ below $x$.
	Because a
	$\g$-cluster is $(\ell,1)$-dense (that is, nearly bifurcating),
	the sum 
	$$
	\sum_{y \in \vcal^\gcal_x}
	\sqrt{\frac{1}{2^{\gdist_{T^0}(x,y)}}},
	$$ 
	increases unboundedly
	as $y$ moves away from $x$---until the leaves are reached.
	Thus there is a $C$ depending only on $\ell$ and $\beta$
	such that, if $x$ is
	overlap-shallow, a witness is guaranteed to exist.
	In other words, $x$ is close to a vertex outside of the overlap or to a leaf. 
	If $y^\#$ is the vertex in $T^\#$ corresponding to
	witness $y$, we say that $y^\#$
	is a {\em $\#$-witness} for $x$. 
	
	We proceed in two steps.
	For the rest of this claim, we let $\beta = 2$.
	(We will need the same definition with a different
	value of $\beta$ in Section~\ref{section:largeOverlap}.)
	Our starting point is the bound
	\begin{equation}
	\label{eq:oo-wg}
	\left|\ocal^0\right|
	\leq \sum_{\gcal} \left|\wcal^{\gcal}\right|,
	\end{equation}
	where the sum runs through all maximal
	$\g$-clusters $\gcal$ in $T^0$. Indeed, the overlap
	forms a sub-forest of $T^0$ and, therefore, it has more vertices
	than edges.
	\begin{enumerate}
		\item {\it A large fraction of vertices in the overlap are shallow.} 
		We first relate $|\wcal^{\gcal}|$ and
		$|\scal^\gcal|$.
		Summing the criterion in~\eqref{eq:shallow-def}
		over all vertices in a maximal $\g$-cluster $\gcal$, we get
		\begin{equation}
		\label{eq:factor-of-two-pre-pre}		
		\sum_{x \in \wcal^\gcal}
		\left[
		\sum_{y \in \wcal^\gcal_x}
		2^{-\frac{\gdist_{T^0}(x,y)}{2}}
		\right]
		= \sum_{y \in \wcal^\gcal}
		\left[
		\sum_{x : y \in \wcal^\gcal_x}
		2^{-\frac{\gdist_{T^0}(x,y)}{2}}		
		\right],
		\end{equation}
		by interchanging the sum. Note that the expression in square brackets on the r.h.s.~is a sum over
		the overlap on the path from $y$ towards the root of $\gcal$. Because the sum is geometric, we obtain the bound
		\begin{equation}
		\label{eq:factor-of-two-pre}
		\sum_{y \in \wcal^\gcal}
		\left[
		\sum_{x : y \in \wcal^\gcal_x}
		2^{-\frac{\gdist_{T^0}(x,y)}{2}}
		\right]
		\leq 
		\sum_{y \in \wcal^\gcal} 
		\left[
		\frac{1}{1 - 1/\sqrt{2}}
		\right]
		= \frac{\left|\wcal^\gcal\right|}{1 - 1/\sqrt{2}},
		\end{equation}
		where we used that $\sum_{z \geq 0} (1/\sqrt{2})^{z} = (1 - 1/\sqrt{2})^{-1}$.
		It follows, by contradiction, that 
		\begin{equation}
		\label{eq:factor-of-two}
		|\scal^\gcal| > \frac{1}{2} |\wcal^\gcal|.
		\end{equation}
		Indeed, if that were not the case, that is, 
		if $|\wcal^\gcal \setminus \scal^\gcal| > \frac{1}{2} |\wcal^\gcal|$, then the
		sum on the l.h.s.~of~\eqref{eq:factor-of-two-pre-pre} would be $> \frac{1}{2} |\wcal^\gcal| \frac{2}{1 - 1/\sqrt{2}}$ by~\eqref{eq:shallow-def}, 
		contradicting~\eqref{eq:factor-of-two-pre}.
		Combining~\eqref{eq:oo-wg} and~\eqref{eq:factor-of-two}, we get the bound
		\begin{equation}
		\label{eq:oo-wg-2}
		\left|\ocal^0\right|
		< 2 \sum_{\gcal} \left|\scal^{\gcal}\right|.
		\end{equation}
		
		\item {\it Overlap-shallow vertices in $T^0$
			can be mapped to vertices in the overlap in
			$T^\#$ with little duplication.} 
		It remains to relate $|\scal^\gcal|$ and $|\ocal^\#|$. This step is delicate 
		because the definition
		of the overlap (Definition~\ref{def:overlap}) {\em differs somewhat in $T^0$ and $T^\#$}. 
		Note, in particular,
		that a vertex $x^0$ in the overlap in $T^0$ is matched to a vertex $x^\#$ in $T^\#$ {\em which may not itself be in the overlap.} Instead, all we can say is that $x^0$ is incident with an edge in $T^0$
		whose corresponding path in $T^\#$ contains
		a vertex $\tilde{x}^\#$ in the overlap. 
		To each vertex $x^0$ in $\cup_\gcal \scal^\gcal$,
		associate a vertex $\tilde{x}^\#$ in the
		overlap in $T^\#$ as we just described, with the following extra condition: two vertices $x^0 \neq z^0$ in the same $\g$-cluster $\gcal$ must be associated 
		with {\em distinct} vertices $\tilde{x}^0 \neq \tilde{z}^0$ in the overlap in $T^\#$. This is always
		possible because, if $\tilde{x}^0, \tilde{z}^0$ 
		are incident with the same edge, we can associate to them distinct vertices from the overlap on the corresponding path in $T^\#$ (say, the closest
		in graph distance to the matching vertex). 
		Let
		$\scal^\#$ be the set of all these $\tilde{x}^\#$s
		and
		observe that 
		\begin{equation}
		\label{eq:ssharp-osharp}
		|\scal^\#|
		\leq
		2 |\ocal^\#|.
		\end{equation}
		Note however that we cannot directly bound the size of $\cup_\gcal \scal^\gcal$ with the size of $\scal^\#$ because some vertices in $\scal^\#$ may be associated
		with vertices in {\em different} $\g$-clusters in $T^0$. Let $\tilde{x}^\# \in \scal^\#$ and let $x^0_1, \ldots, x^0_h$ be the vertices
		in $\cup_{\gcal}\scal^\gcal$ to which it is
		associated.
		What we need is to bound $h$. 
		This will follow from a number of observations.
		See Figure~\ref{fig:overlap-size}
		\begin{figure}
			\centering
			\includegraphics[width = 0.9\textwidth]{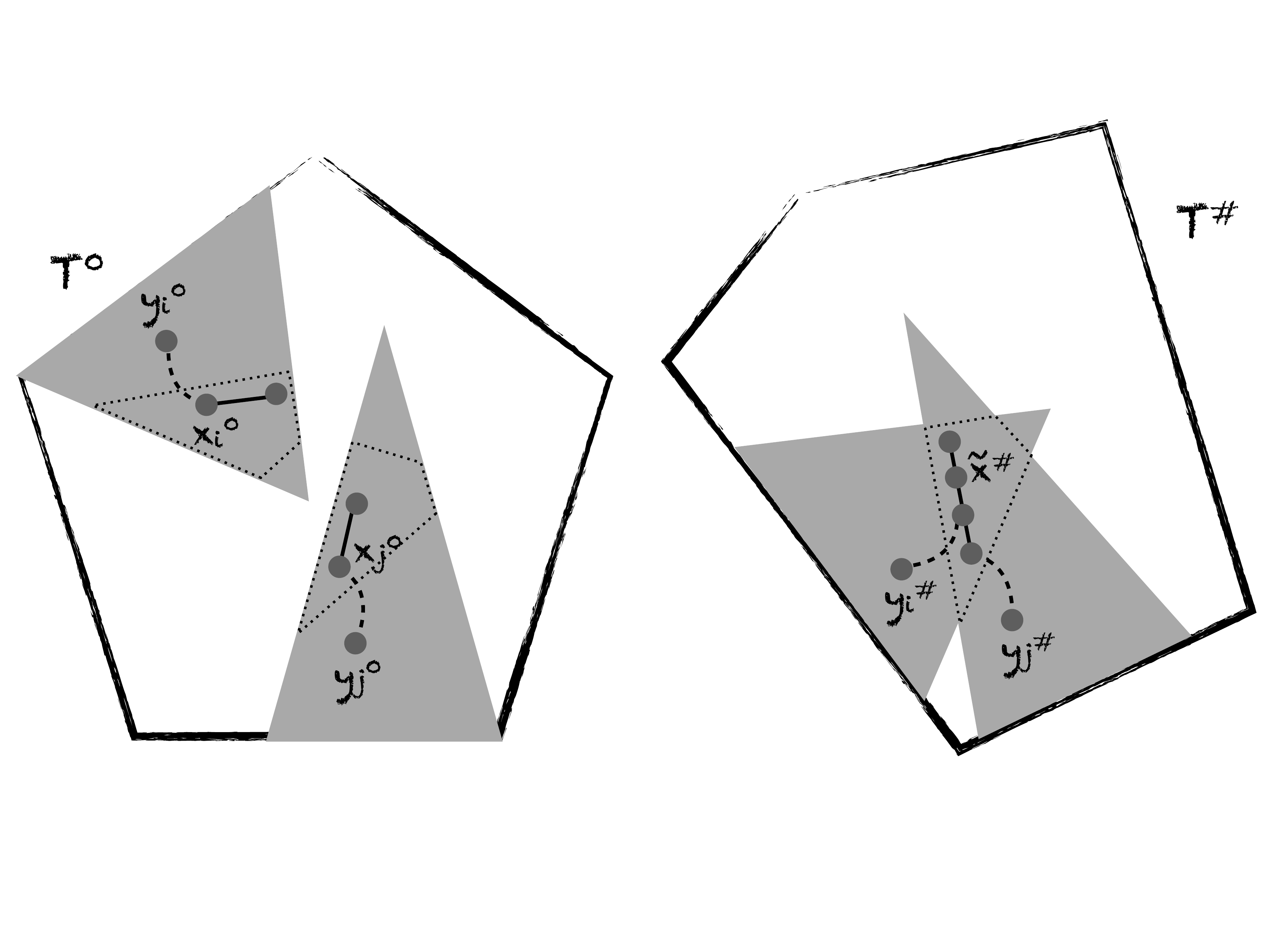}
			\caption{\srevision{
					Witnesses outside the overlap.
					Here $x_i^0$ and $x_j^0$ are associated
					to $\tilde{x}^\#$. Their respective $\#$-witnesses are $y_i^\#$ and $y_j^\#$. The dotted lines surround the
					overlap. }}\label{fig:overlap-size}
		\end{figure}	
		\begin{enumerate}
			\item \label{item:x0i-distinct} By the construction above, each 
			$x^0_i$ belongs to a {\em distinct} maximal $\g$-cluster $\gcal_i$. 
			
			\item Let $y_i^0$ and $y^\#_i$ be 
			a witness and 
			$\#$-witness for $x^0_i$ respectively. The {\em existence} of such witnesses was established immediately after Definition~\ref{def:overlap-shallow}. 
			
			\item \label{item:ysharpi-single} Each $y^\#_i$ belongs to a {\em single} $\g$-cluster, that is, $\gcal_i$. 
			Indeed, by definition, either $y^\#_i$
			is outside the overlap in $\gcal_i$, or it is a leaf in $\gcal_i$.
			(Because of the way the $\g$-clusters
			are constructed in $T^0$, each leaf 
			belongs to one maximal $\g$-cluster.)
			
			\item 
			Combining (a) and (c), $y_1^\#, \ldots, y^\#_h$ must be {\em distinct}
				vertices in $T^\#$.
				
			\item Because $x_i^0$ and $y_i^0$
			are at graph distance $C$ and each edge in $T^\#$
			corresponds to at most $g\invquantum$ edges in $T^0$,
			the {\em graph distance} between $\tilde{x}^\#$ and $y_i^\#$
			is at most $(C+1)g\invquantum$. Here the $+1$
			accounts for the fact that, as explained above, $x_i^0$ and $\tilde{x}^\#$ may not be matching.
			
			\item There are {\em at most
			$3(2^{(C+1)g \invquantum+1}-1) + 1$ vertices} in $T^\#$ at graph distance $(C+1) g \invquantum$ from $\tilde{x}^\#$. That follows from the 
			fact that an $h$-level (counting the root) complete binary tree
			has $2^{h+1} - 1$ vertices. 
			
			\item Combining (d), (e) and (f), we have established 
			that $h \leq 3(2^{(C+1) g \invquantum+1}-1)$. 
			
		\end{enumerate}
		Thus, using~\eqref{eq:ssharp-osharp},
		\begin{equation}
		\label{eq:sg-osharp}
		\sum_{\gcal}|\scal^\gcal|
		\leq |\scal^\#| \cdot 3(2^{(C+1) g \invquantum+1}-1)
		\leq 2 |\ocal^\#| \cdot 3 (2^{(C+1) g \invquantum+1}-1).
		\end{equation}
		
	\end{enumerate}
	It remains to combine~\eqref{eq:oo-wg-2}
	and~\eqref{eq:sg-osharp} to obtain
	$$
	|\ocal^0|
	\leq 
	12 (2^{(C+1) g \invquantum+1}-1) |\ocal^\#|.
	$$
	That concludes the proof.
\end{proof}

We now relate the blow-up distance between $T^0$ and
$T^\#$ to the number of $\r$-vertices, from
which tests can potentially be constructed,
and the size of the overlap.
We first bound the number
of yellow vertices.
\begin{claim}[Bounding the number of yellow vertices]
	\label{claim:bounding-yellow}
	We have
	$$
	\#\y \leq \#\r.
	$$
\end{claim}
\begin{proof}
	From our construction, each $\y$-vertex in $T^0$ has at least two non-$\g$-children.
	Hence, intuitively,
	one can think of the $\y$-vertices as forming the internal vertices of
	a forest of multifurcating trees whose leaves are $\r$-vertices.
	The inequality follows.
	
	Formally, if $x$ is a $\y$-vertex, from the observation
	above we have
	\begin{equation}\label{eq:bounding-yellow-1}
	\sum_{y \in \vcal_\ell(T^0_x)}
	2^{-\frac{\gdist_{T^0}(x,y)}{\ell}} \ind\{\text{$y$ is a $\r$-vertex}
	\} \geq 1,
	\end{equation}
	by induction on the $\ell$-levels starting with
	the level farthest away from the root. Similarly
	if $y$ is an $\r$-vertex,
	we have
	\begin{equation}\label{eq:bounding-yellow-2}
	\sum_{x : y \in \vcal_\ell(T^0_x)}
	2^{-\frac{\gdist_{T^0}(x,y)}{\ell}} \ind\{\text{$x$ is a $\y$-vertex}\} < 1,
	\end{equation}
	where the inequality follows from the fact that
	the sum is over a path from $x$ to the root of $T^0$.
	Summing~\eqref{eq:bounding-yellow-1} over
	$\y$-vertices $x$ and~\eqref{eq:bounding-yellow-2} over
	$\r$-vertices $y$ gives the same
	quantity on the LHS, so that the RHS gives the inequality. 
\end{proof}
\begin{claim}[Relating blowup, $\#\r$, and overlap]
	\label{claim:blowup-overlap}
	There is a constant $0 < C_\ocal < +\infty$,
	depending on $\ell$, $g$ and $\invquantum$,
	such that
	$$
	\blowup\left(T^0,T^\#\right) \leq C_\ocal \left(\#\r + \left|\ocal^\#\right|\right).
	$$
\end{claim}
\begin{proof}
Our goal is to display a blowup from $T^0$ to $T^\#$ 
whose number of edges
is bounded by a constant times the number of $\r$-vertices
plus the size of the overlap in $T^\#$.
We proceed in two steps:
	
	\begin{itemize}
		\item  {\bf Edge removals.}		
		First we remove all edges in $T^0$ that are not
		in a maximal $\g$-cluster. To count how many 
		such edges there are,
		we observe that there are at most
		$2^{\ell+1} - 2$ edges between a non-$\g$-vertex
		and its $\ell$-children. The edge above
		each non-$\g$-vertex is also removed if its parent $\ell$-vertex
		is colored $\g$. Hence, 
		we need to remove at most
		$
		(2^{\ell+1} - 1) (\#\r + \#\y)
		$
		edges. We also remove all edges in the overlap, which
		adds at most an extra $|\ocal^0|$ edges to the total of those removed.
		
		\smallskip
		
		Next we remove every edge adjacent to a degree-2 vertex produced by the removals above. Each edge removed above
		produces at most $4$ such edges, bringing the total number of edges removed
		so far to at most
		\begin{equation}
		\label{eq:blowup-overlap-aux1}
		5[(2^{\ell+1} - 1)(\#\r + \#\y) + |\ocal^0|] 
		\leq 5\cdot 2^{\ell+1}(\#\r+ |\ocal^0|),
		\end{equation}
		where we used Claim~\ref{claim:bounding-yellow}.
		
		\smallskip
		
		We call what is left the backbone.
		Because the backbone is a subset of the
		$\g$-clusters, every vertex of the backbone
		corresponds to a (non-extra) vertex in $T^\#$.
		(Recall that {\em extra} vertices were defined in Definition~\ref{def:matching}.) Every edge in the backbone, on the other hand, corresponds to
		a path in $T^\#$ with at most $g\invquantum$
		edges. Because $T^0$ and $T^\#$ have the same overall
		number of edges, the number of edges in $T^\#$ that do not lie on the backbone is at most $5\cdot 2^{\ell+1}(\#\r+ |\ocal^0|)$ by~\eqref{eq:blowup-overlap-aux1}. Each such edge
		may be incident (in $T^\#$) to at most $2$ edges in the backbone
		that are a path of length at least $2$ in $T^\#$.
		We also remove all such edges from the backbone,
		finally bringing the total of edges removed to
		at most $15\cdot 2^{\ell+1}(\#\r+ |\ocal^0|)$.
		
		\item {\bf Edge additions.} All edges and vertices left after the edge removals above correspond to (non-path) edges and (non-extra) vertices of $T^\#$. 
		Because $T^0$ and $T^\#$ have the same overall
		number of edges, the number of edge additions needed to obtain $T^\#$ at this point is at most $15\cdot 2^{\ell+1}(\#\r+ |\ocal^0|)$.
	\end{itemize}
	From Claim~\ref{claim:overlapRelationship},
	the constant $C_\ocal$ in the statement can
	be taken to be a function of $\ell$, $g$ and $\invquantum$.
\end{proof}


Our next goal is to construct batteries with
a number of tests scaling linearly in the
blowup distance between $T^0$ and $T^\#$.
Using Claim~\ref{claim:blowup-overlap}, we first divide the analysis into two cases depending on the
values of $\#\r$ and $|\ocal^\#|$.
\begin{itemize}
	\item {\bf Large overlap.}
	If
	\begin{equation}\label{eq:largeOverlap}
	|\ocal^\#| \geq \frac{1}{10} \frac{\blowup(T^0,T^\#)}{C_\ocal},
	\end{equation}
	we say that we are in the {\em large overlap} case.
	We will show in Section~\ref{section:largeOverlap}
	that a linear (in the blowup distance)
	number of tests can be built ``around
	the periphery of the overlap.''
	The choice of the factor $1/10$ will be justified
	in Claim~\ref{claim:recoloring-case1}.

	\item {\bf Many $\r$-vertices.} 
	If, instead,
	\begin{equation}\label{eq:smallOverlap}
	|\ocal^\#| < \frac{1}{10} \frac{\blowup(T^0,T^\#)}{C_\ocal},
	\end{equation}
	we say that we are in the {\em many-$\r$} case.
	To justify the name we note that by Claim~\ref{claim:blowup-overlap}, if~\eqref{eq:smallOverlap} holds, then
	\begin{equation}
	\label{eq:smallOverlap-reds}
	\#\r \geq \frac{9}{10}\frac{\blowup(T^0,T^\#)}{C_\ocal}.
	\end{equation}
	In that case, we proceed similarly to the homogeneous case
	and construct a distinguishing test
	for a linear fraction of $\r$-vertices.
	See Section~\ref{section:manyr}.
	
\end{itemize}

\subsection{Constructing a battery of tests: Many-$\r$ case}
\label{section:manyr}

We now construct a battery of tests in the many-$\r$
case. This case is similar to the homogeneous case
although many new difficulties arise.
The basic idea remains the same: each $\r$-vertex has
two $\g$-children which satisfy many of the requirements
of a battery and therefore can potentially be used as a test pair.
In particular, they are the roots of dense subtrees that are matching with their corresponding restricted subtrees
in $T^\#$ and their evolutionary distance differs
in $T^0$ and $T^\#$. Note that, in the
many-$\r$ case, we also have a
number of $\r$-vertices
that scales linearly in the blowup distance.
Compared to the homogeneous case, however, there
are new issues to address to construct
a battery of tests, mainly the possibility of overlapping
$\g$-clusters and of non-co-hanging pairs
in $T^\#$.

In this section, $T^0$ and $T^\#$ are fixed. To simplify
notation, we let $\Delta = \blowup\left(T^0,T^\#\right)$.
Fix $\wp = 1$. Choose $\ell = \ell(g,\wp)$ as
in Proposition~\ref{prop:distinguishing}.
Then take
\begin{equation}
\label{eq:def-gamma}
\Gamma = (6 + 2\invquantum g) \ell,
\end{equation}
and set
$\gamma_t \geq \Gamma$, a multiple of $\ell$, and $C$ as
in Proposition~\ref{prop:distinguishing}.

\paragraph{Choosing non-overlapping $\g$-clusters.}
To satisfy the requirements of the battery, 
the test subtrees must be non-intersecting
in $T^\#$. (By construction, the test subtrees are non-intersecting in $T^0$.)
We proceed by showing that sufficiently many non-overlapping
$\g$-clusters can be found. 
For this purpose, we use a re-coloring procedure.
Re-color $\b$ (for black) those $\r$-vertices
that have at least one $\g$-child who is the root of a
$\g$-cluster that intersects with another $\g$-cluster in $T^\#$.
\srevision{(This recoloring procedure is performed only once.)}
Intuitively, if too many $\r$-vertices are lost in this
recoloring step, then the overlap must be large. That cannot be the case by~\eqref{eq:smallOverlap}. Indeed,
we prove the following.
\begin{claim}[Re-coloring]
	\label{claim:recoloring-case1}
	In the many-$\r$ case,
	after re-coloring, we have
	$$
	\#\r \geq \frac{\Delta}{2C_\ocal},
	$$
	\srevision{where $C_\ocal$,
		which depends on $\ell$, $g$ and $\invquantum$, was defined in
		Claim~\ref{claim:blowup-overlap}.}
\end{claim}
\begin{proof}
	Assume maximal $\g$-cluster $\gcal_i$
	intersects with a distinct maximal $\g$-cluster in $T^\#$
	and let $\mcal_i$
	be the
	matching subtree corresponding to $\gcal_i$
	in $T^\#$.
	Consider a shortest path in graph distance 
	between a leaf in $\mcal_i$ and the overlap
	in $T^\#$. Let $v_i$
	be the vertex in $T^\#$ where this path enters
	the overlap. \srevision{Because 1) $T^\#$ is bifurcating,
	2) at least one edge adjacent to $v_i$ must be in
	$\ocal^\#$, and 3) at least one edge adjacent to
	$v_i$ must be in $\mcal_i$ outside the overlap, 
	it follows that $v_i$ can arise as
	the entrance vertex to the overlap
	for at most two maximal $\g$-clusters.
	Hence, each maximal $\g$-clusters
	intersecting with another maximal $\g$-cluster
	is associated an entrance vertex in the overlap that can be used
	at most twice. So
	the number of such clusters is bounded by
	$$
	2 \cdot 2\left|\ocal^\#\right| \leq 4\frac{1}{10} \frac{\Delta}{C_\ocal},
	$$
	where we used~\eqref{eq:smallOverlap} and
	where we took into account that the number of vertices
	in the overlap is at most twice the number of edges
	in the overlap. Moreover, observe that each such
	cluster contributes to the recoloring of at most
	one $\r$-vertex.}
	That implies that the number of recolored
	$\ell$-vertices is at most $4\frac{1}{10} \frac{\Delta}{C_\ocal}$.
	After recoloring we therefore have
	$$
	\#\r \geq \frac{9}{10}\frac{\Delta}{C_\ocal}
	- 4\frac{1}{10} \frac{\Delta}{C_\ocal}
	= \frac{\Delta}{2 C_\ocal},
	$$
	where we used~\eqref{eq:smallOverlap-reds}.
\end{proof}

In the rest of this subsection, we build a
$(\ell,\wp,\Gamma,\gamma_t,I)$-battery
$$
\{((y^0_i,z^0_i);(Y^0_i,Z^0_i))\}_{i=1}^I
\text{ (in $T^0$)
	and }
\{((y^\#_i,z^\#_i);(Y^\#_i,Z^\#_i))\}_{i=1}^I
\text{ (in $T^\#$)}
$$
with
corresponding $\alpha_i$s as defined in Definition~\ref{def:battery}.
We number the $\r$-vertices $i=1,\ldots,I'$ after
recoloring
and we build one test panel for each
$\r$-vertex. Here it will turn out that $I' \geq I$ as we will later
need to reject some of the test panels to avoid unwanted
correlations.
We root $T^\#$ at an arbitrary vertex $\rt^\#$.
(The rootings of $T^0$ and $T^\#$
need not be consistent at this point.)

\paragraph{Constructing co-hanging test panels.}
Let $x^0_i$ be an $\r$-vertex (after recoloring) in $T^0$. Because $x^0_i$ is colored
$\r$, by definition it has at least $2^\ell - 1$
$\g$-children, but its $\g$-children are connected in a
different way in $T^\#$.
We distinguish between two cases:
\begin{enumerate}
	\item {\em An appropriate co-hanging pair can be found:} All pairs of $\g$-children of $x^0_i$ are
	the roots of co-hanging, non-overlapping matching subtrees in $T^\#$.
	In that case at least one pair of $\g$-children $(y^0_i,z^0_i)$
	must be at a different evolutionary distance
	in $T^0$ than the corresponding pair $(y^\#_i,z^\#_i)$
	in $T^\#$. We use these pairs
	as our test panel, modulo the following re-rooting. If
	a test subtree is not rooted consistently
	in $T^0$ and $T^\#$, we move the corresponding test vertices
	to one of their corresponding $\g$-children where the rooting is consistent.
	This can always be done as there is at most one $\g$-child of a $\g$-vertex
	between itself and the root of $T^\#$. All other choices
	lead to a consistent rooting.
	See Figure~\ref{fig:many-r-co-hanging}
	for an illustration.
\begin{figure}
	\centering
	\includegraphics[width = 0.9\textwidth]{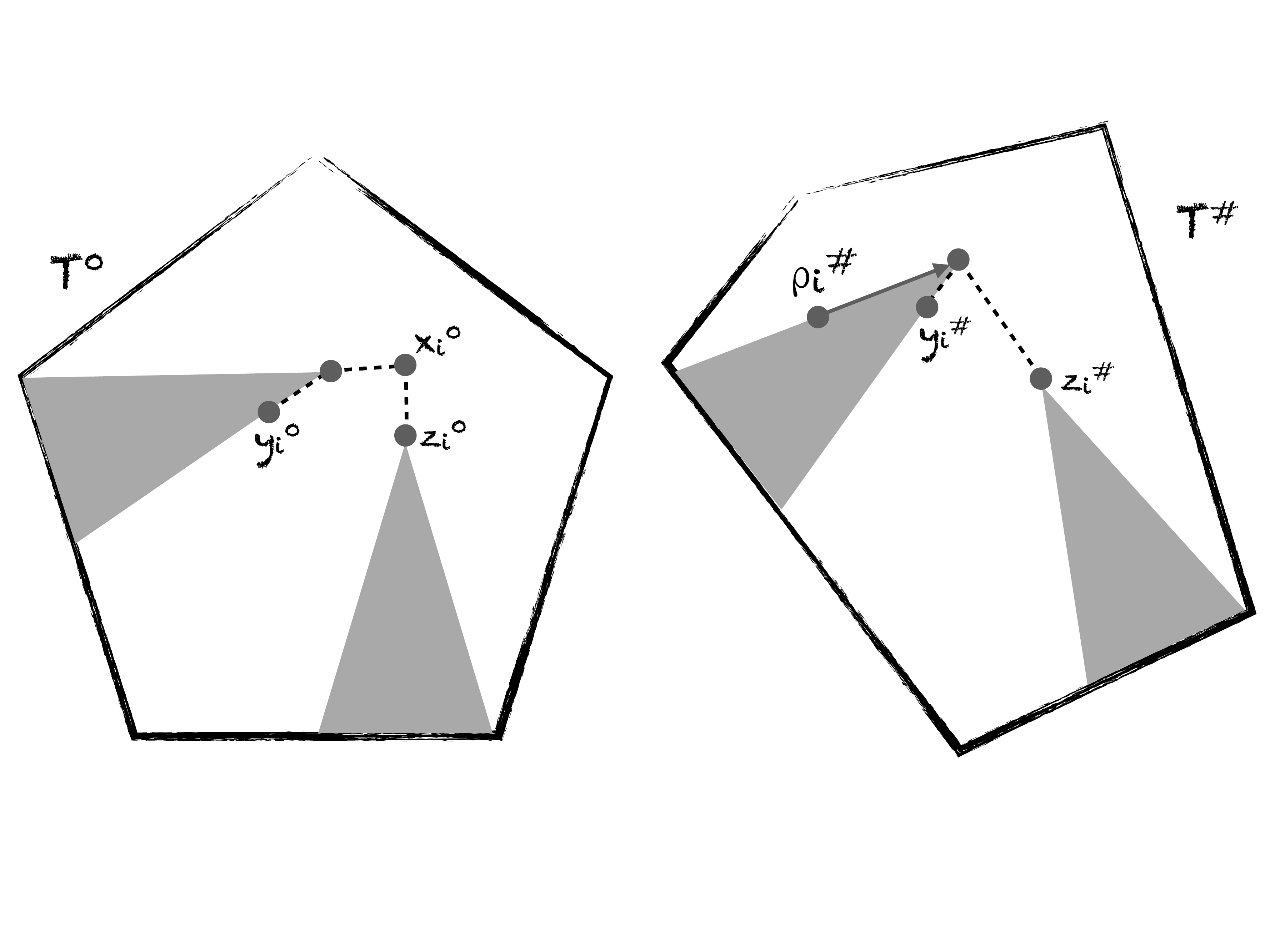}
	\caption{\srevision{
			Construction of the test in the co-hanging sub-case of the many-$\r$ case. The root of 
			the cluster in $T^\#$ is denoted
			by $\rho_i^\#$. Moving the test pair to the $\g$-children $y_i^0$ and $y_i^\#$ has the effect of making the new test subtrees rooted consistently. }}\label{fig:many-r-co-hanging}
\end{figure}

	\item {\em There exists a non-co-hanging pair:} Otherwise at least one pair $(\tilde{y}^0_i,\tilde{z}^0_i)$ of $\g$-children of $x^0_i$, with corresponding
	pair $(\tilde{y}^\#_i,\tilde{z}^\#_i)$, has a connecting path
	in
	$T^\#$ that intersects with the corresponding matching
	test subtrees, $\tilde{Y}^\#_i$ or $\tilde{Z}^\#_i$ (or both). Indeed, although by construction 
	the test subtrees are matching in $T^0$ and $T^\#$,
	the path connecting them may be ``positioned differently.'' See Figure~\ref{fig:many-r-long-path-between} 
	for an illustration of such a case.
	The main goal of the next claim is to show how to construct an appropriate co-hanging test panel in this case.
\end{enumerate}
\begin{claim}[Co-hanging pairs]
	\label{claim:cohanging-case1}
	In the many-$\r$ case,
	after recoloring, for each remaining $\r$-vertex $x^0_i$
	we can find
	a test pair of $\g$-vertices $(y^0_i,z^0_i)$ in the
	subtree rooted at $x^0_i$ in $T^0$,
	with corresponding test pair $(y^\#_i,z^\#_i)$ in $T^\#$,
	such that the test panel satisfy the cluster and
	pair requirements of a battery.
\end{claim}
\begin{proof}
	We consider again the two cases above.
	
	\medskip
	
	\noindent {\bf An appropriate co-hanging pair can be found.} We proceed as we described above the statement
	of the claim.
	By construction the test subtrees, that is,
	the $\g$-clusters rooted at the test vertices, are
	$(\ell,1)$-dense.
	The test subtrees are also matching, co-hanging and their
	roots are at different evolutionary distances in $T^0$
	and $T^\#$ by~\eqref{eq:different-distance}.
	Finally, the test pair in $T^0$ is proximal as
	$$
	\gdist_{T^0}(y^0_i, z^0_i) \leq 4\ell \leq \Gamma,
	$$
	where we note that re-rooting procedure may increase the distance by at most $2\ell$.

	\medskip

	\noindent {\bf There exists a non-co-hanging pair.} For the second case, we use the notation of Item 2 above the statement of the claim.
	Because $T^\#$ has no cycle,
	the path between
	$\tilde{y}^\#_i$ and $\tilde{z}^\#_i$
	must be of the following form:
	there is a vertex $v^\#$ in $\tilde{Y}^\#_i$
	(possibly equal to $\tilde{y}^\#_i$)
	and a vertex $w^\#$ in $\tilde{Z}^\#_i$
	(possibly equal to $\tilde{z}^\#_i$) such that
	the path between $\tilde{y}^\#_i$ and $\tilde{z}^\#_i$
	1) intersects with $\tilde{Y}^\#_i$ between
	$\tilde{y}^\#_i$ and $v^\#$,
	2) does not intersect with either $\tilde{Y}^\#_i$ or
	$\tilde{Z}^\#_i$ between
	$v^\#$ and $w^\#$,
	and 3) intersects with $\tilde{Z}^\#_i$ between
	$w^\#$ and $\tilde{z}^\#_i$.
	\srevision{See Figures~\ref{fig:many-r-long-path-between}
		and~\ref{fig:many-r-short-path-between}
		for an illustration}.
	We let $v^0$ and $w^0$ be the extra vertices
	corresponding respectively to $v^\#$ and $w^\#$ in $T^0$. (Recall that {\em extra} vertices were defined
	in Definition~\ref{def:matching}.)
	We consider two subcases:
	\begin{enumerate}
		\item {\bf $\tilde{y}^\#_i,\tilde{z}^\#_i$ are ``far'' in $T^\#$.} Suppose first that
		\begin{equation}
		\label{eq:yzsharp-far}
		\dist_{T^\#}(\tilde{y}^\#_i,\tilde{z}^\#_i) > 2g\ell.
		\end{equation}
		That case is illustrated in Figure~\ref{fig:many-r-long-path-between}.
		We construct a co-hanging test panel 
		by choosing appropriate $\g$-children as follows.
		Recall that we must ensure in particular that our chosen pairs are co-hanging and at different
		evolutionary distances in $T^0$ and $T^\#$.
		If $v^\# = \tilde{y}^\#_i$, we simply set $y^\#_i = \tilde{y}^\#_i$.
		If $v^\# \neq \tilde{y}^\#_i$,
		let $y^0_i$ be a $\g$-child of $\tilde{y}^0_i$
		which satisfies the following:
		\begin{itemize}
			\item 
			Observe that one of the two children of $\tilde{y}^0_i$
			is the root of a subtree containing $v^0$. 
			We choose $y^0_i$ in the {\em other} subtree.
			This is to guarantee that the path joining
			$y^0_i$ and $z^0_i$ (below) does not
			intersect the resulting subtrees. See Figure~\ref{fig:many-r-long-path-between}.
			
			\item We also choose $y^0_i$ so that the test subtree
			rooted at $y^0_i$ and the test subtree rooted at the
			corresponding vertex $y^\#_i$ in $T^\#$ are rooted
			consistently. This can always be done as there is at most one $\g$-child of $\g$-vertex
			between itself and the root of $T^\#$. All other choices,
			of which there are overall at least $2^\ell/2 - 1$ 
			satisfying the first property above,			
			lead to a consistent rooting.
			We pick $y_i^0$ arbitrarily among them.
			
		\end{itemize}
        We define $z^0_i$ and $z^\#_i$ similarly.
        As a result of this construction, the subtrees rooted at $y^\#_i$ and $z^\#_i$
        are co-hanging in $T^\#$ by construction.
		\begin{figure}
			\centering
			\includegraphics[width = 0.9\textwidth]{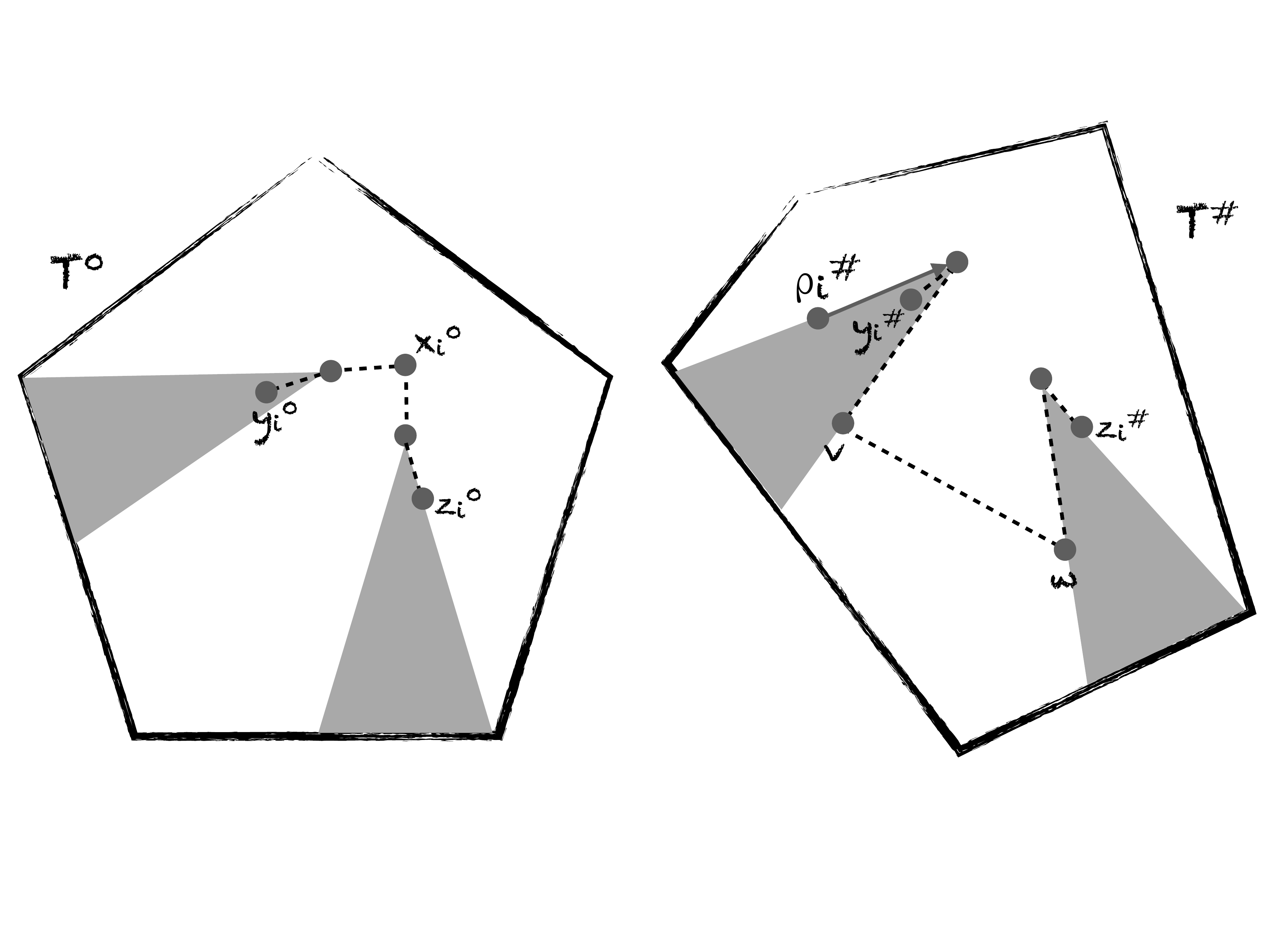}
			\caption{\srevision{
					Construction of the test in the non-co-hanging sub-case of the many-$\r$ case when
					$\tilde{y}^\#_i,\tilde{z}^\#_i$ are ``far.'' The root of 
					the cluster in $T^\#$ is denoted
					by $\rho_i^\#$. }}\label{fig:many-r-long-path-between}
		\end{figure}
		Moreover, the evolutionary distance between
		$y_i^0$ and $z_i^0$ satisfies 
		\begin{eqnarray}
		\dist_{T^0}(y^0_i, z^0_i)
		&=& \dist_{T^0}(y^0_i, \tilde{y}^0_i)
		+ \dist_{T^0}(\tilde{y}^0_i, \tilde{z}^0_i)
		+ \dist_{T^0}(\tilde{z}^0_i, z^0_i)\nonumber\\
		&\leq&
		\dist_{T^0}(y^0_i, \tilde{y}^0_i)
		+ 2g\ell
		+ \dist_{T^0}(\tilde{z}^0_i, z^0_i)\nonumber\\
		&<&
		\dist_{T^\#}(y^\#_i, \tilde{y}^\#_i)
		+ \dist_{T^\#}(\tilde{y}^\#_i,\tilde{z}^\#_i)
		+ \dist_{T^\#}(\tilde{z}^\#_i, z^\#_i)\nonumber\\
		&=& \dist_{T^\#}(y^\#_i,z^\#_i),\label{eq:different-distance}
		\end{eqnarray}
		where, on the second line, we used that $\tilde{y}^0_i$, $\tilde{z}^0_i$
		were chosen to be $\g$-children of $x_i^0$ in $T^0$ and, on third line, we used~\eqref{eq:yzsharp-far} and the fact that
		$\dist_{T^0}(y^0_i, \tilde{y}^0_i) = \dist_{T^\#}(y^\#_i, \tilde{y}^\#_i)$
		and $\dist_{T^0}(\tilde{z}^0_i, z^0_i) = \dist_{T^\#}(\tilde{z}^\#_i, z^\#_i)$
		by the matching condition.
		That is, $\dist_{T^0}(y^0_i, z^0_i)
		\neq \dist_{T^\#}(y^\#_i,z^\#_i)$
		as required.
		Hence the pairs $(y^0_i, z^0_i)$ and $(y^\#_i, z^\#_i)$
		satisfy the cluster and pair requirements of the battery.
		Indeed by construction the test subtrees are
		$(\ell,1)$-dense.
		The test subtrees are also matching, co-hanging and their
		roots are at different evolutionary distances in $T^0$
		and $T^\#$ by~\eqref{eq:different-distance}.
		Finally, the test pair in $T^0$ is proximal as
		$$
		\gdist_{T^0}(y^0_i, z^0_i) \leq 4\ell \leq \Gamma,
		$$
		because $y^0_i, z^0_i$ are $\g$-grandchildren
		of $x^0_i$.
		
		\item {\bf $\tilde{y}^\#_i,\tilde{z}^\#_i$ are ``close'' in $T^\#$.}
		Assume instead that
		\begin{equation}\label{eq:non-co-hanging-close}
		\dist_{T^\#}(\tilde{y}^\#_i,\tilde{z}^\#_i) \leq 2g\ell.
		\end{equation}
		 That case is illustrated in Figure~\ref{fig:many-r-short-path-between}.
		 We consider two sub-cases:
		\begin{enumerate}
			\item If 
			\begin{equation}
			\label{eq:yzsharp-far-different}
						\dist_{T^0}(\tilde{y}^0_i, \tilde{z}^0_i)  \neq \dist_{T^\#}(\tilde{y}^\#_i,\tilde{z}^\#_i),
			\end{equation}
			we proceed as in the ``far'' case above.
			The argument then follows in the same way
			with~\eqref{eq:yzsharp-far-different} playing the role of~\eqref{eq:yzsharp-far} in~\eqref{eq:different-distance}.
			
			\item If instead
			\begin{equation}
			\label{eq:yzsharp-far-same}
			\dist_{T^0}(\tilde{y}^0_i, \tilde{z}^0_i)  = \dist_{T^\#}(\tilde{y}^\#_i,\tilde{z}^\#_i),
			\end{equation}
			we choose the test pairs below $v^0$ and $w^0$ respectively, as shown in Figure~\ref{fig:many-r-short-path-between}.
			Formally, let $y^0_i$ be the closest $\g$-vertex below $v^0$
			resulting in a consistent rooting.
			Let $z^0_i$ be defined similarly. Such vertices
			exist within graph distance at most $2\ell$ of $v^0$ and $w^0$. (Note that the latter are not in general $\g$-vertices themselves which,
			in addition to the $(\ell,1)$-density assumption, explains the $2 \ell$.)
			Let $y^\#_i$ and $z^\#_i$ be the corresponding vertices
			in $T^\#$.
			\begin{figure}
				\centering
				\includegraphics[width = 0.9\textwidth]{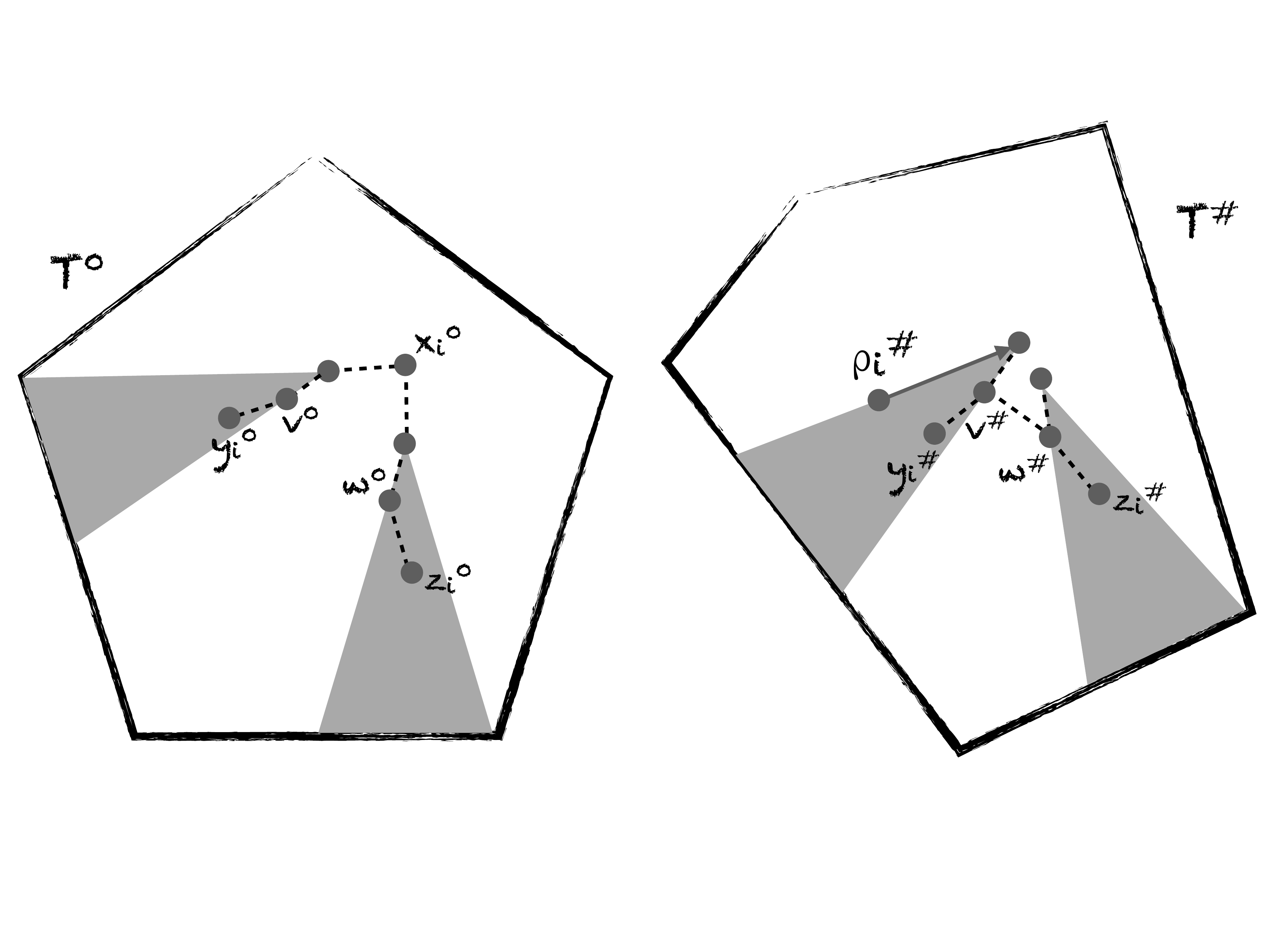}
				\caption{\srevision{
						Construction of the test in the non-co-hanging sub-case of the many-$\r$ case when
						$\tilde{y}^\#_i,\tilde{z}^\#_i$ are ``close.'' The root of 
						the cluster in $T^\#$ is denoted
						by $\rho_i^\#$. }}\label{fig:many-r-short-path-between}
			\end{figure}
			Then, the path connecting $y_i^\#$ and
			$z_i^\#$ in $T^\#$ goes through $v^\#$ and
			$w^\#$, and we have
			\begin{eqnarray}
			\dist_{T^\#}(y^\#_i,z^\#_i)
			&=& \dist_{T^\#}(y^\#_i,v^\#)
			+ \dist_{T^\#}(v^\#,w^\#)
			+ \dist_{T^\#}(w^\#,z^\#_i)\nonumber\\
			&<& \dist_{T^0}(y^0_i,\tilde{y}^0_i)
			+ \dist_{T^0}(\tilde{y}^0_i,\tilde{z}^0_i)
			+ \dist_{T^0}(\tilde{z}^0_i,z^0_i)\nonumber\\
			&=& \dist_{T^0}(y^0_i,z^0_i) \label{eq:different-distance2},
			\end{eqnarray}
			where the inequality holds term by term.
			For the first term, we note that the path
			from $y_i^0$ to $\tilde{y}_i^0$ in $T^0$
			goes through $v^0$,
			and similarly for the third term. For the second term, we use~\eqref{eq:yzsharp-far-same}
			and the fact that the path connecting $v$ and $w$ is a sub-path of the path connecting
			$\tilde{y}_i^0$ and $\tilde{z}_i^0$.
			Hence, we have established that
			$\dist_{T^0}(y^0_i, z^0_i)  \neq \dist_{T^\#}(y^\#_i,z^\#_i)$.
			Moreover
			note that the subtrees rooted at $y^\#_i$ and $z^\#_i$
			are co-hanging in $T^\#$. The resulting test subtrees are also matching and $(\ell,1)$-dense by construction. It remains
			to check the proximality condition. From the choice of $y^\#_i$, $z^\#_i$,
			\begin{eqnarray*}
				\gdist_{T^0}(y^0_i,z^0_i)
				&=& \gdist_{T^0}(y^0_i,v^0)
				+ \gdist_{T^0}(v^0,\tilde{y}^0_i)
				+ \gdist_{T^0}(\tilde{y}^0_i,\tilde{z}^0_i)
				+ \gdist_{T^0}(\tilde{z}^0_i,w^0)
				+ \gdist_{T^0}(w^0,z^0_i)\\
				&=& 
				\gdist_{T^0}(\tilde{y}^0_i,\tilde{z}^0_i)
				+ \gdist_{T^0}(y^0_i,v^0)
				+ \gdist_{T^0}(w^0,z^0_i)
				+ [\gdist_{T^0}(v^0,\tilde{y}^0_i)
				+ \gdist_{T^0}(\tilde{z}^0_i,w^0)]\\
				&\leq& 2\ell + 2\ell + 2\ell + 2\invquantum g\ell\\
				&=& (6 + 2\invquantum g) \ell\\
				&\leq& \Gamma,
			\end{eqnarray*}
			where the equality on the second line is a rearrangement
			of terms and the inequality on the third line
			holds term by term: the first term
			follows from the fact that $\tilde{y}_i^0$ and
			$\tilde{z}_i^0$ are both $\g$-children of $\tilde{x}_i^0$; the second and third terms follow from the choice 
			of $y_i^0$ and $z^0_i$ as described above; 
			and the term in square brackets is an application of~\eqref{eq:non-co-hanging-close} 
			and~\eqref{eq:yzsharp-far-same} converted into
			graph distance through a multiplication by $\invquantum$, together with the observation that the paths from $v^0$ to $\tilde{y}_i^0$ and from $\tilde{z}_i^0$ to $w^0$ match the paths from $v^\#$ to $\tilde{y}_i^\#$ and from $\tilde{z}_i^\#$
			to $w^\#$, which are themselves sub-paths
			of the path from  $\tilde{y}_i^\#$ to $\tilde{z}_i^\#$.
			Recall that $\Gamma$ is defined in~\eqref{eq:def-gamma}.
			Hence the pairs $(y^0_i, z^0_i)$ and $(y^\#_i, z^\#_i)$
			satisfy the cluster and pair requirements of the battery.
		\end{enumerate}
	\end{enumerate}
	
That concludes the proof.
\end{proof}

\paragraph{Sparsification in $T^\#$.}
It remains to satisfy the global requirements of the
battery. By the construction in Claim~\ref{claim:cohanging-case1} the test subtrees
are non-intersecting in both $T^0$ and $T^\#$.
However we
must also ensure that proximal/semi-proximal connecting paths
and non-proximal hats do not intersect with each other
or with test subtrees from other test panels.
By construction, this
is automatically satisfied in $T^0$ where all test pairs
are proximal.
To satisfy this requirement in $T^\#$,
we make the collection of test pairs ``sparser''
by rejecting an appropriate fraction of them.
\begin{claim}[Sparsification in $T^\#$]
	\label{claim:sparsification-case1}
	Let $\hcal' = \{(y^0_i,z^0_i); (y^\#_i,z^\#_i)\}_{i=1}^{I'}$
	be the test panels
	constructed in Claim~\ref{claim:cohanging-case1}.
	We can find a subset $\hcal \subseteq \hcal'$
	of size
	$$
	|\hcal| = I \geq \frac{1}{1 + 2^{2\gamma_t + 2}} I' \geq \frac{\Delta}{2 C_\ocal (1 + 2^{2\gamma_t + 2})}
	$$
	such that the test panels in $\hcal$
	satisfy all global requirements of a battery.
\end{claim}
\begin{proof}
	We sparsify the set $\hcal'$ of test pairs as follows.
	Let $\{(Y^0_i,Z^0_i); (Y^\#_i,Z^\#_i)\}_{i=1}^{I'}$
	be the test subtrees corresponding to $\hcal'$.
	Start with test panel $((y^0_1,z^0_1); (y^\#_1,z^\#_1))$.
	Remove
	from $\hcal'$ all test panels $i \neq 1$ such that
	\begin{equation}\label{eq:definition-hcal}
	\min\{\gdist_{T^\#}(v,w)\ :\  v \in \{y^\#_1,z^\#_1\},
	w \in \vcal(Y^\#_i)\cup\vcal(Z^\#_i)\}
	\leq 2\gamma_t.
	\end{equation}
	Because there are at most $2\cdot 2^{2\gamma_t + 1}$
	vertices $w$ in $T^\#$ satisfying the above condition and that
	the test subtrees are non-overlapping in $T^\#$
	(so that any such vertex belongs to at most one
	test subtree), we remove
	at most $2^{2\gamma_t + 2}$ test panels from
	$\hcal'$.
	
	Let $i$ be the smallest index remaining in $\hcal'$.
	Proceed as above and then repeat until
	all indices in $\hcal'$ have been selected or rejected.
	
	At the end of the procedure, there are at least
	$$
	\frac{1}{1 + 2^{2\gamma_t + 2}} I'
	$$
	test panels remaining, the set of which we denote by $\hcal$.
	Recall that $\gamma_t \geq \Gamma$. Hence by~\eqref{eq:definition-hcal},
	in $\hcal$, the connecting paths of proximal/semi-proximal
	pairs
	and the hats of non-proximal pairs
	cannot intersect with each other or with any of the
	test subtree rooted at test vertices in $\hcal$.
\end{proof}

\paragraph{Summary of many-$\r$ case.}
We have proved the following in the many-$\r$ case.
Recall that $\wp = 1$, $\ell = \ell(g,\wp)$ is chosen as
in Proposition~\ref{prop:distinguishing},
$\Gamma = (6 + 4\invquantum g) \ell$,
and
$\gamma_t \geq \Gamma$, a multiple of $\ell$, and $C$
are chosen as
in Proposition~\ref{prop:distinguishing}.

\begin{proposition}[Battery in the many-$\r$ case]
	\label{prop:battery-many-r}
	In the many-$\r$ case, we can build a
	$(\ell,\wp,\Gamma,\gamma_t,I)$-battery
	$$
	\{((y^0_i,z^0_i);(Y^0_i,Z^0_i))\}_{i=1}^I
	\text{ (in $T^0$) and }
	\{((y^\#_i,z^\#_i);(Y^\#_i,Z^\#_i))\}_{i=1}^I
	\text{ (in $T^\#$)}
	$$
	with
	$$
	I \geq  \frac{\Delta}{2 C_\ocal (1 + 2^{2\gamma_t + 2})}.
	$$
\end{proposition}
\begin{proof}
	The result follows from
	Claims~\ref{claim:recoloring-case1},~\ref{claim:cohanging-case1},
	and~\ref{claim:sparsification-case1}.
\end{proof}

\subsection{Constructing a battery of tests: Large overlap case}
\label{section:largeOverlap}

We now construct a battery of tests in the large overlap
case. By assumption we have,
$$
|\ocal^\#| \geq \frac{1}{10} \frac{\blowup(T^0,T^\#)}{C_\ocal}.
$$
Moreover, by the proof of Claim~\ref{claim:overlapRelationship}, a significant
fraction of the vertices in the overlap are in fact
shallow, that is, they are close to the
boundary of the overlap. To build a battery in this case,
we show that a test pair can be found near each
shallow vertex.
As in the many-$\r$ case, we 
need to deal with a number of issues,
including the overlap of
$\g$-clusters, the possibility of non-co-hanging pairs, and the
proximity of the matching subtrees.

In this section, $T^0$ and $T^\#$ are fixed. To simplify
notation, we let $\Delta = \blowup\left(T^0,T^\#\right)$.
Recall that $T^0$ is rooted.
We also root $T^\#$ arbitrarily. 
Fix $\wp = 5$. Choose $\ell = \ell(g,\wp)$ as
in Proposition~\ref{prop:distinguishing}.
Then take
$$
\Gamma = 6 g \invquantum \log_2 \left(\frac{8}{1 - 1/\sqrt{2}}\right) + 2\ell g \invquantum + 4,
$$
and set
$\gamma_t \geq \Gamma$, a multiple of $\ell$, and $C$ as
in Proposition~\ref{prop:distinguishing}.

\paragraph{Test pairs near the boundary of the overlap.} Let $v^\#$ in $T^\#$ be in
the overlap.
Intuitively, vertex $v^\#$ can be used to construct a test pair
for the following two reasons:
\begin{itemize}
	\item It corresponds to (at least) two
	vertices $v^0_i$, $v^0_j$ in $T^0$ from distinct clusters. The
	evolutionary distance between these vertices differs
	in $T^0$, where it is $> 0$, and in $T^\#$, where
	it is $0$.
	
	\item The vertex $v^\#$ is in the matching $\g$-clusters
	of those
	including $v^0_i$ and $v^0_j$.
	Hence its sequence can be
	reconstructed using the same estimator on
	$T^0$ and $T^\#$.
\end{itemize}
However, to avoid unwanted correlations between
the ancestral reconstructions on $T^\#$,
one must be careful to construct
appropriate co-hanging test subtrees that further
satisfy all requirements of a battery.
We proceed instead by identifying pairs of
edges in $T^0$ that overlap close to its boundary.
\begin{enumerate}
	\item \srevision{ {\em Bounding the number of overlap-shallow edges in $T^0$.} Recall the definition of an overlap-shallow
	vertex from Definition~\ref{def:overlap-shallow} (in
	Claim~\ref{claim:overlapRelationship}).
	We further say that an edge $e = (x,w)$ in $\ocal^0$
	is {\em overlap-shallow with parameter $\beta$}
	if {\em both} $x$ and $w$ are overlap-shallow 
	with parameter $\beta$. We call {\em deep} those vertices
	and edges that are not overlap-shallow.
	Proceeding as in~\eqref{eq:factor-of-two-pre},
	we see that at most a fraction
	$1 - 1/\beta$ of vertices in the overlap 
	are deep. Each such vertex prevents at most $3$ edges
	in $\ocal^0$ from being overlap-shallow. From the fact
	there are at most twice as many vertices in the overlap
	as there are edges, we get that the number of overlap-shallow edges in $T^0$ is at least
	\begin{equation}
	\label{eq:shallow-edges}
	\left[1 - 6 \left(1- \frac{1}{\beta}\right)\right]
	|\ocal^0|.
	\end{equation}
	We will later choose $\beta > 1$ close enough to
	$1$ that the above fraction is positive. }
	
	\item \srevision{ {\em Bounding the number of intersecting pairs of shallow edges.}
	For reasons that will be explained below, our test construction is based on finding pairs
	of shallow edges that {\em intersect}. Formally, we say
	that $e_1, e_2 \in \ocal^0$ intersect if the corresponding
	paths in $T^\#$ share an edge. We say that
	an edge $e_1 \in \ocal^0$ is {\em useful}
	if it is overlap-shallow and if it intersects with at least
	one other overlap-shallow edge $e_2 \neq e_1$. Note that here 
	$e_1$ and $e_2$ must belong to distinct 
	maximal $\g$-clusters (otherwise the corresponding cluster would not be matching in $T^\#$). Let $e_0 \in \ocal^0$
	be deep. Recall that $e_0$ corresponds to a path
	of length at most $g\invquantum$ in $T^\#$. Let $e_1, e_2$
	are overlap-shallow edges intersecting with $e_0$
	but not with each other. Then the paths corresponding
	to each of $e_1$ and $e_2$ in $T^\#$ must intersect with different
	edges on the path corresponding to $e_0$. Put differently,
	any deep edge $e_0$ can prevent at most $g\invquantum$
	shallow edges from intersecting with any other shallow
	edge. Combining this with~\eqref{eq:shallow-edges},
	we get that the number of useful edges in $T^0$ is
	at least
	\begin{equation}
	\label{eq:useful-edges}
	\left[1 - 6 g \invquantum \left(1- \frac{1}{\beta}\right)\right]
	|\ocal^0|.
	\end{equation}
}
		
	\item \srevision{ {\it Existence of four close witnesses.} Let $e_i = (w_i^+,w_i^-)$ be a useful edge in $\gcal_i$. 
	Let $e_j = (w_j^+,w_j^-)$ in $\gcal_j$ be an overlap-shallow edge intersecting with $e_i$
	and let $e^\#$ be an edge in $T^\#$ lying on the paths
	corresponding to both $e_i$ and $e_j$. 
	Assume that, for $\iota = i,j$, $w_\iota^+$ is the parent
	of $w_\iota^-$. By definition of a useful edge, both
	$w_i^+$ and $w_j^+$ are overlap-shallow.
	In the proof of Claim~\ref{claim:overlapRelationship},
	we argued that~\eqref{eq:shallow-def} implies the existence of a close witness,
	that is,
	a leaf or vertex not in the overlap
	at a constant graph distance.
	By restricting the sum in~\eqref{eq:shallow-def}
	to those vertices that are on only one side below $w_\iota^+$ (that is, below
	one of its immediate children), we can find in fact 
	two distinct witnesses $\tilde{y}^0_\iota$ and
	$\tilde{z}^0_\iota$ at constant graph distance
	$
	C' = C'(g, \invquantum)
	$
	from $w_\iota^+$, one on each side.
	The four witnesses 
	$\tilde{y}^0_i$,
	$\tilde{z}^0_i$,
	$\tilde{y}^0_j$,
	$\tilde{z}^0_j$
	jointly
	satisfy the following properties:
	\begin{itemize}
		\item They are not in the overlap.
		\item For $\iota = i,j$, $\tilde{y}^0_\iota$ and $\tilde{z}^0_\iota$ are in $\gcal_\iota$.
		\item For $x = y,z$, the vertices in $T^\#$
		corresponding to $\tilde{x}^0_i$ and $\tilde{x}^0_j$ are on the same side of $e^\#$, that is, the path between them does not cross $e^\#$.
	\end{itemize}
	Let $\tilde{y}^\#_i$, $\tilde{y}^\#_j$, $\tilde{z}^\#_i$, and $\tilde{z}^\#_j$ be the corresponding
	vertices in $T^\#$. }
	
	\item \srevision{ {\it Key observation: quartet topologies differ on $T^0$ and $T^\#$.} We construct
	a test pair nearby $e_i$ as follows.
	The key observation is the following: by construction, the topology of $T^0$ restricted to the witnesses
	$\{\tilde{y}^0_i, \tilde{y}^0_j, \tilde{z}^0_i, \tilde{z}^0_j\}$ is
	$\tilde{y}^0_i \tilde{z}^0_i | \tilde{y}^0_j \tilde{z}^0_j$
	while 
	the
	topology of $T^\#$ restricted to the corresponding
	vertices
	$\{\tilde{y}^\#_i, \tilde{y}^\#_j, \tilde{z}^\#_i, \tilde{z}^\#_j\}$
	is
	$\tilde{y}^\#_i \tilde{y}^\#_j | \tilde{z}^\#_i \tilde{z}^\#_j$. Indeed, on $T^0$, the pairs
	$\{\tilde{y}^0_i, \tilde{z}^0_i\}$, $\{\tilde{y}^0_j, \tilde{z}^0_j\}$ belong to distinct co-hanging clusters
	and their most recent common ancestors are therefore
	separated by the path joining the roots of those clusters.
	On $T^\#$, on the other hand,  by construction $e^\#$
	separates $\{\tilde{y}^\#_i, \tilde{y}^\#_j\}$ from $\{\tilde{z}^\#_i, \tilde{z}^\#_j\}$.
	\srevision{See Figure~\ref{fig:overlap-quartet} for an illustration and 
	refer to Definition~\ref{def:restricted}
	for quartet topology notation.}
	\begin{figure}
		\centering
		\includegraphics[width = 0.9\textwidth]{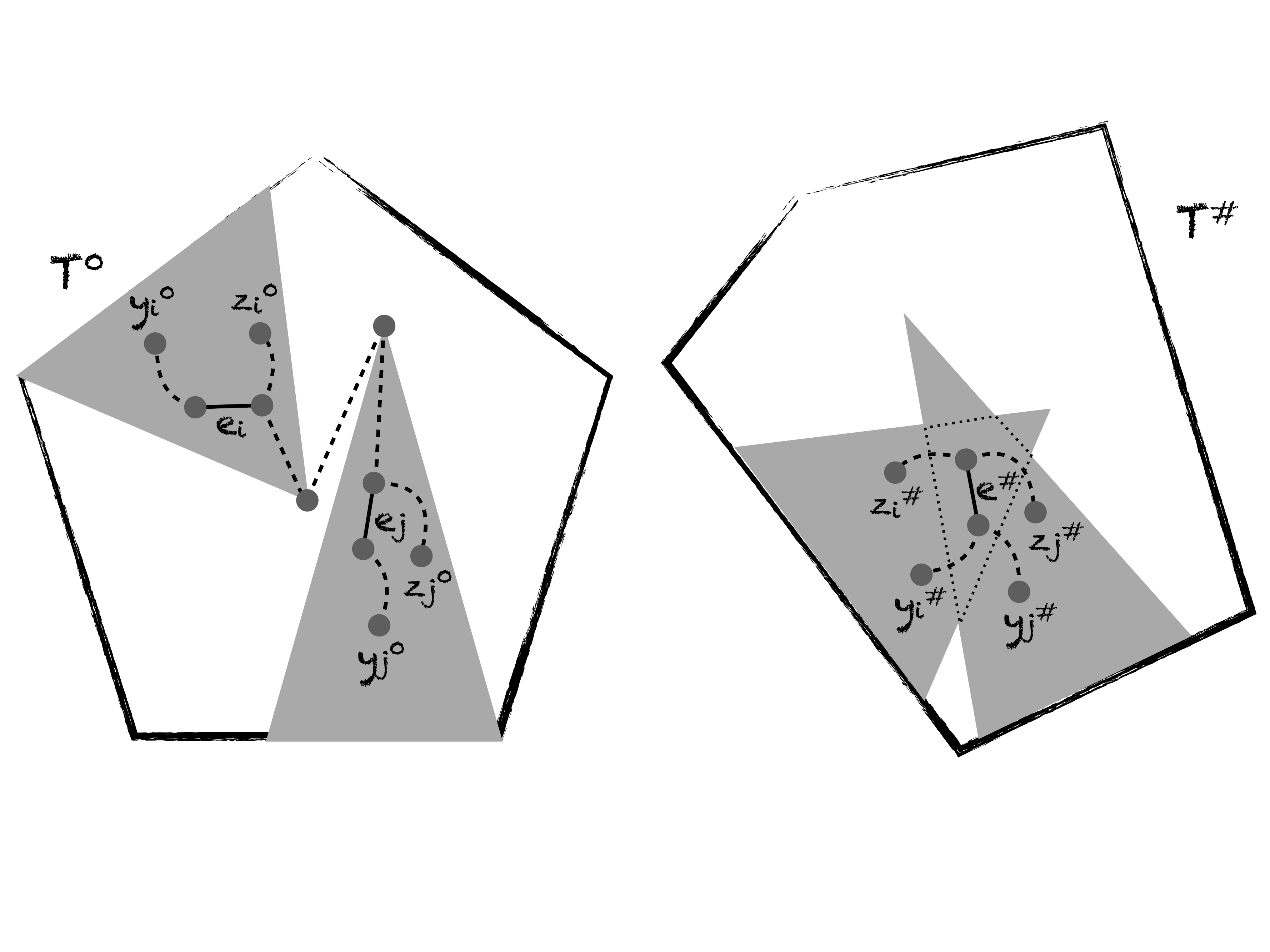}
		\caption{\srevision{
				Construction of the test in the large overlap case. The region inside the dotted line
				is part of the overlap in $T^\#$. }}\label{fig:overlap-quartet}
	\end{figure}
}

	\item \srevision{ {\it Existence of witnesses at different evolutionary
		distances on $T^0$ and $T^\#$.}
	The reason the above observation is significant
	is that it allows us to find a pair among the
	witnesses whose evolutionary distance differs
	on $T^0$ and $T^\#$, as we show next.
	Note that
	\begin{equation}\label{eq:overlapQuartet}
	\dist_{T^0}(\tilde{y}^0_\iota, \tilde{z}^0_\iota)
	= \dist_{T^\#}(\tilde{y}^\#_\iota, \tilde{z}^\#_\iota)
	\end{equation}
	for $\iota = i, j$ by definition
	of the matching subtree of $\gcal_\iota$.
	Moreover, by the four-point condition~\eqref{eq:four-point}
	in the non-degenerate case,
	$$
	\dist_{T^0}(\tilde{y}^0_i, \tilde{y}^0_j)
	+ \dist_{T^0}(\tilde{z}^0_i, \tilde{z}^0_j)
	> \dist_{T^0}(\tilde{y}^0_i, \tilde{z}^0_i)
	+ \dist_{T^0}(\tilde{y}^0_j, \tilde{z}^0_j),
	$$
	and
	$$
	\dist_{T^\#}(\tilde{y}^\#_i, \tilde{z}^\#_i)
	+ \dist_{T^\#}(\tilde{y}^\#_j, \tilde{z}^\#_j)
	> \dist_{T^\#}(\tilde{y}^\#_i, \tilde{y}^\#_j)
	+ \dist_{T^\#}(\tilde{z}^\#_i, \tilde{z}^\#_j),
	$$
	which, with~\eqref{eq:overlapQuartet}, implies
	$$
	\dist_{T^0}(\tilde{y}^0_i, \tilde{y}^0_j)
	+ \dist_{T^0}(\tilde{z}^0_i, \tilde{z}^0_j)
	>  \dist_{T^\#}(\tilde{y}^\#_i, \tilde{y}^\#_j)
	+ \dist_{T^\#}(\tilde{z}^\#_i, \tilde{z}^\#_j).
	$$
	Hence one of the following must hold
	$$
	\dist_{T^0}(\tilde{y}^0_i, \tilde{y}^0_j)
	> \dist_{T^\#}(\tilde{y}^\#_i, \tilde{y}^\#_j)
	\quad \text{or}
	\quad
	\dist_{T^0}(\tilde{z}^0_i, \tilde{z}^0_j)
	> \dist_{T^\#}(\tilde{z}^\#_i, \tilde{z}^\#_j).
	$$
	Without loss of generality, assume that
	$\dist_{T^0}(\tilde{y}^0_i, \tilde{y}^0_j)
	> \dist_{T^\#}(\tilde{y}^\#_i, \tilde{y}^\#_j).
	$
}
	
	\item \srevision{ {\it Distance to witnesses.}
	We will also need to bound $\dist_{T^\#}(\tilde{y}^\#_i, \tilde{y}^\#_j)$.
	It suffices to bound $C'$ above.
	\begin{claim}[Distance to witnesses]
		\label{claim:distance-to-witnesses}
		We have
		$$
		C' \leq 3 \log_2 \left(\frac{4\beta}{1 - 1/\sqrt{2}}\right),
		$$
		for $\ell$ large enoug.
	\end{claim}
	\begin{proof}
		We use the notation of Claim~\ref{claim:overlapRelationship}.
		Assume all vertices in $\vcal^{\gcal_i}_{w_i^-}$ within graph distance $C' - 1$ are in
		$\wcal^{\gcal_i}_{w_i^-}$. Because
		$\gcal_i$ is $(\ell,1)$-dense, 
		we have that within graph distance $C'$
		of $w_i^+$ there is at least
		$$
		\frac{1}{2} (2^\ell - 1)^{\frac{C'}{\ell}}
		$$
		vertices in $\wcal^{\gcal_i}_{w_i^+}$ below $w_i^-$, where we counted
		only the furthest vertices within this ball.
		Hence the
		sum in~\eqref{eq:shallow-def} restricted to
		vertices below $w_i^-$ satisfies
		$$
		\sum_{y \in \wcal^{\gcal_i}_{w_i^-}}
		2^{-\frac{\gdist_{T^0}(x,y)}{2}}
		\geq \frac{1}{2} (2^\ell - 1)^{\frac{C'}{\ell}}
		2^{-\frac{C'}{2}}
		\geq \frac{1}{2} 2^{\frac{C'}{3}}
		> \frac{\beta}{1 - 1/\sqrt{2}},
		$$
		for $\ell$ large enough,
		if
		$$
		C' > 3 \log_2 \left(\frac{4\beta}{1 - 1/\sqrt{2}}\right).
		$$
	\end{proof}
}

\end{enumerate}

\srevision{
We repeat the procedure above for each useful
edge and get a collection of pre-test panels. 
To satisfy the requirements of the battery
we then proceed, similarly to the many-$\r$ case,
by re-rooting and sparsification.
We describe these steps next.
}

\paragraph{Co-hanging pairs.}
Note that the roots of $T^0$ and $T^\#$
may not be consistent, in the sense that the $\g$-clusters
and their matching subtrees may not be rooted at
corresponding vertices.
However we can make it so that
the test subtrees are rooted consistently and ensure
that the test subtrees are co-hanging.

For every pre-test panel
constructed above,
using the same notation,
we proceed as follows. If $\tilde{y}^0_i$
is on the path between the root of $\gcal_i$ in
$T^0$ and the (possibly extra) vertex corresponding
to the root of the matching subtree
in $T^\#$, we move $\tilde{y}^\#_i$ over to one of its
immediate children such that the corresponding vertex
$\tilde{y}^0_i$ is not on this path---\emph{unless}
$\tilde{y}^0_j$ is a descendant
of that vertex. In that case, we instead move
$\tilde{y}^\#_i$ over to the child of its other immediate child such that the corresponding vertex
$\tilde{y}^0_i$ is not on the path above. The reason we
need these two cases is that we seek to preserve
the inequality
$$
\dist_{T^0}(\tilde{y}^0_i, \tilde{y}^0_j)
> \dist_{T^\#}(\tilde{y}^\#_i, \tilde{y}^\#_j).
$$
In both cases, the two sides of the inequality
increase by the same amount, at most $2g$.
We do the same on $\gcal_j$.

At this point, 1)
the $\g$-cluster of $T^0$ rooted at
$\tilde{y}^0_i$ and the matching
subtree rooted at $\tilde{y}^\#_i$ are
rooted consistently (and similarly for
$\tilde{y}^0_j$ and $\tilde{y}^\#_i$)
and 2) the $\g$-clusters rooted at
$\tilde{y}^0_i$ and $\tilde{y}^0_j$ are
co-hanging (and similarly for the matching
subtrees in $T^\#$).

Let $y^0_i$ be a closest $\g$-vertex below
$\tilde{y}^0_i$ on $\gcal_i$ and similarly
for $y^0_j$.
Let $y^\#_i$ and $y^\#_j$
be the corresponding vertices in $T^\#$.
Then the test subtrees (that is the $\g$-clusters)
$\bar{Y}^0_i$, $\bar{Y}^0_j$
$\bar{Y}^\#_i$ and $\bar{Y}^\#_j$ rooted respectively at
$y^0_i$, $y^0_j$
$y^\#_i$ and $y^\#_j$
are such that $(\bar{Y}^0_i, \bar{Y}^0_j)$
and $(\bar{Y}^\#_i, \bar{Y}^\#_j)$ are
co-hanging. In particular, they are non-intersecting by construction.
Indeed, $(\bar{Y}^0_i, \bar{Y}^0_j)$ belong to different $\g$-clusters
in $T^0$ and $(\bar{Y}^\#_i, \bar{Y}^\#_j)$ are on different
sides below $v^\#$ in $T^\#$.

\srevision{
Moreover, by Claim~\ref{claim:distance-to-witnesses},
we have
\begin{equation}
\label{eq:gamma-requirement}
\gdist_{T^\#}(y^\#_i, y^\#_j) \leq
6 g\invquantum \log_2 \left(\frac{4\beta}{1 - 1/\sqrt{2}}\right) + 2\ell g \invquantum + 4
\leq \Gamma,
\end{equation}
where the first term corresponds to the distance
to the closest witnesses, the second term corresponds
to the distance to the closest $\g$-vertex,
and the third term corresponds to the re-rooting 
operation above.
We also used that each edge in $T^0$ corresponds
to at most $g\invquantum$ edges in the matching cluster.
Hence the test pair $(y^\#_i, y^\#_j)$
is proximal.}
We also have
$$
\dist_{T^0}(y^0_i, y^0_j)
> \dist_{T^\#}(y^\#_i, y^\#_j),
$$
because
$
\dist_{T^0}(\tilde{y}^0_i, \tilde{y}^0_j)
> \dist_{T^\#}(\tilde{y}^\#_i, \tilde{y}^\#_j),
$
$
\dist_{T^0}(\tilde{y}^0_\iota, y^0_\iota)
= \dist_{T^\#}(\tilde{y}^\#_\iota, y^\#_\iota),
$
for $\iota=i,j$, and
$y^*_\iota$ is below $\tilde{y}^*_\iota$
for $*=0,\#$ and $\iota = i,j$.

\paragraph{Sparsification.}
It remains to satisfy the global requirements of the
battery.
Unlike the many-$\r$ case, we need to make the
collection of test pairs sparser in
{\em both} $T^\#$ and $T^0$.
Indeed, although
there is no overlap between the $\g$-clusters in $T^0$,
in constructing the tests we may have used
the {\em same} maximal $\g$-cluster repeatedly.
Hence there is in fact no guarantee that the test subtrees
are not overlapping in $T^0$.
In $T^\#$, test subtrees may also be overlapping,
whether or not they belong to the same $\g$-cluster.
Moreover,
although the test subtrees are co-hanging and
proximal in $T^\#$,
we must ensure that the connecting paths
do not intersect with
other test subtrees or their connecting paths.

Let $\{(y^0_i,y^0_j); (y^\#_i,y^\#_j)\}_{(i,j)\in \hcal'}$
be the test panels
constructed above.
\srevision{
By~\eqref{eq:largeOverlap},~\eqref{eq:useful-edges}, and
Claim~\ref{claim:overlapRelationship}, we have
\begin{eqnarray*}
|\hcal'|
&\geq&
\left[1 - 6 g \invquantum \left(1- \frac{1}{\beta}\right)\right]
|\ocal^0|\\
&\geq& 
\left[1 - 6 g \invquantum \left(1- \frac{1}{\beta}\right)\right] \cdot \frac{1}{g\invquantum} 
|\ocal^\#|\\
&\geq& 
\left[1 - 6 g \invquantum \left(1- \frac{1}{\beta}\right)\right]\cdot  \frac{1}{g\invquantum} \cdot
\frac{1}{10}
\frac{\Delta}{C_\ocal}.
\end{eqnarray*}
We choose 
$$
\beta = \frac{12g \invquantum}{12 g \invquantum - 1},
$$
so that the expression in square brackets above is $1/2$
and we have
$$
\left|\hcal'\right| \geq \frac{\Delta}{20 g \invquantum C_\ocal}.
$$
We note that because $g\invquantum \geq 1$, we have $\beta \leq 2$.}
Let $\{(\bar{Y}^0_i,\bar{Y}^0_j);
(\bar{Y}^\#_i,\bar{Y}^\#_j)\}_{(i,j)\in \hcal'}$
be the test subtrees corresponding to $\hcal'$.
\begin{claim}[Sparsification]
	\label{claim:sparsification-case2}
	Let
	$$
	C_w = 3 g \invquantum \log_2 \left(\frac{8}{1 - 1/\sqrt{2}}\right) + \ell g\invquantum + 2.
	$$
	There is a subset $\hcal \subseteq \hcal'$
	of size
	$$
	|\hcal| \geq \frac{1}{1 + 2^{6\gamma_t + C_w + 3}g\invquantum} |\hcal'|
	\geq
	\frac{\Delta}{20 g \invquantum C_\ocal (1 + 2^{6\gamma_t + C_w + 3}g\invquantum)}
	$$
	and $(\ell,5)$-dense modified test subtrees
	$\{(Y^0_i,Y^0_j);
	(Y^\#_i,Y^\#_j)\}_{(i,j)\in \hcal}$
	such that the test panels in $\hcal$
	satisfy the requirements of a battery.
\end{claim}
\begin{proof}
	We proceed in two phases. First we choose
	a subset of test panels such that the test vertices
	in different panels
	are far away from each other. Then we cleave subtrees
	of the $\g$-clusters rooted at the test vertices
	to ensure that proximal/semi-proximal
	connecting paths and non-proximal hats do not
	intersect with test subtrees.
	
	Start with an arbitrary
	test panel $((y^0_i,y^0_j); (y^\#_i,y^\#_j))$
	in $\hcal'$.
	Remove
	from $\hcal'$ all test panels $(i',j')$ such that
	\begin{equation}\label{eq:definition-hcal1-case2}
	\min\{\gdist_{T^\#}(v,w)\ :\  v \in \{y^\#_i,y^\#_j\},
	w \in \{y^\#_{i'}, y^\#_{j'}\}\}
	\leq 6 \gamma_t,
	\end{equation}
	or
	\begin{equation}\label{eq:definition-hcal2-case2}
	\min\{\gdist_{T^0}(v,w)\ :\  v \in \{y^0_i,y^0_j\},
	w \in \{y^0_{i'}, y^0_{j'}\}\}
	\leq 6 \gamma_t.
	\end{equation}
	There are at most $2\cdot 2^{6\gamma_t+1}$ vertices
	within graph distance $2\gamma_t$ of a test pair.
	\srevision{
	Note, however, that some vertices may be used as a test vertex
	multiple times. Nevertheless we claim that each vertex
	can be used at most a constant number of times.
	Indeed, consider a vertex $y^0_{i'}$ in $T^0$
	with corresponding vertex $y^\#_{i'}$ in $T^\#$.
	Recall that each test panel is obtained from an
	overlap-shallow edge
	within graph distance
	$C_w$ in $T^0$ and that each such overlap-shallow edge produces
	at most $g\invquantum$ test panels. Hence
	$y^0_{i'}$ can arise in this way at most
	$2^{C_w+1}g\invquantum$ times.
	Therefore
	we remove at most $2^{6\gamma_t + C_w + 3}g\invquantum$ test panels.}
	
	Pick a remaining test pair in $\hcal'$.
	Proceed as above and then repeat until
	all pairs in $\hcal'$ have been picked or removed.
	At the end of the procedure, there are at least
	$$
	\frac{1}{1 + 2^{6\gamma_t + C_w + 3}g\invquantum} |\hcal'|
	$$
	test panels remaining, the set of which
	we denote by $\hcal$.
	Recalling that $\gamma_t \geq \Gamma$, in $\hcal$,
	the connecting paths of proximal/semi-proximal
	pairs and the hats of non-proximal pairs
	cannot intersect with each other
	by~\eqref{eq:definition-hcal1-case2}
	and~\eqref{eq:definition-hcal2-case2}
	as it would imply the existence of test vertices
	in different panels at graph distance less than
	$2\gamma_t \leq 6\gamma_t$.
	
	For each $(i,j) \in \hcal$, it remains to define
	the corresponding test subtrees $((Y^0_i,Y^0_j);
	(Y^\#_i,Y^\#_j))$.
	Let $((\bar{Y}^0_i,\bar{Y}^0_j);
	(\bar{Y}^\#_i,\bar{Y}^\#_j))$
	be as above and note that, since we may
	have re-used the same $\g$-clusters multiple
	times, these subtrees
	may not satisfy the global requirements of a battery
	as they may intersect with
	each other or with connecting paths
	and hats.
	We modify $\bar{Y}^0_i$ as follows,
	and proceed similarly for $\bar{Y}^0_j$.
	For each $(i',j') \in \hcal$ not equal
	to $(i,j)$
	and each subtree
	$Z \in \{\bar{Y}^0_{i'}$, $\bar{Y}^0_{j'}\}$,
	if $Z$ has its root {\em below} the root of $\bar{Y}^0_i$,
	remove from $\bar{Y}^0_i$ all those nodes in
	$Z$ as well as all descendants of the vertices
	on the upward path of length $2\gamma_t$ starting at the root
	of $Z$. Note that the latter path cannot reach $y^0_i$
	because both $y^0_{i'}$ and $y^0_{j'}$ are at
	graph distance at least $6\gamma_t$
	from $y^0_i$
	from the construction of $\hcal$.
	We let $Y^0_i$ be
	the remaining subtree in $T^0$
	and $Y^\#_i$, its matching subtree
	in $T^\#$.	
	We claim that the resulting restricted subtrees
	$(Y^0_i,Y^0_j)$
	are $(\ell,3)$-dense. Note first that
	the subtrees in
	$(\bar{Y}^0_i,\bar{Y}^0_j)_{(i,j)\in\hcal}$
	are $(\ell,1)$-dense as they were obtained
	from the procedure in Section~\ref{sec:finding-matching}. Moreover, because 1) the roots
	of the removed subtrees are
	at graph distance at most
	$2\gamma_t$ from a $\ell$-vertex in $(y^0_i,y^0_j)_{(i,j)\in \hcal}$,
	2) test vertices in different pairs are at graph distance
	at least $6\gamma_t$ from each other, and
	3) $\gamma_t$ is a multiple of $\ell$,
	if we remove a subtree rooted at a
	$\g$-child of a $\g$-vertex in
	$(\bar{Y}^0_i,\bar{Y}^0_j)$
	we cannot remove more than one other subtree rooted
	at another
	$\g$-child of the same $\g$-vertex as that would imply
	the existence of two test vertices {\em in different pairs} at graph distance
	less than
	$$
	2(2\gamma_t) + 2g\ell < 6\gamma_t,
	$$
	in $T^0$, a contradiction.
	
	We then proceed similarly in $T^\#$.
	The resulting restricted subtrees
	$$
	\{(Y^0_i,Y^0_j);
	(Y^\#_i,Y^\#_j)\}_{(i,j)\in\hcal},
	$$
	are then $(\ell,5)$-dense (in fact, $(\ell,4)$-dense
	as there are no non-proximal pairs in $T^\#$).
\end{proof}

\paragraph{Summary of the large overlap case.}
We have proved the following in the large overlap case.
Recall that $\wp = 5$, $\ell = \ell(g,\wp)$ is chosen as
in Proposition~\ref{prop:distinguishing},
$$
\Gamma = 6 g \invquantum \log_2 \left(\frac{8}{1 - 1/\sqrt{2}}\right) + 2\ell g \invquantum + 4,
$$
and
$\gamma_t \geq \Gamma$, a multiple of $\ell$, and $C$
are chosen as
in Proposition~\ref{prop:distinguishing}.

\begin{proposition}[Battery in the large overlap case]
	\label{prop:battery-large-overlap}
	In the large overlap case, we can build a
	$(\ell,\wp,\Gamma,\gamma_t,I)$-battery
	$$
	\{((y^0_i,z^0_i);(Y^0_i,Z^0_i))\}_{i=1}^I
	\text{ (in $T^0$) and }
	\{((y^\#_i,z^\#_i);(Y^\#_i,Z^\#_i))\}_{i=1}^I
	\text{ (in $T^\#$)}
	$$
	with
	$$
	I \geq  \frac{\Delta}{
		20 g\invquantum C_\ocal (1 + 2^{6\gamma_t + C_w + 3}g\invquantum)
	}.
	$$
\end{proposition}
\begin{proof}
	The result follows from
	Claim~\ref{claim:sparsification-case2}.
\end{proof}

\clearpage

\section*{Acknowledgments}

	We thank the anonymous reviewers of a previous
version for helpful comments.

\bibliographystyle{alpha}
\bibliography{my,thesis}

\newpage

\appendix

\section{Preliminary lemmas}

In this section, we collect a few 
useful lemmas.

\subsection{Ancestral reconstruction}
\label{section:appendix-1}

An important part of our construction
involves reconstructing ancestral states.
We will use the following lemma
from~\cite{EvKePeSc:00} which we
typically apply to a rooted subtree.
Let $T = (V,E;\phi;\weight)
\in \phy$ rooted at $\rt$.
Let $e = (x,y) \in E$
and assume that $x$ is closest to $\rt$ (in topological distance).
We define
$\path(\rt,e) = \path(\rt,y)$,
$|e|_\rt = |\path(\rt,e)|$, and
\begin{equation}
\label{eq:def-resistance}
R_\rt(e) = \left(1 - \theta_e^2\right)
\Theta_{\rt,y}^{-2},
\end{equation}
where $\Theta_{\rt,y} = e^{-\dist_T(\rt,y)}$ and
$\theta_e = e^{-\weight_e}$.
\begin{lemma}[Ancestral reconstruction~\cite{EvKePeSc:00}]
\label{lem:ekps}
For any unit flow $\Psi$ from $\rt$ to $[n]$,
\begin{equation}
\label{eq:flow}
\E_T\left|
\P_T[\s_\rt = +1|\s_X]
- \P_T[\s_\rt = -1|\s_X]
\right|
\geq
\frac{1}{1+
\sum_{e \in E} R_\rt(e) \Psi(e)^2},
\end{equation}
where the LHS is the difference between the
probability of correct and incorrect
reconstruction using MLE.
(See~\cite[Equation (14), Lemma 5.1 and Theorem 1.2']{EvKePeSc:00}.)
\end{lemma}

\subsection{Random cluster representation}
\label{section:appendix-2}

We use a convenient percolation-based representation of the
CFN model known as the random cluster model (see e.g.~\cite{Grimmett:06}). Let $T = (V,E;\phi;\weight)
\in \phy$
with corresponding $(\delta_e)_{e\in E}$.
\begin{lemma}[Random cluster representation]
\label{lem:fk}
Run a percolation process on $T$
where edge $e$ is open with probability
$1 - 2 \delta_e$. Then associate
to each open connected component a state according
to the uniform distribution on $\{+1,-1\}$.
The state vector on the vertices
so obtained $(\s_v)_{v\in V}$ has the same distribution
as the corresponding CFN model.
\end{lemma}

\subsection{Concentration inequalities}
\label{section:appendix-3}

Recall the following standard concentration inequality (see e.g.~\cite{MotwaniRaghavan:95}):
\begin{lemma}[Azuma-Hoeffding Inequality]\label{lemma:azuma}
Suppose ${\bf Z}=(Z_1,\ldots,Z_m)$
are independent random variables taking values in a set
$S$, and $h:S^m \to \real$ is any $t$-Lipschitz function:
$|h({\bf z}) - h({\bf z'})|\leq t$ whenever ${\bf z}, {\bf z'} \in S^m$
differ at just one coordinate. Then,
$\forall \zeta > 0$,
\begin{equation*}
\prob\left[|h({\bf Z}) - \expec[h({\bf Z})]| \geq \zeta \right]
\leq 2\exp\left(-\frac{\zeta^2 }{2 t^2 m}\right).
\end{equation*}
\end{lemma}

\end{document}